\pdfoutput = 1
\documentclass[moor,sglanonrev]{informs44}
\RequirePackage{tgtermes}
\RequirePackage{newtxtext}
\RequirePackage{newtxmath}
\RequirePackage{bm}
\RequirePackage{endnotes}
\SingleSpacedXI
\usepackage{algorithm}
\usepackage{algpseudocode}
\usepackage{tikz}
\usepackage[colorlinks=true,breaklinks=true,bookmarks=true,urlcolor=blue,
citecolor=blue,linkcolor=blue,bookmarksopen=false,draft=false]{hyperref}
\usepackage{cleveref}
\usepackage{amsmath, enumitem}
\usepackage{cancel}

\makeatletter
\g@addto@macro\normalsize{%
\setlength{\abovedisplayskip}{6pt}%
\setlength{\belowdisplayskip}{6pt}%
\setlength{\abovedisplayshortskip}{4pt}%
\setlength{\belowdisplayshortskip}{4pt}%
}
\makeatother

\usepackage{comment}
\usepackage{soul,mathtools}
\usepackage{math}
\newcommand{\rr}{\mathbb{R}}

\usepackage[sort&compress]{natbib}
 \bibpunct[, ]{[}{]}{,}{n}{}{,}%
 \def\bibfont{\small}%
\EquationsNumberedThrough    
\TheoremsNumberedThrough     
\ECRepeatTheorems  %
\MANUSCRIPTNO{MOOR-0001-2024.00}

\definecolor{myred}{RGB}{188, 24, 10}

\begin{document}

\RUNAUTHOR{Zhao, Anitescu, and Na}
\RUNTITLE{Continuous-Time Overlapping Schwarz Scheme for LQ Programs}
\TITLE{Overlapping Schwarz Scheme for Linear-Quadratic Programs in Continuous Time}
\ARTICLEAUTHORS{%
\AUTHOR{Hongli Zhao}
\AFF{Department of Statistics, The University of Chicago, \EMAIL{honglizhaobob@uchicago.edu}}
\AUTHOR{Mihai Anitescu}
\AFF{Department of Statistics, The University of Chicago, \EMAIL{anitescu@uchicago.edu}}
\AFF{Mathematics and Computer Science Division, Argonne National Laboratory}
\AUTHOR{Sen Na}
\AFF{School of Industrial and Systems Engineering, Georgia Institute of Technology, \EMAIL{senna@gatech.edu}}
}

\ABSTRACT{
We present an optimize-then-discretize framework for solving linear-quadratic optimal control problems (OCP) governed by time-inhomogeneous ordinary differential equations (ODEs). Our method employs a modified overlapping Schwarz decomposition based on the Pontryagin Minimum Principle, partitioning the temporal domain into overlapping intervals and independently solving Hamiltonian systems in continuous time. We demonstrate that the convergence is ensured by appropriately updating the boundary conditions of the individual Hamiltonian dynamics. The cornerstone of our analysis is to prove that the exponential decay of sensitivity (EDS) exhibited in discrete-time OCPs carries over to the continuous-time setting. Unlike the discretize-then-optimize approach, our method can flexibly incorporate different numerical integration methods for solving the resulting~Hamiltonian two-point boundary-value subproblems, including adaptive-time integrators. A numerical experiment on a linear-quadratic OCP illustrates the practicality of our approach in broad scientific applications.
}

\maketitle

\section{Introduction}\label{sec:Intro}

Optimal control problems (OCPs) constrained by differential equations are fundamental in a wide range of scientific and engineering applications, including fluid dynamics, biomedical engineering, and aerospace design \cite{Ulbrich2009Analytical, Gunzburger2002Perspectives, Troeltzsch2010Optimal}. These problems require determining an optimal control function that minimizes a given cost functional while satisfying the governing differential constraints. The numerical formulation of OCPs traditionally follows one of the two major paradigms: the direct ``discretize-then-optimize'' approach, and the indirect ``optimize-then-discretize'' approach \cite{Herzog2010Algorithms, Antil2018Brief}.

The direct approach first applies numerical discretization techniques, such as finite element \cite{Brenner2008Mathematical}, finite difference \cite{LeVeque2007Finite}, or spectral methods \cite{Shen2011Spectral}, to approximate the governing differential equations.
This transforms the original infinite-dimensional problem in the space of solution functions into a finite-dimensional constrained optimization problem, where the control and state variables are indexed by a discrete set of design variables. 
Then, to solve the resulting nonlinear program, one can employ constrained optimization techniques, including interior-point methods \cite{Waechter2005implementation} and sequential quadratic programming (SQP) methods \cite{Gill2011Sequential, Fletcher2010Sequential}, which enforce constraints on both state and control variables. 
Significant advances in direct methods have focused on improving numerical efficiency and stability. For example, reduced-space Newton-Krylov solvers \cite{Hernandez2018Newton,Hernandez2019Comparative} offer iterative techniques to compute the reduced Karush-Kuhn-Tucker (KKT) matrix-vector products efficiently, while penalty-barrier methods \cite{Neuenhofen2020Integral} introduce regularization terms on the primal-dual problem to improve stability. Despite these improvements, direct methods remain inherently dependent on a fully discretized formulation of the OCP, making them increasingly computationally intensive as the \mbox{discretization}~level~grows.

In contrast, the indirect approach derives first-order optimality conditions in the continuous setting before applying numerical discretization. This formulation yields necessary optimality conditions in the form of a Hamiltonian system \cite{Craig2008Hamiltonian}. A key component of the indirect method is the adjoint equation, which enables efficient gradient computation with respect to the control variables by solving an auxiliary system of equations \cite{Eichmeir2020Adjoint}, thereby eliminating the need for explicit numerical differentiation with respect to the discretized design variables.
To solve the resulting Hamiltonian system, state-of-the-art methods employ numerical techniques that preserve the structure of the continuous problem. 
For example, shooting methods \cite{Betts2010Practical} solve boundary-value problems by iteratively adjusting the control variables to satisfy the necessary conditions for optimality; and matrix-free adjoint techniques \cite{Hager2000Runge} alleviate memory costs by avoiding the explicit storage of big Jacobian matrices, making them effective for numerically solving large-scale partial differentiation equations (PDEs). Modern scientific machine learning methods \cite{Cuomo2022Scientific}, which leverage automatic differentiation techniques, are also applicable in this regard. By preserving the problem’s continuous formulation, optimize-then-discretize approaches offer greater flexibility in handling varying spatial and temporal resolutions, making them particularly advantageous in high-fidelity simulations.

A fundamental challenge in ODE/PDE-constrained OCPs is scalability, as the significant number of design variables and constraints make classical \textit{centralized} approaches computationally prohibitive for problems in high dimensions \cite{DelosReyes2015Numerical}. The optimization tends to struggle with memory limitations and long solver runtime when applying centralized methods to large-scale problems. Further, since centralized methods rely on a single powerful processor, their performance also degrades rapidly in resource-constrained computing environments. To address these computational bottlenecks, domain decomposition methods have been developed to divide the computational domain into smaller subdomains, enabling parallel workload distribution across multiple processors \cite{Dolean2015Introduction,Prudencio2005Domain}.

Schwarz alternating methods have long been regarded as effective preconditioners for solving large-scale elliptic and parabolic PDEs, leveraging iterative subdomain solutions to accelerate convergence \cite{Lui1999Schwarz}. More recently, these methods have been adapted to PDE-constrained control problems, where allowing design variables to overlap across subdomains has been shown to improve convergence rates \cite{Shin2019Parallel, Na2022Convergence}. A key theoretical insight is that \textit{perturbations at the primal-dual subdomain boundaries exhibit an exponential decay phenomenon as they propagate into the subdomain}. 
This property serves as the foundation for iterative Schwarz-based optimization schemes, ensuring that the concatenation of overlapping subdomain solutions ultimately converge to the full, global optimum under appropriate boundary coordinations \cite{Na2022Convergence}. Building on this insight, recent work has refined the local design of the Schwarz method to a globally convergent fast overlapping temporal decomposition (FOTD) method \cite{Na2024Fast}. While these developments have demonstrated the effectiveness of Schwarz-based decomposition in structured settings, existing approaches have purely focused on discrete OCP formulations and relied on explicit domain partitioning.

In this work, we take the first step toward developing an overlapping Schwarz decomposition framework for general nonlinear PDE-constrained OCPs by studying the linear-quadratic setting. While our method can readily accommodate arbitrary spatial and temporal variables that arise in a PDE, we focus solely on temporal variables and concentrate our efforts on decomposing the temporal domain for clarity. This simplification allows us to clearly distinguish the continuous-time Schwarz method from aforementioned discrete-time counterparts. In particular, our method follows the indirect ``optimize-then-discretize'' paradigm and leverages the Pontryagin Minimum Principle (PMP) to partition the time domain into overlapping subdomains, each associated with a local OCP formulated as a Hamiltonian two-point boundary value problem in continuous time. At each Schwarz iteration, we specify the boundary values based on the previous iteration; solve local OCPs in parallel; and then concatenate the resulting local solutions to form a global trajectory. The method iterates until a gradient-based stopping criterion is satisfied, ensuring convergence to the global optimum.

Performing decomposition directly on continuous-time problems offers several advantages: (i) it preserves the clean analytical structure of the optimality conditions in function space, ensuring consistency with the underlying control problem; (ii) it enables efficient computation of gradient information without requiring full discretization; and (iii) it facilitates the use of flexible solvers with adaptive time-stepping, making it well-suited for stiff or multi-scale systems. In this work, we establish a convergence analysis for the continuous-time overlapping Schwarz method. We show that under standard controllability and coercivity conditions, \textit{the method converges linearly with a rate that improves exponentially with the overlap size}.

A key component of our analysis is to show that the exponential decay of sensitivity (EDS), previously established in discrete-time OCPs \cite{Na2020Exponential, Shin2022Exponential}, carries over to continuous-time settings. 
The EDS analysis draws conceptual connections to the \emph{turnpike property} in optimal control, which states that in long-horizon OCPs (without a terminal cost), for any initial state, the optimal trajectory remains close to a steady state for most of the horizon, deviating only near the boundaries \cite[Theorem 2]{Faulwasser2017turnpike}. More quantitatively, under strict dissipativity of the cost functional and system controllability conditions, the optimal trajectory approaches the steady state exponentially fast \cite{Gruene2013Economic,Damm2014Exponential}.
Continuous-time exponential sensitivity and turnpike estimates have also been developed for infinite-horizon optimal control problems. In particular, \cite{Gruene2020Exponential} established exponential sensitivity and turnpike estimates for linear-quadratic control of general evolution equations, building on sensitivity analysis for parabolic PDE-constrained problems motivated by model predictive control \cite{Gruene2019Sensitivity}. We also refer readers to \cite{Gruene2021Abstract} for an analysis of nonlinear sensitivity framework in abstract function spaces. These results provide an important continuous-time foundation for understanding why perturbations in boundary data or~forcing terms may have an exponentially localized influence along the optimal trajectory.

On the other hand, our analysis is complementary in both purpose and use. We specialize to time-varying linear-quadratic OCPs and explicitly track the perturbations induced by the Schwarz updates, including those to the left state boundary and the right terminal penalty parameter in each truncated subproblem. This enables us to derive explicit constants for the resulting perturbations of the state, control, and adjoint trajectories, and to translate these estimates into a contraction bound for the Schwarz iteration, whose convergence factor improves exponentially with the overlap length. Thus, while the underlying continuous-time decay property is closely related to turnpike theory, the present work leverages it to motivate a continuous-time optimize-then-discretize overlapping Schwarz method for long- (but finite-) horizon problems. Furthermore, our construction of Schwarz subproblems introduces an additional terminal cost due to domain truncation, which captures the perturbation from the terminal state to the system, a scenario typically not covered in classical turnpike theory. In addition, we derive explicit exponential convergence bounds, not with respect to steady states, but rather to the full-horizon optimal trajectory without perturbation. This EDS property of continuous-time OCPs, in turn, allows us to establish the linear convergence of the Schwarz methods, with a rate improved exponentially in the overlap size. Numerical experiments validate our theory and demonstrate promising performance of our infinite-dimensional Schwarz framework, offering flexibility across numerical solvers~and highlighting its efficiency and structure-preserving properties for general continuous-time OCPs.

\subsection{Organization} 

In \Cref{sec:problem-formulation}, we introduce the linear-quadratic OCPs in infinite dimensions, along with the assumptions and preliminary results. In \Cref{sec:proposedalg}, we present the continuous-time overlapping Schwarz algorithm. 
\Cref{sec:parameterized-ocp} analyzes a generic parameterized linear-quadratic OCP and establishes the exponential decay of sensitivity result that serves as the basis of the convergence analysis of the Schwarz scheme. In Section \ref{sec:convergence-analysis}, we prove the pointwise linear convergence of the Schwarz scheme and show that the linear rate improves exponentially with the overlap size. Numerical experiments are presented in \Cref{sec:experiments}, followed by conclusions and future directions in \Cref{sec:conclusions}. Some proofs and derivations are deferred to~the~appendices.

\subsection{Notation}

Throughout the paper, we consider state and control trajectories \( x : [0,T] \rightarrow \mathbb{R}^{n_x} \) and \( u : [0,T] \rightarrow \mathbb{R}^{n_u} \) defined over the finite time interval $[0, T]$. Let \( \| \cdot \| \) denote the $\ell_2$ norm for vectors and operator norm for matrices. For a symmetric matrix \( A \in \mathbb{R}^{n \times n} \), \( A \succ(\succeq) 0 \) means that \( A \) is positive (semi-)definite. We use \( \langle \cdot, \cdot \rangle \) to denote the norm-induced inner product in appropriate Banach spaces (including the Euclidean space $\mR^n$).
Let \( X \) be a Banach space and \( \mathcal{J} : f\in X \rightarrow \mathbb{R} \) be a functional. We denote the total (Gateaux) derivative of \( \mathcal{J} \) with respect to \( f \) by \( \delta \mathcal{J}/\delta f \). When \( \mathcal{J} \) depends on multiple functions, we use \( \partial \mathcal{J} / \partial f \) to denote partial derivatives with respect to \( f \). For further details on functional derivatives in infinite-dimensional spaces, we refer the reader to \cite{Zeidler1990Nonlinear}. The notation for functional derivatives is used throughout to describe \Cref{alg:mainprocedure}, drawing an analogy with its finite-dimensional discrete counterpart.

\section{Problem formulation}\label{sec:problem-formulation} 

\hskip-0.1cm Consider the following linear-quadratic OCP of the Bolza-type objective~\cite{Shapiro1966Lagrange}:
\begin{subequations} \label{eqn:linear-quadratic-ocp}
\begin{align}
\min_{u(\cdot),x(\cdot)} \;\;\;\; & \frac{1}{2}\int_0^T \begin{bmatrix}
x(t)\\ u(t)
\end{bmatrix}^\top\begin{bmatrix}
Q(t) & H^\top(t) \\
H(t) & R(t)
\end{bmatrix}\begin{bmatrix}
x(t)\\u(t)
\end{bmatrix} dt + \frac{1}{2}x(T)^\top Q_T x(T), \\
\text{s.t.}\quad\;\; & \dot{x}(t) = A(t)x(t) + B(t)u(t), \quad t\in (0,T], \label{eqn:linear-quadratic-ocp:b}\\
& x(0) = x_0, \label{eqn:linear-quadratic-ocp:c}
\end{align}
\end{subequations}
where $A: [0,T]\rightarrow\mathbb{R}^{n_x\times n_x}$, $B: [0,T]\rightarrow\rr^{n_x\times n_u}$ are the time-varying system matrices; $Q: [0,T]\rightarrow\mathbb{R}^{n_x\times n_x}$, $H: [0,T]\rightarrow\rr^{n_u\times n_x}$, $R: [0,T]\rightarrow\rr^{n_u\times n_u}$ are the running cost matrices; $x_0 \in \mathbb{R}^{n_x}$ is the initial state; and $Q_T \in \mathbb{R}^{n_x\times n_x}$ defines the terminal cost.

For Problem \eqref{eqn:linear-quadratic-ocp}, we impose the following regularity conditions. For any mapping $f: [0, T] \rightarrow\mR^{p\times q}$, we let $\|f\|_{\infty} = \sup_{t\in[0, T]}\|f(t)\|$. When $\mR^{p\times q}$ reduces to $\mR^p$ (e.g., for $u(\cdot), x(\cdot)$), $\|\cdot\|$ is understood as the~$\ell_2$~norm.

\begin{assumption}  \label{assumption:matrices-are-bounded}
We assume the time-varying matrices in the OCP \eqref{eqn:linear-quadratic-ocp} satisfy the following conditions:
\begin{enumerate}[label=(\alph*),topsep=0.3em]
\setlength\itemsep{0.4em}
\item The mappings \( A: [0,T] \to \mathbb{R}^{n_x \times n_x} \) and \( B: [0,T] \to \mathbb{R}^{n_x \times n_u} \) are continuous and uniformly bounded: $\|A\|_{\infty}\leq \lambda_A$, $\|B\|_\infty\leq \lambda_B$ for some constants $\lambda_A, \lambda_B>0$ independent of $T$.

\item The mappings $Q: [0,T] \to \mathbb{R}^{n_x \times n_x}$, $H: [0,T] \to \mathbb{R}^{n_u \times n_x}$, and $R: [0,T] \to \mathbb{R}^{n_u \times n_u}$ are continuous and uniformly bounded: $\|Q\|_\infty\leq \lambda_Q$, $\|H\|_\infty\leq \lambda_H$, $\|R\|_\infty\leq \lambda_R$ for some constants \mbox{$\lambda_Q, \lambda_H, \lambda_R>0$} independent of $T$.
\end{enumerate}
\end{assumption}

We should mention that the uniformity of the boundedness, i.e., the independence of the constants from the terminal time $T$, is not required if we consider $T$ to be fixed. Our analysis of the sensitivity of OCPs and the Schwarz method still holds without the uniformity; however, this stronger boundedness condition will indicate that the convergence rate of the method, determined by the boundedness constants, is independent of $T$. This uniform result is particularly desirable for considered long-horizon OCPs (i.e., as $T$ becomes~large)~\cite{Na2020Exponential, Na2022Convergence}.

The regularity conditions in \Cref{assumption:matrices-are-bounded} ensure the existence and uniqueness of the \textit{mild} state solution $x(\cdot)$ for any bounded control input $u(\cdot)$ \cite[Chapter 5.1]{Pazy1983Semigroups}. In particular, under \Cref{assumption:matrices-are-bounded}, for any bounded control mapping $u^*: [0, T] \rightarrow \mathbb{R}^{n_u}$, we define the corresponding state trajectory $x^*: [0, T] \rightarrow \mathbb{R}^{n_x}$ as the continuous function:
\begin{equation} \label{eqn:mild-solution}
x^*(t) = \Phi_A(t, 0)x_0 + \int_0^t \Phi_A(t, s) B(s) u^*(s)\,ds,\quad\quad t\in[0, T],
\end{equation}
where $\Phi_A(t, s)$ is the \textit{linear evolution operator} associated with the homogeneous system $\dot{x}(t) = A(t)x(t)$, satisfying $x(t) = \Phi_A(t, s) x(s)$. The representation of \eqref{eqn:mild-solution} is commonly referred to as the \emph{mild solution} of the state equation. The existence, uniqueness, and continuity of $ \Phi_A(\cdot,\cdot)$ follow from classical results in semigroup theory under the continuity of $A(t)$; see \cite[Theorems 5.1 and 5.2]{Pazy1983Semigroups}.

The mild solution \eqref{eqn:mild-solution} will serve as the foundation for analyzing how control inputs influence the state trajectory over time, which leads to the following notion of controllability.

\begin{definition}[Complete Controllability]

A state $x_{t_0} \in \rr^{n_x}$ is said to be \emph{\mbox{controllable}} at time \mbox{$t_0\ge 0$} if there exist $t_1> t_0$ and a control function $\bar{u}: [t_0, t_1] \rightarrow\rr^{n_u}$, depending on $t_0$ and $x_{t_0}$, such that the mild solution of the state equation \eqref{eqn:linear-quadratic-ocp:b} with the initialization $x(t_0)=x_{t_0}$ and the control input $\bar{u}$ satisfies $x(t_1)=0$.~That~is,
\begin{equation*}
\Phi_A(t_1, t_0)x_{t_0} + \int_{t_0}^{t_1} \Phi_A(t_1, s) B(s) \bar{u}(s)\,ds = 0.
\end{equation*}
The state system \eqref{eqn:linear-quadratic-ocp:b} is said to be \emph{completely controllable} if every state $x_{t_0}$ is controllable at all $t_0\ge 0$.
\end{definition}

We recall a fundamental result on controllability, which will be used frequently in our later analysis.

\begin{lemma}[Controllability Gramian] \label{lemma:controllabilitygramian}

The state system defined by the pair $(A, B)$ in \eqref{eqn:linear-quadratic-ocp:b} is completely controllable on an interval $[t_0, t_1] \subseteq [0, T]$ if and only if the symmetric matrix
\begin{equation} \label{eqn:definegrammain}
W_{A,B}(t_0, t_1) := \int_{t_0}^{t_1} \Phi_A(t_0, s) B(s) B^\top(s) \Phi_A^\top(t_0, s)\,ds \in \mR^{n_x\times n_x}
\end{equation}
is positive definite. The matrix \( W_{A,B} \) is referred to as the controllability Gramian for the pair \( (A, B) \). Furthermore, when \( W_{A,B}(t_0, t_1) \) is nonsingular, an explicit control that steers any state $x_{t_0}$ at time $t_0$ to the origin at time $t_1$ is given by
\begin{equation*}
u(t; t_0, x_{t_0}) = 
\begin{cases}
-B^\top(t)\Phi_A^\top(t_0, t)\,W_{A,B}^{-1}(t_0, t_1)\,x_{t_0}, & t \in [t_0, t_1], \\
0, & t > t_1.
\end{cases}
\end{equation*}
\end{lemma}

\begin{proof}
The results directly follow from \cite[Proposition 5.2, (5.9), and (5.10)]{Borkar2024Contributionsa}.	
\end{proof}

While complete controllability ensures that the system state can be driven to the origin over some finite time horizon, our analysis requires a quantitative refinement on the time required to achieve controllability. Specifically, we impose explicit lower and upper bounds on both the controllability Gramian and the evolution operator associated with the linear dynamics.

\begin{assumption}[Uniform Complete Controllability (UCC)] \label{assumption:uniform-controllability}
We assume there exist a fixed constant $\sigma\in(0, T)$ and some positive constants $\alpha_0(\sigma), \alpha_1(\sigma), \beta_0(\sigma), \beta_1(\sigma) > 0$ (may depend on $\sigma$), such that for every $t_0\in[0, T - \sigma]$, there exists a time $t_1 = t_1(t_0) \in[t_0, t_0 + \sigma]\subseteq[0, T]$ for which the following bounds~hold:
\begin{equation}\label{eqn:uniformboundednessofgramian}
\alpha_0(\sigma) I \preceq W_{A,B}(t_0, t_1) \preceq \alpha_1(\sigma) I \quad\; \text{ and } \quad\;
\beta_0(\sigma) I \preceq \Phi_A(t_1, t_0) W_{A,B}(t_0, t_1) \Phi_A^\top(t_1, t_0) \preceq \beta_1(\sigma) I.
\end{equation}

\end{assumption}

\Cref{assumption:uniform-controllability} is introduced in \cite[Definition 5.13]{Borkar2024Contributionsa} and is standard in the numerical control literature; see, for example, \cite[Assumption 2]{Na2022Convergence} and \cite[Definition 2]{Huerta2024Controllability}. It essentially states that the system, at any time $t_0$ and state $x_{t_0}$, can be controlled within an interval of length $\sigma$. Notably, we do not require controllability on the terminal segment $[T - \sigma, T]$ since we consider a finite- (but long-) horizon problem.

To ensure the well-posedness of Problem \eqref{eqn:linear-quadratic-ocp}, we impose the following coercivity conditions on the cost functional. We abuse the constant $\lambda_Q$ from Assumption \ref{assumption:matrices-are-bounded} for notational consistency.

\begin{assumption}\label{assumption:matrices-are-bounded-below}
There exist uniform constants $\gamma_R, \gamma_Q > 0$ independent of $T$ such that
\begin{equation*}
R(t) \succeq \gamma_R I, \quad \quad
Q(t) - H^\top(t) R^{-1}(t) H(t) \succeq \gamma_Q I,\quad\quad \forall t\in[0, T].
\end{equation*}
In addition, the terminal cost matrix satisfies $\gamma_Q I \preceq Q_T\preceq \lambda_Q I$ for some constants $\lambda_Q\geq \gamma_Q> 0$.
\end{assumption}

\Cref{assumption:matrices-are-bounded-below} guarantees the strong convexity of the cost functional, thereby ensuring the existence and uniqueness of the solution to OCP \eqref{eqn:linear-quadratic-ocp}.

\begin{remark}
Similar to Assumption \ref{assumption:matrices-are-bounded}, the uniformity of the constants is not required when we treat $T$ as a fixed constant, while the uniformity indicates that our results, including the convergence rate, are independent of $T$. We highlight that, under Assumption \ref{assumption:matrices-are-bounded}(b), Assumption \ref{assumption:matrices-are-bounded-below} is equivalent to what is commonly assumed in the literature (e.g., \cite[Assumption 3.1]{Na2020Exponential}, \cite[Assumption 2.1]{Xu2018Exponentially}), where a constant $\gamma' > 0$ independent of $T$ is assumed such that
\begin{equation}\label{nequ:1}
\begin{bmatrix}
Q(t) & H^\top(t) \\
H(t) & R(t)
\end{bmatrix} \succeq \gamma'I,\quad\quad \forall t\in [0,T].
\end{equation}
In particular, the condition \eqref{nequ:1} implies that
\begin{equation*}
R(t) \succeq \gamma' I,\quad\quad \begin{bmatrix}
Q(t) - \gamma' I & H^\top(t) \\
H(t) & R(t)
\end{bmatrix} \succeq 0,\quad\quad \forall t\in[0, T],
\end{equation*}
suggesting that Assumption \ref{assumption:matrices-are-bounded-below} holds with $\gamma_R = \gamma_Q = \gamma'$ \cite[Theorem 1.12]{Zhang2005Schur}. On the other hand, Assumption \ref{assumption:matrices-are-bounded-below} implies for any $\gamma'\leq \epsilon\gamma_R$ with $\epsilon\in(0, 1)$ that $R(t) - \gamma' I \succeq (1-\epsilon) R(t)$. Thus, we have 
\begin{multline*}
Q(t) - H^\top(t)(R(t) - \gamma' I)^{-1}H(t) \succeq Q(t) -  H^\top(t)R(t)^{-1}H(t)/(1-\epsilon) \\ \succeq Q(t) - H^\top(t)R(t)^{-1}H(t) - \frac{\epsilon}{1-\epsilon}\frac{\lambda_H^2}{\gamma_R}I \succeq \rbr{\gamma_Q - \frac{\epsilon}{1-\epsilon}\frac{\lambda_H^2}{\gamma_R}} I.
\end{multline*}
Thus, let $\epsilon$ be small enough such that $\gamma_Q - \frac{\epsilon}{1-\epsilon}\frac{\lambda_H^2}{\gamma_R} \geq \epsilon \gamma_Q$, then \eqref{nequ:1} holds with $\gamma' = \epsilon\min\{\gamma_Q, \gamma_R\}$. Without Assumption \ref{assumption:matrices-are-bounded}, Assumption \ref{assumption:matrices-are-bounded-below} is strictly weaker than \eqref{nequ:1} since one can see that $Q(t) = t^2+1$, $H(t)=t$, $R(t)=1$ satisfies Assumption \ref{assumption:matrices-are-bounded-below} but not \eqref{nequ:1} (the smallest eigenvalue of the quadratic matrix goes to zero as $t\rightarrow\infty$).

\end{remark}

\begin{remark}
We should mention that the strongly convex linear-quadratic setting is the natural first step of our study for the Schwarz scheme in continuous time, as it guarantees uniqueness of the optimal solution and permits explicit stability and sensitivity estimates, same as the discrete-time setting studied in \cite{Xu2018Exponentially}. If strong convexity fails, solutions may become nonunique, and global sensitivity estimates of the type used later in the paper may no longer hold. For nonlinear, nonconvex, or inequality-constrained problems, one should instead expect a local theory around a strict local minimizer, together with second-order sufficient conditions. 
Although the present study is not trivially generalizable to these more involved settings, it is still expected to be applied in conjunction with globalization or convexification mechanisms~\cite{Verschueren2017Sparsity,Na2020Exponential, Na2024Fast}. Furthermore, in active-set or SQP formulations (see~\Cref{sec:Intro}), many constrained algorithms compute local steps by solving equality-constrained or linear-quadratic subproblems. Thus, the present analysis should be~viewed as a fundamental building block toward more general nonlinear OCPs.
		
We also note that our assumptions do not require the open-loop system matrix $A$ itself to be stable. Open-loop unstable modes can be accommodated provided that the pair $(A,B)$ satisfies the UCC condition in Assumption \ref{assumption:uniform-controllability}. For example, consider the scalar system $\dot{x}(t)=x(t)+u(t)$, for which $A=1$ is open-loop unstable. Nevertheless, with $B=1$, the system is still uniformly completely controllable. Indeed, for any fixed \(\sigma\in(0,T)\) and any \(t_0\in[0,T-\sigma]\), choose \(t_1=t_0+\sigma\). In this case, the state-transition matrix satisfies $\Phi_A(t_1,t_0)=e^{t_1-t_0}$ and $\Phi_A(t_0,s)=e^{t_0-s}$ for any $s\in[t_0, t_1]$. Plugging the latter part into \eqref{eqn:definegrammain} and we obtain$\hskip0.5cm$
\begin{equation*}
W_{A,B}(t_0,t_1) = \int_{t_0}^{t_1} e^{2(t_0-s)}\,ds = \frac{1-e^{2(t_0-t_1)}}{2} = \frac{1-e^{-2\sigma}}{2}.	
\end{equation*}
Moreover,
\begin{equation*}
\Phi_A(t_1,t_0)W_{A,B}(t_0,t_1)\Phi_A(t_1,t_0)^\top = e^{2\sigma}\frac{1-e^{-2\sigma}}{2} = \frac{e^{2\sigma}-1}{2}.
\end{equation*}
Therefore, \Cref{assumption:uniform-controllability} holds with
\begin{equation*}
\alpha_0(\sigma)=\alpha_1(\sigma)=\frac{1-e^{-2\sigma}}{2},
\qquad\qquad
\beta_0(\sigma)=\beta_1(\sigma)=\frac{e^{2\sigma}-1}{2}.
\end{equation*}
This example shows that, although the open-loop dynamics are unstable, the unstable mode is uniformly~controllable, and hence our analysis still applies. If an unstable mode is not uniformly controllable, then the~stability and sensitivity estimates developed later may fail.

\end{remark}

Given the above assumptions, we state the first-order optimality conditions for the linear-quadratic time-varying OCP \eqref{eqn:linear-quadratic-ocp}. These conditions will be later used to truncate time intervals and formulate subproblems by specifying appropriate boundary conditions.

\begin{theorem}[Pontryagin's Minimum Principle (PMP)] \label{theorem:pontryagin}
	
Suppose Assumptions \ref{assumption:matrices-are-bounded} and \ref{assumption:matrices-are-bounded-below} hold for the OCP in \eqref{eqn:linear-quadratic-ocp}. Then, there is a unique, continuous optimal control solution $u^*: [0, T] \rightarrow \mathbb{R}^{n_u}$ along with a unique, continuously differentiable optimal state solution $x^*: [0, T] \rightarrow \mathbb{R}^{n_x}$ that solves \eqref{eqn:linear-quadratic-ocp}. Furthermore, we define the Hamiltonian as
\begin{equation} \label{eqn:problem-hamiltonian}
\mathcal{H}(t, x(t), u(t), \lambda(t)) := 
\frac12
\begin{bmatrix}
x(t)\\ u(t)
\end{bmatrix}^\top\begin{bmatrix}
Q(t) & H^\top(t) \\
H(t) & R(t)
\end{bmatrix}\begin{bmatrix}
x(t)\\u(t)
\end{bmatrix}+
\lambda(t)^{\top}(A(t)x(t) + B(t)u(t)).
\end{equation}
Then, there exists a unique, continuously differentiable adjoint (costate) solution $\lambda^*: [0,T] \rightarrow \rr^{n_x}$ such that \vskip-0.4cm
\begin{subequations}\label{equ:PMP}
\begin{align}
\dot{x}^*(t) & = \nabla_\lambda \mathcal{H}(t, x^*(t), u^*(t), \lambda^*(t)),  & x^*(0) &= x_0, \label{subeqn:pmp-state} \\
\dot{\lambda}^*(t) &= -\nabla_x \mathcal{H}(t, x^*(t), u^*(t), \lambda^*(t)),  & \lambda^*(T) &= Q_Tx^*(T), \label{subeqn:pmp-adjoint}\\
&\hskip-2cm  \mathcal{H}(t, x^*(t), u^*(t), \lambda^*(t)) \le \mathcal{H}(t, x^*(t), u, \lambda^*(t)), & & \hskip-1cm  \forall u \in \mathbb{R}^{n_u}, \; \forall t\in[0, T]. \label{eqn:optimalcontrolminimizes}
\end{align}
\end{subequations}
Conversely, if the triple $(x^*, u^*, \lambda^*)$ with $x^*$, $\lambda^*$ being continuously differentiable and $u^*$ being continuous satisfies the conditions \eqref{equ:PMP}, then $(x^*, u^*)$ is the unique solution to OCP \eqref{eqn:linear-quadratic-ocp}.
\end{theorem}

\begin{proof}
We refer the reader to the discussions of \cite[Sections 2.3 and 3.4]{Anderson2007Optimal} for the existence and uniqueness of the optimal solution to the linear-quadratic OCP \eqref{eqn:linear-quadratic-ocp}; also refer to \cite[Theorems 3.9 and 3.11, and Remark 3.12]{Chachuat2007Nonlinear} for the necessary and (Mangasarian) sufficient conditions of OCPs.
\end{proof}

\Cref{theorem:pontryagin} generalizes the KKT conditions of nonlinear optimization problems to the setting of infinite-dimensional problems. In particular, \Cref{theorem:pontryagin} provides an \textit{open-loop} characterization of the optimality: for each $t$, the optimal control \( u^*(t) \) is not expressed as a function of the state \( x^*(t) \), but rather satisfies the minimization conditions derived from the Hamiltonian. This form is particularly useful in the design and analysis of approximation algorithms. In contrast, the \textit{closed-loop} characterization, where the control is expressed explicitly as a function of the current state, is useful for studying stability. Such a formulation~typically arises from the dynamic programming principle, which we now briefly review.

\begin{theorem}[Hamilton-Jacobi-Bellman (HJB) Equation] \label{theorem:hjbequation} 
For each \( t \in [0,T] \), we define the optimal cost-to-go function as
\begin{equation}\label{eqn:optimal-cost-to-go}
J^*(t, z) := \min_{u(\cdot),x(\cdot)} \;\;\;\;
\frac{1}{2}\int_t^T \begin{bmatrix}
x(s)\\ u(s)
\end{bmatrix}^\top\begin{bmatrix}
Q(s) & H^\top(s) \\
H(s) & R(s)
\end{bmatrix}\begin{bmatrix}
x(s)\\u(s)
\end{bmatrix} ds + \frac{1}{2}x(T)^\top Q_T x(T),
\end{equation}
where the state trajectory \( x(\cdot) \) satisfies the linear dynamics $\dot{x}(s) = A(s)x(s)+B(s)u(s)$, $s \in (t, T]$ and $x(t) = z$. Let the Hamiltonian be defined as in \eqref{eqn:problem-hamiltonian}. Under Assumptions \ref{assumption:matrices-are-bounded} and \ref{assumption:matrices-are-bounded-below}, $J^*(t,z)$ is the solution of the Hamilton-Jacobi-Bellman (HJB) partial differential equation:
\begin{subequations}\label{eqn:hjb-equation}
\begin{align}
\frac{\partial J^*}{\partial t}(t, z) + \min_{u \in \mathbb{R}^{n_u}} \mathcal{H}(t, z, u, \nabla_z J^*(t, z)) &= 0, \label{subeqn:minimization-pde}\\
J^*(T, z) &= \frac12z^{\top}Q_Tz. \label{subeqn:minimization-pde1}
\end{align}
\end{subequations}
	
\end{theorem}

\begin{proof}
See \cite[Sections 2.2 and 2.3]{Anderson2007Optimal} for the argument. We also refer the readers to the literature in optimal control theory, including \cite{Athans2013Optimal, Bertsekas2005Dynamic,Hinze2006optimal,Troeltzsch2010Optimal, Bokanowski2021Relationship}.
\end{proof}

The upcoming section introduces the subproblem formulation used in our Schwarz decomposition, along with its theoretical justification, leveraging the optimality theorems presented above.

\section{Continuous-Time Overlapping Schwarz Decomposition}  \label{sec:proposedalg}  

In this section, we introduce a continuous-time overlapping Schwarz decomposition scheme for solving linear-quadratic OCPs. We decompose the time domain into overlapping subdomains and formulate a subproblem for each subdomain. The subproblem formulation depends on properly chosen boundary conditions. The Schwarz method enables parallel computation and achieves convergence to the global solution by iteratively updating these boundary conditions. Our subproblem formulation leverages the PMP in \Cref{theorem:pontryagin}, and we offer the flexibility to incorporate different numerical solvers for solving the continuous-time subproblems.

We begin by partitioning the time domain $[0,T]$ into $m$ subdomains using a sequence of time points:
\begin{equation*}
0 = t_0 < t_1 < \cdots < t_m = T, \quad \text{so that} \quad [0,T] = \bigcup_{j=1}^m [t_{j-1}, t_j].
\end{equation*}
We then define the overlap parameters \( (\tau_j^0, \tau_j^1) \) for each subdomain $j$, where \( \tau_j^0 >0 \) denotes the backward overlap size and \( \tau_j^1 > 0 \) the forward overlap size. Using these parameters, we define the overlapping subdomains $[t_j^0, t_j^1]$ as
\begin{equation*}
t_j^0 = \max\{t_{j-1} - \tau_j^0, 0\} \;\;\text{ and }\;\; t_j^1 = \min\{t_j +  \tau_j^1, T\}\quad\text{ for }\; j = 1,\ldots,m.
\end{equation*}

Each subproblem will be formulated and solved over its corresponding overlapping subdomain, allowing the scheme to iteratively propagate information across the entire time domain while exploiting parallelism at each \mbox{iteration}. To \mbox{formalize} the \mbox{proposed} scheme, we define the \mbox{following} \mbox{parameterized} \mbox{subproblems}. From this point onward, we denote the full OCP \eqref{eqn:linear-quadratic-ocp} as $\mathcal{P}([0,T]; x_0)$, with its \mbox{optimal} \mbox{solution} denoted by~$(x^*, u^*, \lambda^*)$.

\begin{definition}\label{definition:define-full-and-subproblems} 

For each $1 \le j \le m$, we define the $j$-th subproblem on the interval $[t_j^0, t_j^1]$ with boundary parameters $(p_j, q_j)\in \mathbb{R}^{n_x}\times \mathbb{R}^{n_x}$, denoted by $\mathcal{P}_j([t_j^0, t_j^1]; p_j, q_j)$, as
\begin{subequations}\label{eqn:defineparameterizedsubproblem}
\begin{align}
\min_{u_j(\cdot), x_j(\cdot)} \quad & \mathcal{J}_j[u_j, x_j; q_j] \coloneqq \frac12\int_{t_j^0}^{t_j^1} 
\begin{bmatrix}
x_j(t) \\ u_j(t)
\end{bmatrix}^\top\begin{bmatrix}
Q(t) & H^\top(t) \\
H(t) & R(t)
\end{bmatrix}\begin{bmatrix}
x_j(t) \\ u_j(t)
\end{bmatrix} dt + L_j(x_j(t_j^1); q_j), \label{subeqn:subproblem_cost} \\
\text{s.t. } \quad\;\;\;\; 
& \dot{x}_j(t) = A(t)x_j(t) + B(t)u_j(t), \quad t \in (t_j^0, t_j^1], \label{subeqn:subproblem_dynamics} \\
& x_j(t_j^0) = p_j, \label{subeqn:subproblem_initial}
\end{align}
\end{subequations}
where the terminal cost $L_j(x_j(t_j^1); q_j)$ is given by
\begin{equation}  \label{eqn:terminalcostparameterized}
L_j(x_j(t_j^1); q_j) \coloneqq \begin{cases}
\frac12x_j^\top(t_j^1)Q(t_j^1)x_j(t_j^1) - x_j^{\top}(t_j^1)Q(t_j^1)q_j, & \text{if } 1 \le j < m, \\
\frac12 x_j(T)^\top Q_T x_j(T), & \text{if } j = m.
\end{cases}
\end{equation}

\end{definition}

In \Cref{definition:define-full-and-subproblems}, the parameter $p_j\in \mathbb{R}^{n_x}$ defines the initial state, and $q_j\in \mathbb{R}^{n_x}$ defines the target state, which is also used to define the terminal cost. In \eqref{eqn:terminalcostparameterized}, the first case promotes continuity across overlapping intervals through a quadratic penalty, and the second case recovers the original terminal cost from the full problem \eqref{eqn:linear-quadratic-ocp}. The inclusion of an appropriate penalty on the boundary centered at a reference point is essential to ensure boundary consistency. We provide a graphical illustration of the Schwarz scheme in Figure~\ref{fig:osd-schematic}. The~following proposition establishes that, with properly chosen boundary parameters, the subproblems defined in \eqref{eqn:defineparameterizedsubproblem} recover the truncated optimal solution of the full problem.

\begin{figure}[t]
\centering
\begin{tikzpicture}[x=0.9cm,y=1.0cm,>=stealth,font=\small]
\tikzset{
core/.style={very thick},
subprob/.style={thick},
guide/.style={dashed,gray!60},
pt/.style={circle,fill,inner sep=1.3pt},
toplabel/.style={fill=white,inner sep=1.5pt},
plabel/.style={fill=white,inner sep=1.2pt,font=\scriptsize},
tau/.style={font=\footnotesize},
sol1/.style={very thick,red!70!black},
sol2/.style={very thick,blue!70!black},
sol3/.style={very thick,green!50!black},
exchangeup/.style={dashed,blue!70!black,->,line width=0.8pt},
exchangedown/.style={dashed,blue!70!black,->,line width=0.8pt}
}

\def\xA{0}
\def\xB{3.2}
\def\xC{8.1}
\def\xD{12.6}

\def\xaone{0}
\def\xbone{4.0}

\def\xatwo{2.3}
\def\xbtwo{9.1}

\def\xathree{7.1}
\def\xbthree{12.6}

\def\yaxis{0}
\def\ycore{0.34}
\def\ysubname{0.58}

\def\yone{1.45}
\def\ytwo{2.95}
\def\ythree{4.45}

\def\pdown{0.30}   
\def\tauup{0.16}   
\def\tautext{0.05} 

\def\ysolhi{6.10}
\def\ysollo{5.20}

\fill[fill=red!8,draw=none]   (\xA,0.12) rectangle (\xB,7.20);
\fill[fill=blue!8,draw=none]  (\xB,0.12) rectangle (\xC,7.20);
\fill[fill=green!8,draw=none] (\xC,0.12) rectangle (\xD,7.20);

\foreach \x in {\xA,\xB,\xC,\xD}{
\draw[guide] (\x,-0.08) -- (\x,7.23);
}

\draw[->] (-0.25,\yaxis) -- (13.6,\yaxis) node[right] {$t$};

\foreach \x/\lab in {
\xA/{$0=t_0$},
\xB/{$t_1$},
\xC/{$t_2$},
\xD/{$t_3=T$}
}{
\draw (\x,0.08) -- (\x,-0.08);
\node[below] at (\x,-0.08) {\lab};
}

\node[toplabel] at ({(\xA+\xD)/2},6.80)
{State $x$};

\draw[gray!35] (\xaone,\ysolhi-0.5) -- (\xbone,\ysolhi-0.5);
\draw[gray!35] (\xatwo,\ysollo) -- (\xbtwo,\ysollo);
\draw[gray!35] (\xathree,\ysolhi-0.5) -- (\xbthree,\ysolhi-0.5);

\draw[sol1,domain=\xaone:\xbone,samples=120,smooth]
plot (\x,{ \ysolhi + 1.10*exp(-1.25*(\x-\xA)) + 0.40*exp(-6.5*(\x-\xA)) -0.5});

\draw[sol2,domain=\xatwo:\xbtwo,samples=140,smooth]
plot (\x,{ \ysollo + 1.10*exp(-1.25*(\x-\xA)) + 0.40*exp(-6.5*(\x-\xA)) });

\draw[sol3,domain=\xathree:\xbthree,samples=120,smooth]
plot (\x,{ \ysolhi + 1.10*exp(-1.25*(\x-\xA)) + 0.40*exp(-6.5*(\x-\xA)) -0.5 });

\foreach \x in {0.08,0.18,0.30,0.44,0.60,0.78,0.98,1.20,1.45,1.75,2.10,2.50,2.95,3.45,3.85}{
\fill[red!70!black]
(\x,{ \ysolhi + 1.10*exp(-1.25*(\x-\xA)) + 0.40*exp(-6.5*(\x-\xA)) -0.5})
circle (1.1pt);
}

\foreach \x in {2.45,2.80,3.20,3.65,4.15,4.70,5.30,5.95,6.65,7.35,8.10,8.75}{
\fill[blue!70!black]
(\x,{ \ysollo + 1.10*exp(-1.25*(\x-\xA)) + 0.40*exp(-6.5*(\x-\xA)) })
circle (1.0pt);
}

\foreach \x in {7.25,8.00,8.90,10.00,11.20,12.30}{
\fill[green!50!black]
(\x,{ \ysolhi + 1.10*exp(-1.25*(\x-\xA)) + 0.40*exp(-6.5*(\x-\xA)) -0.5})
circle (1.1pt);
}

\draw[core] (\xA,\ycore) -- (\xB,\ycore);
\draw[core] (\xB,\ycore) -- (\xC,\ycore);
\draw[core] (\xC,\ycore) -- (\xD,\ycore);

\node at ({(\xA+\xB)/2},\ysubname) {Subproblem 1};
\node at ({(\xB+\xC)/2},\ysubname) {Subproblem 2};
\node at ({(\xC+\xD)/2},\ysubname) {Subproblem 3};

\draw[subprob] (\xaone,\yone) -- (\xbone,\yone);
\node[pt] at (\xaone,\yone) {};
\node[pt] at (\xbone,\yone) {};
\node[left=4pt]  at (\xaone,\yone) {$p_1$};
\node[right=4pt] at (\xbone,\yone) {$q_1$};

\draw[blue,<->] (\xB,\yone+\tauup) -- (\xbone,\yone+\tauup);
\node[tau,anchor=south] at ({(\xB+\xbone)/2},\yone+\tauup+\tautext) {$\tau_1^1$};

\node[plabel,fill=red!8] at ({(\xaone+\xbone)/2},\yone-\pdown)
{$P_1([t_1^0,t_1^1];\,p_1,q_1)$};

\draw[subprob] (\xatwo,\ytwo) -- (\xbtwo,\ytwo);
\node[pt] at (\xatwo,\ytwo) {};
\node[pt] at (\xbtwo,\ytwo) {};
\node[left=4pt]  at (\xatwo,\ytwo) {$p_2$};
\node[right=4pt] at (\xbtwo,\ytwo) {$q_2$};

\draw[blue,<->] (\xatwo,\ytwo+\tauup) -- (\xB,\ytwo+\tauup);
\node[tau,anchor=south] at ({(\xatwo+\xB)/2},\ytwo+\tauup+\tautext) {$\tau_2^0$};

\draw[blue,<->] (\xC,\ytwo+\tauup) -- (\xbtwo,\ytwo+\tauup);
\node[tau,anchor=south] at ({(\xC+\xbtwo)/2},\ytwo+\tauup+\tautext) {$\tau_2^1$};

\node[plabel,fill=blue!8] at ({(\xatwo+\xbtwo)/2},\ytwo-\pdown)
{$P_2([t_2^0,t_2^1];\,p_2,q_2)$};

\draw[subprob] (\xathree,\ythree) -- (\xbthree,\ythree);
\node[pt] at (\xathree,\ythree) {};
\node[pt] at (\xbthree,\ythree) {};
\node[left=4pt]  at (\xathree,\ythree) {$p_3$};
\node[right=4pt] at (\xbthree,\ythree) {$T$};

\draw[blue,<->] (\xathree,\ythree+\tauup) -- (\xC,\ythree+\tauup);
\node[tau,anchor=south] at ({(\xathree+\xC)/2},\ythree+\tauup+\tautext-0.1) {$\tau_3^0$};

\node[plabel,fill=green!8,xshift=0.35cm] at ({(\xathree+\xbthree)/2+1},\ythree-\pdown)
{$P_3([t_3^0,t_3^1];\,p_3,0)$};


\draw[exchangedown] (\xatwo,\ytwo-0.02) -- (\xatwo,\yone+0.02);

\draw[exchangeup] (\xbone,\yone+0.02) -- (\xbone,\ytwo-0.02);

\draw[exchangedown] (\xathree,\ythree-0.02) -- (\xathree,\ytwo+0.02);

\draw[exchangeup] (\xbtwo,\ytwo+0.02) -- (\xbtwo,\ythree-0.02);

\end{tikzpicture}
\caption{Schematic illustration of the continuous-time overlapping Schwarz scheme. The original time interval is partitioned into non-overlapping subintervals $[t_{j-1},t_j]$ of potentially unequal lengths, while each subproblem is solved on an extended interval $[t_j^0,t_j^1]$ with boundary parameters $p_j$ and $q_j$. The three disjoint colored curves at the top represent the state solutions of subproblems $1$, $2$, and $3$ on their enlarged subdomains. This illustrates how the continuous-time formulation can numerically accommodate stiff state evolutions using different discretizations. Subproblem $1$ (red), which resolves the rapidly varying portion of the solution, uses the finest grid; subproblem $2$ (blue), associated with a moderately varying region, uses a coarser grid; and subproblem $3$ (green), corresponding to the slowly varying tail, uses the coarsest grid. The vertical dashed arrows indicate the exchange of interface information between neighboring overlapping subproblems. After solving the local OCPs, Algorithm~1 retains only the~original subintervals and discards the overlapped regions when forming the global state $x$}.
\label{fig:osd-schematic}			
\end{figure}

We note that it is beneficial to formulate the subproblems so that their optimality conditions take the same form as the PMP equations \eqref{equ:PMP} for the full problem. This structural alignment ensures that a single numerical solver can be reused across parallel subproblems with minimal modification.

\begin{proposition} \label{proposition:fullproblemsolvedimpliessubproblemsolved}

Suppose Assumptions \ref{assumption:matrices-are-bounded} and \ref{assumption:matrices-are-bounded-below} hold for $\mathcal{P}([0,T]; x_0)$. For each $1\leq j\leq m$, we consider the subproblem $\mathcal{P}_j([t_j^0, t_j^1]; p_j^*,q_j^*)$ with the boundary parameters $(p_j^*,q_j^*)$ defined as
\begin{equation}\label{eqn:optimalparameterchoices}
p_j^* =  \begin{cases}
x_0, & \text{if } j = 1, \\
x^*(t_j^0), & \text{if } j > 1,
\end{cases} \quad \quad\quad
q_j^* = 
\begin{cases}
x^*(t_j^1) - Q^{-1}(t_j^1) \lambda^*(t_j^1), & \text{if } j < m, \\
0, & \text{if } j = m.
\end{cases}
\end{equation}
Then, $\mathcal{P}_j([t_j^0, t_j^1]; p_j^*,q_j^*)$ admits a unique solution $(x_j^*, u_j^*, \lambda_j^*): [t_j^0, t_j^1]\rightarrow \mR^{n_x}\times \mR^{n_u}\times \mR^{n_x}$, given by
\begin{equation*}
x_j^*(t) = x^*(t), \quad u_j^*(t) = u^*(t), \quad \lambda_j^*(t) = \lambda^*(t), \quad\quad t \in [t_j^0, t_j^1].
\end{equation*}
That is, the restriction of the full problem's solution to the interval \( [t_j^0, t_j^1] \) solves the subproblem exactly.

\end{proposition}

\begin{proof}
By the necessary and sufficient conditions stated in Theorem \ref{theorem:pontryagin}, it suffices to verify that $\{x^*(t), u^*(t), \\ \lambda^*(t)\}_{t\in[t_j^0, t_j^1]}$ satisfies the PMP system for the subproblem $\mathcal{P}_j([t_j^0, t_j^1]; p_j^*,q_j^*)$ under the parameter choices given in \eqref{eqn:optimalparameterchoices}. For any $1\leq j \leq m$, let us define the Hamiltonian of the $j$-th subproblem as
\begin{equation*}
\mathcal{H}_j(t, x_j(t), u_j(t), \lambda_j(t)) \coloneqq \frac{1}{2}\begin{bmatrix}
x_j(t)\\ u_j(t)
\end{bmatrix}^\top\begin{bmatrix}
Q(t) & H^\top(t) \\
H(t) & R(t)
\end{bmatrix}\begin{bmatrix}
x_j(t)\\u_j(t)
\end{bmatrix}+
\lambda_j(t)^{\top}(A(t)x_j(t) + B(t)u_j(t)).
\end{equation*}
Then, the PMP system for the subproblem $\mathcal{P}_j([t_j^0, t_j^1]; p_j,q_j)$ is
\begin{align*}
\dot{x}_j^*(t) & = \nabla_{\lambda_j} \mathcal{H}_j(t, x_j^*(t), u_j^*(t), \lambda_j^*(t)) = A(t) x_j^*(t) + B(t) u_j^*(t),  & x_j^*(t_j^0) &= p_j,  \\
\dot{\lambda}_j^*(t) &= -\nabla_{x_j} \mathcal{H}(t, x_j^*(t), u_j^*(t), \lambda_j^*(t)) = - \begin{bmatrix}
Q(t) & H^\top(t)
\end{bmatrix}\begin{bmatrix}
x_j^*(t)\\
u_j^*(t)
\end{bmatrix} - A^\top(t)\lambda_j^*(t),  & \lambda_j^*(t_j^1) &= \nabla_{x_j} L_j(x_j^*(t_j^1); q_j), \\
&\hskip-0.5cm  \mathcal{H}_j(t, x_j^*(t), u_j^*(t), \lambda_j^*(t)) \le \mathcal{H}_j(t, x_j^*(t), u, \lambda_j^*(t)), & & \hskip-1cm  \forall u \in \mathbb{R}^{n_u}, \; \forall t\in[t_j^0, t_j^1]. 
\end{align*}
Comparing the above system with the full PMP system in \eqref{equ:PMP}, we see that it suffices to verify $x^*(t_j^0) = p_j^*$ and $\lambda^*(t_j^1) = \nabla_{x_j} L_j(x^*(t_j^1); q_j^*)$, since all other conditions are subconditions and are directly implied by \eqref{equ:PMP}. Note that $x^*(t_j^0) = p_j^*$ is implied by \eqref{eqn:optimalparameterchoices} and \eqref{eqn:linear-quadratic-ocp:c}. When $j = m$, 
\begin{equation*}
\lambda^*(t_m^1) = \lambda^*(T) \stackrel{\eqref{subeqn:pmp-adjoint}}{=} Q_Tx^*(T) = \nabla_{x_m} L_m(x^*(t_m^1); q_m^*).
\end{equation*}
When $1\leq j<m$, we have
\begin{equation*}
\nabla_{x_j} L_j(x^*(t_j^1); q_j^*) \stackrel{\eqref{eqn:terminalcostparameterized}}{=} Q(t_j^1) \left(x^*(t_j^1) - q_j^*\right) \stackrel{\eqref{eqn:optimalparameterchoices}}{=}\lambda^*(t_j^1).
\end{equation*}
This completes the proof.
\end{proof}

\vspace{-0.7cm}

\begin{algorithm}[!htb]
\caption{Continuous-Time Overlapping Schwarz Decomposition}
\label{alg:mainprocedure}
\begin{algorithmic}[1]
\State \textbf{Input:} Initial state, control, and adjoint trajectory $(x^{(0)}, u^{(0)}, \lambda^{(0)}):[0,T]\rightarrow \rr^{n_x}\times \rr^{n_u}\times \rr^{n_x}$ with $x^{(0)}(0) = x_0$; domain decomposition and its overlapping subdomains $[t_{j-1},t_j]\subseteq [t_j^0, t_j^1]$, $1\leq j\leq m$.
\For{$k = 0, 1, 2, \ldots$}
\For{$j = 1$ \textbf{to} $m$ \textbf{in parallel}}
\State Specify boundary parameters: $p_j^{(k+1)} = x^{(k)}(t_j^0)$ and $q_j^{(k+1)} = x^{(k)}(t_j^1) - Q^{-1}(t_j^1) \lambda^{(k)}(t_j^1)$.
\State Solve the subproblem $\mathcal{P}_j([t_j^0,t_j^1]; p_j^{(k+1)}, q_j^{(k+1)})$ to obtain $(x_j^{(k+1)}, u_j^{(k+1)}, \lambda_j^{(k+1)})$ (e.g., with warm initialization $(x_j^{(k)}, u_j^{(k)}, \lambda_j^{(k)})$ and Algorithm \ref{alg:gradient-descent-subinterval}).
\EndFor
\State Aggregate $(x^{(k+1)}, u^{(k+1)}, \lambda^{(k+1)})(t) \coloneqq (x_j^{(k+1)}, u_j^{(k+1)}, \lambda_j^{(k+1)})(t)$ for $t\in[t_{j-1}, t_j)$ when $1\leq j\leq m-1$ and for $t\in[t_{m-1}, t_m]$ when $j=m$.
\EndFor
\end{algorithmic}
\end{algorithm}

\vspace{-0.3cm}

Since each subproblem admits a unique optimal solution, Proposition \ref{proposition:fullproblemsolvedimpliessubproblemsolved} guarantees that, with correctly specified boundary parameters, solving $\mathcal{P}_j([t_j^0, t_j^1]; p_j^*,q_j^*)$ yields the restriction of the global solution $(x^*, u^*, \lambda^*)$ to the sub-interval $[t_j^0, t_j^1]$. This observation motivates our decomposition scheme. Since in practice the optimal boundary parameters are unknown without solving full problem $\mathcal{P}([0,T];x_0)$, we present the continuous-time overlapping Schwarz algorithm in \Cref{alg:mainprocedure} to address this, which iteratively refines the parameters $(p_j, q_j)$. Starting with an imperfect set of parameters, the algorithm improves them over iterations, and ultimately drives them to converge to the full problem solution. We also present a gradient descent procedure for solving the subproblems in Appendix \ref{appendix:deriveadjointgradient}, where the required gradient information is derived via functional calculus.

\begin{remark}[Relation to shooting methods]
The Schwarz method considered here is related in spirit to multiple shooting, in the sense that both approaches decompose the time interval and solve local dynamical subproblems on shorter subintervals \cite{Bock1984Multiple}. However, the algorithmic roles of the subintervals differ substantially. In classical direct multiple shooting, the dynamics on each shooting interval are integrated as an initial-value problem starting from an artificial shooting state. These intermediate shooting states are introduced as global decision variables, and continuity between neighboring shooting intervals is enforced explicitly through matching constraints in a global nonlinear program. This lifted formulation improves numerical robustness by shortening the integration horizons and reducing the nonlinearity of the resulting nonlinear~program; it is particularly useful for unstable or sensitive dynamics.

By contrast, \Cref{alg:mainprocedure} is an overlapping Schwarz iteration. Each local OCP is solved independently~on an overlapping subdomain using initial and boundary data inherited from the previous global iterate, and consistency is enforced iteratively through boundary updates rather than by solving a single global matching system. Such overlaps are not part of the standard multiple-shooting formulation and would require additional consistency constraints. Consequently, our convergence analysis relies on exponential decay of boundary~sensitivities rather than on the conditioning of a global multiple-shooting KKT system.
\end{remark}

In the following sections, we turn to the analysis of the properties of the linear-quadratic OCPs to investigate the theoretical foundations of the convergence of the Schwarz method.

\section{Exponential Decay of Sensitivity in Linear-Quadratic Control} \label{sec:parameterized-ocp} 

While \Cref{proposition:fullproblemsolvedimpliessubproblemsolved} provides a theoretical foundation for recovering the full problem solution using the overlapping Schwarz scheme under ideal boundary conditions, it relies on knowledge of the exact optimal trajectory that is not available in practice. To rigorously justify the convergence behavior of \Cref{alg:mainprocedure}, we conduct a stability analysis of the OCP \eqref{eqn:linear-quadratic-ocp}. In particular, we examine how errors in the boundary data influence the resulting solution, and how such perturbations propagate through the system dynamics over time.

We begin by introducing a parameterized linear-quadratic OCP that generalizes \eqref{eqn:linear-quadratic-ocp} by including additional dependence on an auxiliary input mapping \( d: [0,T] \rightarrow \mathbb{R}^{n_d} \) and two boundary vectors $d_0\in\mR^{n_x}, d_T\in\mR^{n_d}$:
\begin{subequations}\label{eqn:parameterized-linear-quadratic-program}
\begin{align}
\min_{u(\cdot),\,x(\cdot)}\quad
& \frac12\int_0^T
\begin{bmatrix}
x(t)\\ u(t)\\ d(t)
\end{bmatrix}^\top
\begin{bmatrix}
Q(t) & H^\top(t) & G^\top(t) \\
H(t) & R(t) & W^\top(t) \\
G(t) & W(t) & 0
\end{bmatrix}
\begin{bmatrix}
x(t)\\ u(t)\\ d(t)
\end{bmatrix}
dt + \frac12 \begin{bmatrix}
x(T)\\ d_T
\end{bmatrix}^\top\begin{bmatrix}
Q_T & G_T^\top \\
G_T & 0
\end{bmatrix}\begin{bmatrix}
x(T)\\ d_T
\end{bmatrix}, \\
\text{s.t.} \quad\;\;\; & \dot{x}(t) = A(t)x(t) + B(t)u(t) + C(t)d(t), \quad t \in (0,T], \\
& x(0) = d_0.
\end{align}
\end{subequations}
Here, $G: [0,T] \rightarrow \mathbb{R}^{n_d \times n_x}$, $W: [0,T] \rightarrow \mathbb{R}^{n_d \times n_u}$, and $G_T \in \mathbb{R}^{n_x \times n_d}$ represent additional linear contributions of the parameter trajectory to the cost; and $C: [0,T] \rightarrow \mathbb{R}^{n_x \times n_d}$ introduces a linear dependence of the dynamics on $d$. To ensure well-posedness, we extend Assumption \ref{assumption:matrices-are-bounded} to hold for parameter-dependent matrices.

\begin{assumption} \label{assumption:additional-boundedness-for-parameterized}
We assume the mappings $G, W, C$ are continuous and uniformly bounded: $\|G\|_\infty\leq \lambda_G$, $\|W\|_\infty\leq \lambda_W$, $\|C\|_\infty\leq \lambda_C$ for some constants $\lambda_G, \lambda_W, \lambda_C>0$ independent of $T$. In addition, $\|G_T\|\leq \lambda_G$.

\end{assumption}

In essence, Problem \eqref{eqn:parameterized-linear-quadratic-program} allows bounded linear dependence on the mapping $d$ and vectors $d_0, d_T$, which include the boundary parameterization as given by \Cref{proposition:fullproblemsolvedimpliessubproblemsolved}, where only $d_0, d_T$ are possibly nonzero at domain boundaries of subproblems. We further explain the motivation behind the formulation of Problem \eqref{eqn:parameterized-linear-quadratic-program} in the next remark.

\begin{remark}

The purpose of introducing Problem~\eqref{eqn:parameterized-linear-quadratic-program} is to provide a sensitivity framework for the perturbations that arise in the overlapping Schwarz iteration. At Schwarz iteration $k+1$, the left boundary value $p_j^{(k+1)}$ and the right terminal parameter $q_j^{(k+1)}$ of the $j$-th subproblem are inherited from the previous global iterate (see Line 4 of Algorithm \ref{alg:mainprocedure}), and can therefore be interpreted as defining a perturbed problem. These boundary data generally do not coincide with those induced by the global optimal solution in \eqref{eqn:optimalparameterchoices}, which can be interpreted as defining an unperturbed problem. Consequently, a crucial step in the convergence analysis of the Schwarz scheme is to quantify how errors in these boundary quantities affect the optimal state,~control, and adjoint trajectories within each subproblem domain.

Problem~\eqref{eqn:parameterized-linear-quadratic-program} is parameterized by an input trajectory \(d(\cdot)\) and two boundary vectors \(d_0\) and \(d_T\). The optimal state, control, and adjoint trajectories, $(x^*(\cdot),u^*(\cdot),\lambda^*(\cdot))$, can therefore be regarded as functions of these parameters. In this context, sensitivity analysis refers to the study of how perturbations in $d(\cdot)$, $d_0$, and $d_T$ affect the optimal solution trajectories. The vector $d_0$ represents perturbations of the initial boundary state, while $d_T$ represents perturbations of the terminal parameter appearing in the right boundary penalty. The trajectory $d(\cdot)$ is introduced to allow more general affine perturbations of the running cost and dynamics. This additional generality also makes it possible to describe perturbations of entire solution segments generated by~\Cref{alg:mainprocedure}, rather than only perturbations at isolated boundary points.

For sake of the analysis of the overlapping Schwarz scheme itself, only a special case of Problem \eqref{eqn:parameterized-linear-quadratic-program} is needed. Specifically, for the Schwarz subproblem $\P_j([t_j^0,t_j^1];p_j,q_j)$ in \Cref{definition:define-full-and-subproblems}, one recovers the same formulation with \(n_d=n_x\) by taking $d(t)=0$, $C(t)=0$, $G(t)=0$, $W(t)=0$, and $d_0=p_j$. For $j<m$, the terminal penalty is represented by setting $Q_T=Q(t_j^1)$, $G_T=-Q(t_j^1)$, and $d_T=q_j$, so that the terminal term~becomes
\begin{equation*}
\frac12 x(t_j^1)^\top Q(t_j^1)x(t_j^1)
- x(t_j^1)^\top Q(t_j^1)q_j,
\end{equation*}
which matches \eqref{eqn:terminalcostparameterized}. For the last subproblem $j=m$, $G_T=0$, $d_T=0$, and the original terminal cost is retained instead. Thus, Problem \eqref{eqn:parameterized-linear-quadratic-program} should be understood as a general framework for encoding the perturbations generated by the Schwarz iteration while keeping the Riccati equation associated with the quadratic part~unchanged.
\end{remark}

We now state the equations that are central to our analysis.

\begin{theorem}  \label{theorem:state-closed-loop-equation}

Under Assumptions \ref{assumption:matrices-are-bounded}--\ref{assumption:additional-boundedness-for-parameterized}, the optimal state trajectory $x^*(t)$ of the parameterized OCP \eqref{eqn:parameterized-linear-quadratic-program} satisfies the closed-loop system:
\begin{subequations}
\begin{align} \label{eqn:state-closed-loop-system}
\dot{x}^*(t) & = \cbr{A(t) - B(t)R^{-1}(t)H(t) - B(t)R^{-1}(t)B^\top(t)S(t)}x^*(t) - B(t)R^{-1}(t)B^\top(t)v(t)  \nonumber \\
&\quad  + \cbr{C(t) - B(t)R^{-1}(t)W^\top(t)}d(t), \quad t\in(0, T],\\
x^*(0) &= d_0,
\end{align}
\end{subequations}
where the mapping $S: [0,T] \rightarrow \mathbb{R}^{n_x \times n_x}$ solves the matrix Riccati equation:
\begin{subequations} \label{eqn:riccati-equation}
\begin{align}
\dot{S}(t) &= S(t)B(t)R^{-1}(t)B^\top(t)S(t) - S(t)\cbr{A(t) - B(t)R^{-1}(t)H(t)} \nonumber\\
&\quad - \cbr{A(t) - B(t)R^{-1}(t)H(t)}^\top S(t) - \cbr{Q(t) - H^\top(t)R^{-1}(t)H(t)}, \quad t\in [0,T),\\
S(T) &= Q_T,
\end{align}
\end{subequations}
and the mapping $v: [0,T] \rightarrow \mathbb{R}^{n_x}$ solves the following backward linear dynamics
\begin{subequations}\label{eqn:vector-equation}
\begin{align} 
\dot{v}(t) &= -\cbr{A(t) - B(t)R^{-1}(t)H(t) - B(t)R^{-1}(t)B^\top(t)S(t)}^\top v(t) \nonumber\\
& \quad + \cbr{W(t)R^{-1}(t)\sbr{B^\top(t)S(t) + H(t)} - (G(t) + C^\top(t)S(t))}^\top d(t), \quad t\in[0, T),\\
v(T) & = G_T^\top d_T.
\end{align}
\end{subequations}
Furthermore, the optimal control trajectory is given by the expression:
\begin{equation} \label{eqn:sensitivity-control-expression}
u^*(t) = -R^{-1}(t)\cbr{H(t)x^*(t) + B^\top(t)\lambda^*(t) + W^\top(t)d(t)},
\end{equation}
where the adjoint is given by
\begin{equation*}
\lambda^*(t) = S(t)x^*(t) + v(t).
\end{equation*}
\end{theorem}

\begin{proof}
The proof is provided in Appendix \ref{appendix:closed-loop-sensitivity-ocp}.
\end{proof}

The stability of the closed-loop system \eqref{eqn:state-closed-loop-system} is governed by the matrix-valued function $S$, which evolves according to \eqref{eqn:riccati-equation} and characterizes the sensitivity of the cost-to-go function \eqref{eqn:optimal-cost-to-go}. We will demonstrate that if $S$ remains positive definite and uniformly bounded (especially from below), the resulting states \eqref{eqn:state-closed-loop-system} exhibit exponential stability. Thus, quantifying the spectral properties of $S$ is essential for analyzing how perturbations in the parameters propagate through the system.

To facilitate the analysis of $S$, we introduce an auxiliary quadratic control problem that possesses the same Riccati equation \eqref{eqn:state-closed-loop-system} but without the linear dependence on the parameter $d$, which simplifies the derivation of spectral bounds.

\begin{lemma}[Shifted OCP] \label{theorem:shifted-ocp}
Suppose $Q, R, H, A, B$ be time-varying matrices satisfying Assumptions \ref{assumption:matrices-are-bounded} and \ref{assumption:matrices-are-bounded-below}. For any initial state $x_0 \in \mathbb{R}^{n_x}$, consider the following OCP on the interval $[0, T]$:
\begin{subequations} \label{eqn:riccati-equivalent-problem}
\begin{align}
\min_{u(\cdot), \, x(\cdot)}\quad & \frac12 \int_{0}^T
\begin{bmatrix}
x(t)\\ u(t)
\end{bmatrix}^\top\begin{bmatrix}
Q(t) - H^\top(t) R^{-1}(t) H(t) & 0\\
0 & R(t)
\end{bmatrix}\begin{bmatrix}
x(t)\\ u(t)
\end{bmatrix} dt + \frac12x^\top(T) Q_T x(T), \\
\text{s.t.} \quad & \dot{x}(t) = \big(A(t) - B(t) R^{-1}(t) H(t)\big) x(t) + B(t) u(t), \label{subeqn:shfitedpair}\\
& x(0) = x_0. \nonumber
\end{align}
\end{subequations}
Then, the associated Riccati matrix $S$ of \eqref{eqn:riccati-equivalent-problem} is given by the same Riccati equation \eqref{eqn:riccati-equation} as in \Cref{theorem:state-closed-loop-equation}.
\end{lemma}

\begin{proof}
See \cite[Section 3.12]{Kirk2004Optimal} for the matrix Riccati equation corresponding to a linear-quadratic OCP without cross-product terms in the cost functional. By directly applying \cite[(3.12-14)]{Kirk2004Optimal} to \eqref{eqn:riccati-equivalent-problem}, we obtain the desired result.
\end{proof}

We now study the properties of the Riccati matrix through the lens of the shifted problem \eqref{eqn:riccati-equivalent-problem}. Note that the auxiliary problem shares the same Riccati equation as the original parameterized control problem \eqref{eqn:parameterized-linear-quadratic-program}; thus, any boundedness or decay properties of $S$ can be analyzed through problem \eqref{eqn:riccati-equivalent-problem}. We begin the discussion with an intermediate result.

\begin{lemma} \label{lemma:riccati-positive-definite}
Under Assumptions \ref{assumption:matrices-are-bounded} and \ref{assumption:matrices-are-bounded-below}, a unique solution $S: [0,T]\rightarrow \rr^{n_x\times n_x}$ of the matrix Riccati equation \eqref{eqn:riccati-equation} exists and satisfies $S(t)\succ 0$ for any $t\in[0, T]$.
\end{lemma}

\begin{proof}
For Problem \eqref{eqn:riccati-equivalent-problem}, the existence and uniqueness of a positive semidefinite solution $S$ to the Riccati equation \eqref{eqn:riccati-equation} under Assumptions \ref{assumption:matrices-are-bounded} and \ref{assumption:matrices-are-bounded-below} is classical; see \cite[Equations 6.1--6.4]{Borkar2024Contributionsa}. We provide an explicit verification of strict positive definiteness. For any $t_0\in[0, T]$, we let $(x^*, u^*)$ denote the truncated optimal solution of \eqref{eqn:riccati-equivalent-problem} on $[t_0, T]$. By the PMP necessary conditions in \Cref{theorem:pontryagin}, there exists a continuously differentiable adjoint trajectory $\lambda^*: [t_0,T]\rightarrow \rr^{n_x}$ satisfying $\lambda^*(T) = Q_Tx^*(T)$ and
\begin{equation}\label{nequ:5}
\dot{\lambda}^*(t) = -\big(A(t) - B(t)R^{-1}(t)H(t)\big)^\top\lambda^*(t) - \big(Q(t) - H^\top(t)R^{-1}(t)H(t)\big)x^*(t), \quad\quad t\in[t_0, T).
\end{equation} 
In fact, applying \Cref{theorem:state-closed-loop-equation} with $d = 0$ (so $v=0$), we know the optimal adjoint is explicitly given by $\lambda^*(t) = S(t)x^*(t)$ and the optimal control is explicitly given by $u^*(t) = -R^{-1}(t)B^\top(t)\lambda^*(t)$. With these formulations and for any initial state $0\neq x_{t_0}\in \mR^{n_x}$, we have
\begin{multline}\label{nequ:6}
x_{t_0}^\top S(t_0)x_{t_0} = x^*(T)^\top Q_Tx^*(T) - \int_{t_0}^T \frac{d}{dt}(x^*(t)^\top S(t)x^*(t)) dt \\ = x^*(T)^\top Q_Tx^*(T) - \int_{t_0}^T \frac{d}{dt} \big(x^*(t)^ \top \lambda^*(t)\big)dt .
\end{multline}
On the other hand, we apply the chain rule and substitute the dynamics of $x^*$, $u^*$, and $\lambda^*$, and obtain for any $t\in(t_0, T)$,
\begin{align*}
\frac{d}{dt}  \big(x^*(t)^\top \lambda^*(t)\big) & = \dot{x}^*(t)^\top \lambda^*(t) + x^*(t)^\top \dot{\lambda}^*(t) \\
&\hskip-0.45cm \stackrel{\eqref{subeqn:shfitedpair}, \eqref{nequ:5}}{=} \cbr{(A(t)-B(t)R^{-1}(t)H(t))x^*(t) + B(t)u^*(t)}^\top \lambda^*(t) \\
& \quad - x^*(t)^\top \cbr{ \big(A(t) - B(t)R^{-1}(t)H(t)\big)^\top \lambda^*(t) + \big(Q(t) - H^\top(t)R^{-1}(t)H(t)\big) x^*(t) } \\
&= -x^*(t)^\top \big(Q(t) - H^\top(t) R^{-1}(t) H(t)\big) x^*(t) - \lambda^*(t)^\top B(t) R^{-1}(t) B^\top(t) \lambda^*(t) \\
& \leq -x^*(t)^\top \big(Q(t) - H^\top(t) R^{-1}(t) H(t)\big) x^*(t) \leq 0,
\end{align*} 
where the third equality is because $u^*(t) = -R^{-1}(t)B^\top(t)\lambda^*(t)$ and the last two \mbox{inequalities} are due to Assumption \ref{assumption:matrices-are-bounded-below}. Since $x_{t_0}\neq 0$, by the continuity of $x^*$, we know there exists $t_0'>t_0$ such that $x^*(t)\neq 0$ for any $t\in[t_0, t_0']$. Within $[t_0, t_0']$, the above derivation implies $\frac{d}{dt} \big(x^*(t)^ \top \lambda^*(t)\big)<0$ strictly (by Assumption \ref{assumption:matrices-are-bounded-below}). Thus, we have
\begin{equation*}
\int_{t_0}^T \frac{d}{dt} \big(x^*(t)^ \top \lambda^*(t)\big)dt \leq \int_{t_0}^{t_0'} \frac{d}{dt} \big(x^*(t)^ \top \lambda^*(t)\big)dt <0.
\end{equation*}
Combining the above display with \eqref{nequ:6} and noting that $x^*(T)^\top Q_Tx^*(T)\geq 0$, we have $x_{t_0}^\top S(t_0)x_{t_0}>0$. Since $x_{t_0}$ is any nonzero vector, we have $S(t_0)\succ 0$. This completes the proof.
\end{proof}

While \Cref{lemma:riccati-positive-definite} ensures that the solution to the matrix Riccati equation exists and is positive definite on $[0,T]$, this property alone is not sufficient to determine how perturbations propagate. In our context, it is important that $S$ is not only positive definite but also uniformly bounded from both above and below. To proceed, we require an intermediate result confirming that if the uniform controllability condition (\Cref{assumption:uniform-controllability}) holds for the system matrix pair $(A, B)$, then it also holds for the pair $(A - BR^{-1}H, B)$. The latter~corresponds to the shifted system in \eqref{eqn:riccati-equivalent-problem}.

\begin{lemma}\label{lemma:ucc-carries-over}

Consider the linear system $\dot{x}(t) = A(t)x(t) + B(t)u(t)$. Suppose $A: [0,T] \rightarrow \mathbb{R}^{n_x \times n_x}$ and $B: [0,T] \rightarrow \mathbb{R}^{n_x \times n_u}$ define a matrix pair $(A, B)$ satisfying Assumptions \ref{assumption:matrices-are-bounded} and \ref{assumption:uniform-controllability}. Then, for any continuous mapping $F: [0,T] \to \rr^{n_u \times n_x}$ satisfying $\|F\|_\infty\leq \lambda_F$ for some constant $\lambda_F>0$ independent of $T$, the system
\begin{equation*}
\dot{x}(t) = (A(t) + B(t)F(t))x(t) + B(t)u(t)
\end{equation*}
also satisfies UCC in the sense of \eqref{eqn:uniformboundednessofgramian}. In particular, there exist constants $\alpha_0'(\sigma), \alpha_1'(\sigma), \beta_0'(\sigma), \beta_1'(\sigma) > 0$, depending on $\sigma$ from Assumption \ref{assumption:uniform-controllability} but independent of $T$, such that for every $t_0\in[0, T-\sigma]$, there exists a time $t_1 = t_1(t_0)\in[t_0, t_0+\sigma]\subseteq[0, T]$ for which the following bounds hold:
\begin{equation*}
\alpha_0'(\sigma) I \preceq W_{A+BF, B}(t_0, t_1) \preceq \alpha_1'(\sigma) I \quad \text{and}\quad \beta_0'(\sigma) I \preceq \Phi_{A+BF}(t_1, t_0) W_{A+BF, B}(t_0, t_1) \Phi_{A+BF}^\top(t_1, t_0) \preceq \beta_1'(\sigma) I.
\end{equation*}
(The expressions of $\alpha_0'(\sigma), \alpha_1'(\sigma), \beta_0'(\sigma), \beta_1'(\sigma) > 0$ are provided in \eqref{eqn:shifted-gramian-lower-bound}, \eqref{eqn:shifted-gramian-upper-bound}, and \eqref{eqn:define-gramian-phi-upper-lower-bound}.)
\end{lemma}

\begin{proof}
By Assumption \ref{assumption:matrices-are-bounded}, we note that for any $0\le s\le t\le T$ (see \cite[Theorem 5.2(i)]{Pazy1983Semigroups}),
\begin{equation}\label{nequ:7}
\|\Phi_{A+BF}(t, s)\| \le \exp\left(\int_s^t\|A(\tau)+B(\tau)F(\tau)\|d\tau\right) \le e^{\left(\lambda_A+\lambda_B\lambda_F\right)(t-s)}.
\end{equation} 
Furthermore, we note that
\begin{align*}
1 & = \|I\| = \|\Phi_{A+BF}(t,s)\Phi^{-1}_{A+BF}(t,s)\| = \|\Phi_{A+BF}(t,s)\Phi_{A+BF}(s,t)\| \\
& = \|\Phi_{A+BF}(t,s)\Phi_{-(A+BF)^{\top}}^\top(t,s)\| \le \|\Phi_{A+BF}(t,s)\|\cdot\|\Phi_{-(A+BF)^{\top}}(t,s)\| \leq \|\Phi_{A+BF}(t,s)\| e^{\left(\lambda_A+\lambda_B\lambda_F\right)(t-s)},
\end{align*}
where the last inequality is due to \eqref{nequ:7}. The above inequality implies 
\begin{equation}\label{eqn:lower-bound-of-evolution-operator}
\|\Phi_{A+BF}(t,s)\| \ge e^{-\left(\lambda_A+\lambda_B\lambda_F\right)(t-s)}.
\end{equation}
Now, let us consider bounding the controllability Gramian for the pair $(A+BF,B)$. For any $t_0\in [0, T-\sigma]$ and $0\neq v \in\rr^{n_x}$, we let $t_1 = t_1(t_0)$ from Assumption \ref{assumption:uniform-controllability}. Then, by the definition \eqref{eqn:definegrammain} we have
\begin{align*}
& v^{\top}W_{A+BF,B}(t_0, t_1)v =\int_{t_0}^{t_1}v^\top \Phi_{A+BF}(t_0, s)B(s)B^{\top}(s)\Phi_{A+BF}^{\top}(t_0,s) vds =  \int_{t_0}^{t_1} \|B^{\top}(s)\Phi_{A+BF}^{\top}(t_0,s)v\|^2ds \\
& \leq  \lambda_B^2\|v\|^2\int_{t_0}^{t_1}\|\Phi_{-(A+BF)^\top}(s, t_0)\|^2ds \le \lambda_B^2\|v\|^2\int_{t_0}^{t_1}e^{2(\lambda_A+\lambda_B\lambda_F)(s - t_0)}ds \le\lambda_B^2\|v\|^2 \int_{t_0}^{t_0+\sigma}e^{2(\lambda_A+\lambda_B\lambda_F)(s-t_0)}ds \\
& = \frac{\lambda_B^2\|v\|^2}{2(\lambda_A + \lambda_B \lambda_F)} \rbr{\exp\cbr{2\sigma(\lambda_A + \lambda_B \lambda_F)}-1},
\end{align*}
which implies
\begin{equation}\label{eqn:shifted-gramian-upper-bound}
\|W_{A+BF,B}(t_0, t_1)\| \leq \frac{\lambda_B^2\rbr{\exp\cbr{2\sigma(\lambda_A + \lambda_B \lambda_F)}-1}}{2(\lambda_A + \lambda_B \lambda_F)}  \eqqcolon \alpha_1'(\sigma).
\end{equation}
For the lower bound, given any $0\neq v\in \rr^{n_x}$, we consider the zero-steering control defined on $[t_0, t_1]$ to be $u(t;v) := -B^{\top}(t)\Phi_A^{\top}(t_0,t)W_{A,B}^{-1}(t_0,t_1)v$ (see \Cref{lemma:controllabilitygramian}). With this choice of control, the solution of the system $\dot{x}(t) = A(t)x(t) + B(t)u(t;v)$, $t\in(t_0, t_1]$, $x(t_0)= v$, denoted by $x(t;v)$, satisfies $x(t_1,v) = 0$. Then, let us consider the system $(A+BF, B)$. Define the control trajectory
\begin{equation}\label{eqn:shifted-problem-zero-steering-control}
\tilde{u}(t;v) := u(t;v) - F(t)x(t; v), \quad \quad \text{ for }\;\;  t\in[t_0, t_1],
\end{equation}
and the corresponding state trajectory of the system with the above control: $\dot{x}(t) = (A(t)+B(t)F(t))x(t)+B(t)\tilde{u}(t;v)$, $t\in(t_0, t_1]$, $x(t_0) = v$, denoted by $\tilde{x}(t; v)$, satisfies $\tilde{x}(t; v) = x(t;v)$, $\forall t\in[t_0, t_1]$. To see this, we note that
\begin{align*}
\dot{\tilde{x}}(t; v) - \dot{x}(t; v) & = (A(t)+B(t)F(t))\tilde{x}(t;v) + B(t)\tilde{u}(t;v) - A(t)x(t,v) - B(t)u(t,v)\\
& = (A(t)+B(t)F(t))\tilde{x}(t;v) + \cancel{B(t)u(t;v)} - B(t)F(t)x(t;v) - A(t)x(t;v)-\cancel{B(t)u(t;v)} \\
& = (A(t)+B(t)F(t))(\tilde{x}(t; v) - x(t;v)).
\end{align*}
Combining the above display with the fact that $\tilde{x}(t_0;v) - x(t_0;v) = 0$, we know $\tilde{x}(t; v) = x(t;v)$, $\forall t\in[t_0, t_1]$. In particular, $\tilde{x}(t_1;v) = x(t_1;v) = 0$. On the other hand, by \cite{Zaitsev2015Criteria} and Lemma \ref{lemma:controllabilitygramian}, we know $W_{A+BF,B}(t_0,t_1)$ is positive definite, and the control trajectory
\begin{equation}\label{nequ:8}
\bar{u}(t; v) \coloneqq -B^{\top}(t)\Phi_{A+BF}^{\top}(t_0,t)W_{A+BF, B}^{-1}(t_0, t_1)v, \quad\; t\in [t_0, t_1]
\end{equation}
steers the state trajectory $x(t)$ of $(A+BF, B)$ from $v$ at $t_0$ to 0 at $t_1$. In fact, by \cite[Section 3.5, Theorem 5]{Sontag1998Mathematical}, $\bar{u}(t; v)$ is the unique control trajectory that attains the minimum square norm. That is,
\begin{equation*}
\int_{t_0}^{t_1}\|\bar{u}(t;v)\|^2dt \leq \int_{t_0}^{t_1}\|\tilde{u}(t;v)\|^2dt.
\end{equation*}
For the left-hand side, we have
\begin{align*}
\int_{t_0}^{t_1}\|\bar{u}(t;v)\|^2dt & \; = \int_{t_0}^{t_1}\bar{u}^{\top}(t;v)\bar{u}(t;v)dt \\
& \; \stackrel{\mathclap{\eqref{nequ:8}}}{=} \int_{t_0}^{t_1}v^{\top}W_{A+BF,B}^{-1}(t_0,t_1)\Phi_{A+BF}(t_0,t)B(t)B^{\top}(t)\Phi_{A+BF}^{\top}(t_0,t)W_{A+BF,B}^{-1}(t_0,t_1)vdt\\
& \;\stackrel{\mathclap{\eqref{eqn:definegrammain}}}{=} v^{\top}W_{A+BF,B}^{-1}(t_0,t_1)v.
\end{align*}
Combining the above two displays, we apply Assumption \ref{assumption:uniform-controllability} and the boundedness of $F(t)$, and obtain
\begin{align}\label{nequ:9}
& v^{\top}W_{A+BF,B}^{-1}(t_0,t_1)v  \le \int_{t_0}^{t_1}\|\tilde{u}(t;v)\|^2dt \stackrel{\eqref{eqn:shifted-problem-zero-steering-control}}{=} \int_{t_0}^{t_1}\| u(t;v) - F(t)x(t;v)\|^2dt \nonumber\\
& \le 2\int_{t_0}^{t_1}\|u(t;v)\|^2dt + 2\int_{t_0}^{t_1}\|F(t)x(t;v)\|^2dt = 2v^{\top}W_{A,B}^{-1}(t_0,t_1)v + 2\int_{t_0}^{t_1}\|F(t)x(t;v)\|^2dt \nonumber\\
& \le 2\alpha_0^{-1}(\sigma)\|v\|^2 + 2\lambda_F^2\int_{t_0}^{t_1}\|x(t;v)\|^2dt.
\end{align}
To upper bound the second term on the right-hand side, we use the following explicit expression of $x(t;v)$ (by applying \eqref{eqn:mild-solution}):
\begin{align*}
x(t;v) &= \Phi_A(t,t_0)v + \int_{t_0}^t\Phi_A(t,s)B(s)u(s;v)ds \\
& = \Phi_A(t,t_0)\left(I - W_{A,B}(t_0,t)W^{-1}_{A,B}(t_0,t_1)\right)v, \quad t\in [t_0,t_1].
\end{align*}
From the definition \eqref{eqn:definegrammain}, we have for any $t\in [t_0,t_1]$ that
\begin{equation}\label{eqn:gramian-can-only-grow}
W_{A,B}(t_0,t_1) = W_{A,B}(t_0,t)+W_{A,B}(t,t_1) \succeq W_{A,B}(t_0,t).
\end{equation}
Therefore, we have
\begin{align*}
 \int_{t_0}^{t_1}\|x(t;v)\|^2dt  & = \int_{t_0}^{t_1}\left\|\Phi_A(t,t_0)\left(I - W_{A,B}(t_0,t)W_{A,B}^{-1}(t_0,t_1)	\right)v\right\|^2dt \\
& \stackrel{\mathclap{\eqref{eqn:gramian-can-only-grow}}}{\le} \int_{t_0}^{t_1}e^{2\lambda_A\left(t-t_0\right)}\cdot\left( 
1 + \frac{\alpha_1(\sigma)}{\alpha_0(\sigma)}\right)^2\|v\|^2ds \le \int_{t_0}^{t_0+\sigma}e^{2\lambda_A\left(t-t_0\right)}\cdot\left( 
1 + \frac{\alpha_1(\sigma)}{\alpha_0(\sigma)}\right)^2\|v\|^2dt \\
& = \frac{1}{2\lambda_A}\left( 1 + \frac{\alpha_1(\sigma)}{\alpha_0(\sigma)}\right)^2\left(e^{2\lambda_A\sigma} - 1\right)\|v\|^2.
\end{align*}
Plugging the above display into \eqref{nequ:9}, we obtain
\begin{equation*}
v^{\top}W_{A+BF,B}^{-1}(t_0,t_1)v \le \left( \frac{2}{\alpha_0(\sigma)} + \frac{\lambda_F^2}{\lambda_A}\left( 1 + \frac{\alpha_1(\sigma)}{\alpha_0(\sigma)}\right)^2\cbr{\exp(2\lambda_A\sigma)-1}\right)\|v\|^2.
\end{equation*} 
Since $v\neq 0\in \rr^{n_x}$ is arbitrary and $W_{A+BF,B}(t_0,t_1)$ is positive definite, this proves that 
\begin{equation}  \label{eqn:shifted-gramian-lower-bound}
W_{A+BF,B}(t_0,t_1) \succeq \alpha_0'(\sigma)I \quad\text{with}\quad \alpha_0'(\sigma) =\left( \frac{2}{\alpha_0(\sigma)} + \frac{\lambda_F^2}{\lambda_A}\left( 1 + \frac{\alpha_1(\sigma)}{\alpha_0(\sigma)}\right)^2\cbr{\exp(2\lambda_A\sigma)-1}\right)^{-1}.
\end{equation}
Furthermore, we define
\begin{equation}\label{eqn:define-gramian-phi-upper-lower-bound}
\beta_0'(\sigma)\coloneqq \alpha_0'(\sigma)\cdot e^{-2\left(\lambda_A+\lambda_B\lambda_F\right)\sigma} \quad\quad\quad \text{and}\quad\quad\quad \beta_1'(\sigma) \coloneqq \alpha_1'(\sigma)\cdot e^{2\left(\lambda_A+\lambda_B\lambda_F\right)\sigma},
\end{equation}
and obtain for any $0\neq v\in \rr^{n_x}$ that
\begin{align*}
v^{\top}\Phi_{A+BF}(t_1,t_0)W_{A+BF,B}(t_0,t_1)\Phi_{A+BF}^{\top}(t_1,t_0)v & \stackrel{\eqref{nequ:7}}{\leq} \alpha_1'(\sigma)\cdot e^{2\left(\lambda_A+\lambda_B\lambda_F\right)\sigma}\|v\|^2
= \beta_1'(\sigma)\|v\|^2,\\
v^{\top}\Phi_{A+BF}(t_1,t_0)W_{A+BF, B}(t_0,t_1)\Phi_{A+BF}^{\top}(t_1,t_0)v & \stackrel{\eqref{eqn:lower-bound-of-evolution-operator}}{\geq} \alpha_0'(\sigma)\cdot e^{-2\left(\lambda_A+\lambda_B\lambda_F\right)\sigma}\|v\|^2 = \beta_0'(\sigma)\|v\|^2.
\end{align*}
Thus, we conclude the proof.
\end{proof}

The above lemma is helpful because we effectively require UCC of the matrix pair $(A-BR^{-1}H, B)$ to establish the upper and lower bounds of Riccati equation solution $S(t)$ (cf. \Cref{theorem:shifted-ocp}) and to further analyze perturbations in $d$ for the parameterized problem \eqref{eqn:parameterized-linear-quadratic-program}, while Assumption \ref{assumption:uniform-controllability} imposes UCC only on $(A,B)$. The above lemma thus provides the necessary connection. 
We emphasize that the UCC property is not required on the interval $[T-\sigma, T]$ since this segment has length at most $\sigma$. The upper and lower bounds on relevant quantities are effectively guaranteed by the explicit relationship between the Riccati matrix $S$ and the optimal cost functional (details in Appendix \ref{appendix:proof-main-proposition}). We now state the boundedness result for $S(t)$.

\begin{lemma} \label{proposition:mainprop}
Under Assumptions \ref{assumption:matrices-are-bounded}--\ref{assumption:additional-boundedness-for-parameterized}, there exist constants $c_0(\sigma), c_1(\sigma)>0$ independent of $T$ such that the unique solution of the matrix Riccati equation \eqref{eqn:riccati-equation} satisfies
\begin{equation*}
c_0(\sigma)I \preceq S(t_0) \preceq c_1(\sigma)I, \quad\; \text{for all }\; t_0 \in [0,T].
\end{equation*} 
(The expressions of $c_0(\sigma), c_1(\sigma)>0$ are provided in \eqref{eqn:final-riccati-bounds}.)
\end{lemma}

\begin{proof} 
The proof is provided in Appendix \ref{appendix:proof-main-proposition}.
\end{proof}

We have established that the Riccati matrix is uniformly bounded from above and below on the entire solution interval $[0,T]$ with constants independent of $T$. This result provides the foundation for analyzing sensitivity properties of the optimal state system \eqref{eqn:state-closed-loop-system}. The result is summarized in the next lemma.

\begin{lemma}[Exponential stability of optimal state system] \label{proposition:feedback-stable}

\hskip-0.2cm Define the following~\mbox{time-dependent}~matrix:
\begin{equation}\label{eqn:Zdefinition}
Z(t) := A(t) - B(t) R^{-1}(t) H(t) - B(t) R^{-1}(t) B^\top(t) S(t), \quad\quad t\in[0, T].
\end{equation}
Consider the homogeneous system $\dot{x}(t) = Z(t)x(t)$. Under Assumptions \ref{assumption:matrices-are-bounded}--\ref{assumption:additional-boundedness-for-parameterized}, the corresponding solution evolution operator satisfies the exponential decay bound: for any $0\le t_0\le t_1\le T$,
\begin{equation}  \label{eqn:operator-exponential-bound}
\|\Phi_Z(t_1, t_0)\| \le c_Z(\sigma) e^{-\rho_Z(\sigma) (t_1 - t_0)},
\end{equation}
where
\begin{equation}\label{nequ:18}
c_Z(\sigma) = \sqrt{\frac{c_1(\sigma)}{c_0(\sigma)}} \quad\quad\text{ and }\quad\quad \rho_Z(\sigma)= \frac{\gamma_Q}{2c_1(\sigma)}.
\end{equation}
\end{lemma}

\begin{proof}
Let us define the candidate Lyapunov function $V(t, x(t)) \coloneqq x(t)^\top S(t) x(t)$, $t\in[t_0, t_1]$. In particular, when $x(t)$ follows the dynamics $\dot{x}(t) = Z(t)x(t)$, we have
\begin{align*}
\frac{d}{dt} V(t, x(t)) & = \dot{x}^\top(t) S(t) x(t) + x^\top(t) \dot{S}(t) x(t) + {x}^\top(t) S(t) \dot{x}(t) \\
& = x^\top(t)Z^\top(t)S(t)x(t) + x^\top(t) \dot{S}(t) x(t) + x^\top(t)S(t)Z(t)x(t)\\
& = -{x}^\top(t) S(t) B(t) R^{-1}(t) B^\top(t) S(t) x(t) - {x}^\top(t) \big(Q(t) - H^\top(t) R^{-1}(t) H(t)\big) x(t) \\
& \leq - \gamma_Q\|x(t)\|^2 \leq - \frac{\gamma_Q}{c_1(\sigma)} V(t,x(t)),
\end{align*}
where the third equality is due to \eqref{eqn:Zdefinition} and \eqref{eqn:riccati-equation}, and the last inequality is due to \Cref{proposition:mainprop}. By Gr\"onwall's inequality, this implies that for any $t\in[t_0, t_1]$,
\begin{multline*}
c_0(\sigma)\|\Phi_{Z}(t,t_0)x(t_0)\|^2 = c_0(\sigma)\|x(t)\|^2 \le {x}^{\top}(t)S(t)x(t) = V(t, x(t)) \leq V(t_0, x(t_0)) \exp\left( -\frac{\gamma_Q}{c_1(\sigma)}(t - t_0) \right) \\
= {x}^{\top}(t_0)S(t_0){x}(t_0)\exp\left( -\frac{\gamma_Q}{c_1(\sigma)}(t - t_0) \right) \leq c_1(\sigma)\exp\left( -\frac{\gamma_Q}{c_1(\sigma)}(t - t_0) \right)\|x(t_0)\|^2.
\end{multline*}
This leads to
\begin{equation*}
\|\Phi_{Z}(t,t_0)x(t_0)\| \leq \sqrt{\frac{c_1(\sigma)}{c_0(\sigma)}}\exp\left( 
-\frac{\gamma_Q}{2c_1(\sigma)}\left(t-t_0\right)
\right)\|x(t_0)\|.
\end{equation*}
Since $x(t_0)$ can be any vector in $\mR^{n_x}$, we apply the fact that $\|\Phi_Z(t,t_0)\| := \max_{0\neq x\in\rr^{n_x}}\|\Phi_Z(t,t_0)x\|/\|x\|$ and complete the proof.
\end{proof}

\Cref{proposition:feedback-stable} provides a direct consequence which we call exponential decay of sensitivity (EDS). That is, the effect of a single-point perturbation on the optimal solutions of our parameterized problem \eqref{eqn:parameterized-linear-quadratic-program} is damped exponentially fast as one moves away from the perturbation point. We present this main result below. We need some additional notation.

Let us denote $x^*(d, d_0, d_T):[0, T]\rightarrow\mR^{n_x}$, $u^*(d, d_0, d_T):[0, T]\rightarrow\mR^{n_u}$, $\lambda^*(d, d_0, d_T):[0, T]\rightarrow\mR^{n_x}$ to be the unique optimal state, control, and adjoint solutions to Problem \eqref{eqn:parameterized-linear-quadratic-program} with the parameters $d:[0,T]\rightarrow \rr^{n_d}$ and $d_0\in\mR^{n_x}, d_T\in \rr^{n_d}$. For any $t'\in[0, T]$, we consider a single-point perturbation of $d$ at $t'$, given by
\begin{equation}\label{nequ:17}
d_p(t) = d(t) + l_{t'}\delta (t-t'), \quad\quad t\in[0, T],
\end{equation}
where $0\neq l_{t'}\in \mR^{n_d}$ is the perturbation direction and $\delta(\cdot)$ is the Dirac delta functional. With the perturbed parameters $(d_p, d_0, d_T)$, we denote the perturbed solutions of \eqref{eqn:parameterized-linear-quadratic-program} as $(x^*_p, u^*_p, \lambda^*_p)(d_p, d_0, d_T)$. The next result investigates the perturbation of the solutions, characterized by $\delta x = x_p^* - x^*$, $\delta u = u_p^* - u^*$, and $\delta \lambda = \lambda_p^* - \lambda^*$.

\begin{theorem}\label{theorem:eds-result-for-middle} 
Consider the perturbation of $d$ at $t'$ in \eqref{nequ:17} and suppose Assumptions \ref{assumption:matrices-are-bounded}--\ref{assumption:additional-boundedness-for-parameterized} hold. Then, there exists a constant $\Lambda = \Lambda(\sigma)>0$, independent of $T$, such that
\begin{equation*}
\|\delta x(t)\| + \|\delta u(t)\| + \|\delta \lambda(t)\| \le \|l_{t'}\|\cdot\Lambda(\sigma) e^{-\rho_Z(\sigma)\left|t-t'\right|},\quad\quad \forall t'\neq t\in [0, T],
\end{equation*} 
where $\rho_Z(\sigma)$ is defined in \eqref{nequ:18}. (The expression of $\Lambda(\sigma)$ is provided in \eqref{eqn:middle-eds}.)
\end{theorem}

\begin{proof}
We model the perturbation as a Dirac delta functional $\delta_{t'}(t)\coloneqq \delta(t-t')$, which, once taken an inner product with a function, results in a point evaluation at $t'\in [0,T]$ in the following sense: for any $[t_a,t_b]\subseteq [0,T]$,
\begin{equation}  \label{eqn:integrate-with-dirac-delta}
\int_{t_a}^{t_b}x(t)\delta_{t'}(t)dt = \begin{cases}
x(t'), &\quad \text{if } t'\in [t_a,t_b],\\
0, &\quad \text{if } t'\notin [t_a,t_b].
\end{cases}  
\end{equation} 
We refer the reader to classic texts in Sobolev spaces and distribution theory for a detailed discussion of the function-space implications of modeling with the Dirac delta \cite{IvarStakgold2011Green}. In particular, introducing a point perturbation at $t=t'$ induces a jump discontinuity in the optimal trajectories. By the arguments made in \cite[Section 2]{Barron1986Pontryagin}, Theorems \ref{theorem:pontryagin} and \ref{theorem:hjbequation} still hold uniquely but are to be interpreted in the almost everywhere sense.

Let us define $v_p:[0,T]\rightarrow \rr^{n_x}$ to be the vector equation \eqref{eqn:vector-equation} associated with perturbed parameters $d_p$. In particular, we have
\begin{align*} 
\dot{v}_p(t) &= -\cbr{A(t) - B(t)R^{-1}(t)H(t) - B(t)R^{-1}(t)B^\top(t)S(t)}^\top v_p(t) \nonumber\\
& \quad + \cbr{W(t)R^{-1}(t)\sbr{B^\top(t)S(t) + H(t)} - (G(t) + C^\top(t)S(t))}^\top d_p(t) \\
& \stackrel{\mathclap{\eqref{eqn:Zdefinition}}}{=} -Z^\top(t)v_p(t) + Y^\top(t)d_p(t), \quad t\in[0, T),\\
v_p(T) & = G_T^\top d_T,
\end{align*}
where in the second equality we also define
\begin{equation*}
Y(t):=W(t)R^{-1}(t)\sbr{B^\top(t)S(t) + H(t)} - (G(t) + C^\top(t)S(t)), \quad\quad t\in [0,T].
\end{equation*}
Let $v$ denote the original solution that satisfies the unperturbed equation \eqref{eqn:vector-equation}, and define $\delta v:= v_p-v$. Then, the equation of $\delta v$ is obtained by taking the difference between the above display and \eqref{eqn:vector-equation}, and is given by
\begin{equation*}
\delta\dot{v}(t)  = -Z^\top(t) \delta v(t) + Y^\top(t) l_{t'}\delta_{t'}(t), \quad t\in[0, T), \quad\quad\quad \delta v(T)  = 0.
\end{equation*}
For consistency of using the solution evolution operator as in \eqref{eqn:mild-solution}, we transfer the above backward equation to forward equation by performing the change of variable $r := T-t,\; t\in [0,T]$, and defining $\delta\tilde{v}(r):=-\delta v(T-r)$. Thus, we have $\delta\dot{\tilde v}(r)=\delta \dot{v}(T-r)$. Within the $r$-coordinate system and defining $\tilde{Z}(r):=Z(T-r)$, we obtain
\begin{align*}
\delta\dot{\tilde{v}}(r) & = \delta \dot{v}(T-r) = - Z^\top(T-r) \delta v(T-r) + Y^\top(T-r) l_{t'}\delta_{t'}(T-r) = \tilde{Z}^\top(r)\delta\tilde{v}(r) + Y^\top(T-r) l_{t'}\delta(T-r-t')\\
& = \tilde{Z}^\top(r)\delta\tilde{v}(r) + Y^\top(T-r) l_{t'}\delta(r - (T-t')) = \tilde{Z}^\top(r)\delta\tilde{v}(r) + Y^\top(T-r) l_{t'}\delta_{T-t'}(r), \quad\quad r\in(0, T],\\
\delta \tilde{v}(0) &= 0.
\end{align*}
Applying the mild solution in \eqref{eqn:mild-solution}, we have
\begin{align*}
\delta{\tilde{v}}(r) = \int_{0}^r\Phi_{\tilde{Z}^\top}(r, s)Y^\top(T-s) l_{t'}\delta_{T-t'}(s) ds \stackrel{\eqref{eqn:integrate-with-dirac-delta}}{=}\begin{cases}
0, \quad & r< T-t',\\
\Phi_{\tilde{Z}^\top}(r, T-t')Y^\top(t') l_{t'}, \quad &r\geq T-t'.
\end{cases}
\end{align*}
The above implies that
\begin{equation}\label{nequ:19}
\delta{{v}}(t) = -\delta\tilde{v}(T-t) = 
\begin{cases}
-\Phi_{\tilde{Z}^{\top}}(T-t,T-t')Y^\top(t') l_{t'}, \quad &t\le t'\\
0, \quad & t>t'.
\end{cases}
\end{equation}
We now relate $\Phi_{\tilde{Z}^\top}$ to $\Phi_Z$. For any fixed $r\in[0, T]$ and any $s\in[r, T]$, we know from \cite[Theorem 5.2(iv)]{Pazy1983Semigroups} that $\Phi_{\tilde{Z}^\top}(s, r)$ as a function of $s$ satisfies 
\begin{equation*}
\frac{\partial \Phi_{\tilde{Z}^{\top}}(s,r)}{\partial s} = \tilde{Z}^\top(s)\Phi_{\tilde{Z}^{\top}}(s,r) = Z^\top(T-s)\Phi_{\tilde{Z}^{\top}}(s,r),\quad s\in(r, T],\quad\quad\quad \Phi_{\tilde{Z}^{\top}}(r,r) = I.
\end{equation*}
On the other hand, the operator $\Phi_{Z}^{\top}(T-r,T-s)$ as a function of $s$ satisfies
\begin{equation*}
\frac{\partial \Phi_{Z}^{\top}(T-r,T-s)}{\partial s} = \left(\Phi_{Z}(T-r,T-s)Z(T-s)\right)^{\top} = Z^{\top}(T-s)\Phi^{\top}_{Z}(T-r,T-s), \; s\in(r, T],\;\; \Phi_{Z}(T-r,T-r)=I.
\end{equation*}
By the uniqueness of the solution \cite[Theorem 5.1]{Pazy1983Semigroups}, we finally obtain
\begin{equation}\label{nequ:22}
\Phi_{\tilde{Z}^{\top}}(s,r)= \Phi_Z^{\top}(T-r,T-s), \quad\quad 0\le r\le s\le T.
\end{equation}

Plugging the above display into \eqref{nequ:19}, and defining $\b1_{t\leq t'} = 1$ if $t\leq t'$ and 0 otherwise, we have
\begin{align}\label{nequ:20}
\|\delta v(t)\| & \leq \|\Phi_{\tilde{Z}^\top}(T-t, T-t')\|\cdot \|Y(t')\|\cdot \|l_{t'}\| \b1_{t\leq t'} = \|\Phi_Z(t', t)\|\cdot\|Y(t')\|\cdot\|l_{t'}\|\b1_{t\leq t'} \nonumber\\
& \leq \|l_{t'}\|c_Z(\sigma)e^{-\rho_Z(\sigma)(t'-t)}\cbr{\frac{\lambda_W}{\gamma_R}\rbr{\lambda_B c_1(\sigma)+\lambda_H} + \lambda_C c_1(\sigma)+\lambda_G }\b1_{t\leq t'} \nonumber\\
& = \|l_{t'}\|c_Z(\sigma) \cbr{\rbr{\frac{\lambda_W\lambda_B}{\gamma_R}+ \lambda_C}c_1(\sigma)+ \frac{\lambda_W\lambda_H}{\gamma_R}+\lambda_G } e^{-\rho_Z(\sigma)(t'-t)}\b1_{t\leq t'} \nonumber\\
& \eqqcolon \|l_{t'}\| c_v(\sigma) e^{-\rho_Z(\sigma)(t'-t)}\b1_{t\leq t'},
\end{align}
where the third inequality is due to Lemmas \ref{proposition:mainprop} and \ref{proposition:feedback-stable} and Assumptions \ref{assumption:matrices-are-bounded} and \ref{assumption:additional-boundedness-for-parameterized}.

With the above result \eqref{nequ:20}, we now establish the decay of $\delta x$. By \Cref{theorem:state-closed-loop-equation}, we have for any $t'\neq t \in [0,T]$,
\begin{equation*}
x^*(t)=\Phi_Z(t,0)d_0 + \int_0^{t}\Phi_Z(t,s)g(s;d,v)ds,
\end{equation*} 
where we define
\begin{equation*}
g(t;d,v) := -B(t)R^{-1}(t)B^{\top}(t)v(t) + \cbr{C(t)-B(t)R^{-1}(t)W^{\top}(t)}d(t).
\end{equation*}
Applying \Cref{theorem:state-closed-loop-equation} to the perturbed problem and taking the difference, we have
\begin{align*}
\delta x(t) & = \int_0^t\Phi_Z(t,s)\left(g(s;d_p,v_p)-g(s;d,v)\right)ds \\
& =\int_0^t\Phi_Z(t,s)\left(-B(s)R^{-1}(s)B^{\top}(s)\delta v(s) + \cbr{C(s) - B(s)R^{-1}(s)W^{\top}(s)}l_{t'}\delta_{t'}(s)\right)ds \\
& \stackrel{\mathclap{\eqref{eqn:integrate-with-dirac-delta}}}{=} - \int_0^t\Phi_Z(t,s)B(s)R^{-1}(s)B^{\top}(s)\delta v(s)ds + \Phi_Z(t,t')\cbr{C(t')-B(t')R^{-1}(t')W^{\top}(t')}l_{t'}\cdot\mathbf{1}_{t>t'}.
\end{align*}
Thus, we apply \Cref{proposition:feedback-stable} and obtain
\begin{align}\label{eqn:eds-for-optimal-state}
& \|\delta x(t)\| \leq \frac{\lambda_B^2}{\gamma_R}\int_{0}^{t}\|\Phi_Z(t,s)\|\|\delta v(s)\| ds + \rbr{\lambda_C + \frac{\lambda_W\lambda_B}{\gamma_R}}\|\Phi_Z(t,t')\| \|l_{t'}\| \mathbf{1}_{t>t'} \nonumber\\
& \stackrel{\mathclap{\eqref{nequ:20}}}{\leq} \|l_{t'}\|c_v(\sigma)\frac{\lambda_B^2}{\gamma_R}\int_{0}^{\min\{t, t'\}} \|\Phi_Z(t, s)\| e^{-\rho_Z(\sigma)(t'-s)} ds + \|l_{t'}\|\rbr{\lambda_C + \frac{\lambda_W\lambda_B}{\gamma_R}}\|\Phi_Z(t,t')\| \mathbf{1}_{t>t'} \nonumber\\
& \leq \|l_{t'}\|c_v(\sigma)c_Z(\sigma)\frac{\lambda_B^2}{\gamma_R}\int_{0}^{\min\{t, t'\}}e^{-\rho_Z(\sigma)(t-s)}e^{-\rho_Z(\sigma)(t'-s)} ds + \|l_{t'}\|c_Z(\sigma)\rbr{\lambda_C + \frac{\lambda_W\lambda_B}{\gamma_R}}e^{-\rho_Z(\sigma)(t-t')}\mathbf{1}_{t>t'} \nonumber\\
& = \|l_{t'}\|c_v(\sigma)c_Z(\sigma)\frac{\lambda_B^2}{\gamma_R} \cdot e^{-\rho_Z(\sigma)(t+t')}\frac{e^{2\rho_Z(\sigma)\min\{t,t'\}}-1}{2\rho_Z(\sigma)} + \|l_{t'}\|c_Z(\sigma)\rbr{\lambda_C + \frac{\lambda_W\lambda_B}{\gamma_R}}e^{-\rho_Z(\sigma)(t-t')}\mathbf{1}_{t>t'} \nonumber\\
& \leq \|l_{t'}\|\frac{c_v(\sigma)c_Z(\sigma)\lambda_B^2}{2\rho_Z(\sigma)\gamma_R}e^{-\rho_Z(\sigma)|t-t'|} + \|l_{t'}\|c_Z(\sigma)\rbr{\lambda_C + \frac{\lambda_W\lambda_B}{\gamma_R}}e^{-\rho_Z(\sigma)(t-t')}\mathbf{1}_{t>t'} \nonumber\\
& \leq \|l_{t'}\|c_Z(\sigma) \cbr{\frac{c_v(\sigma)\lambda_B^2}{2\rho_Z(\sigma)\gamma_R} + \frac{\lambda_W\lambda_B}{\gamma_R} + \lambda_C}e^{-\rho_Z(\sigma)|t-t'|} \eqqcolon \|l_{t'}\| \Lambda_x(\sigma)e^{-\rho_Z(\sigma)|t-t'|},
\end{align}
where $c_v(\sigma)$ is defined in \eqref{nequ:20}. Next, we study the adjoint trajectory. By \Cref{theorem:state-closed-loop-equation}, we have for any $t'\neq t\in[0, T]$,
\begin{align}\label{eqn:eds-of-lambda}
\|\delta\lambda(t)\| &= 
\|\lambda_p^*(t)-\lambda^*(t)\| = \|S(t)(x_p^*(t)-x^*(t))+(v_p(t)-v(t))\| \nonumber\\
& \leq c_1(\sigma)\|\delta x(t)\| + \|\delta v(t)\| \leq \|l_{t'}\|c_1(\sigma)\Lambda_{x}(\sigma)e^{-\rho_Z(\sigma)|t-t'|} + \|l_{t'}\| c_v(\sigma) e^{-\rho_Z(\sigma)(t'-t)}\b1_{t\leq t'} \nonumber\\
& \leq \|l_{t'}\|\cbr{c_1(\sigma)\Lambda_{x}(\sigma) + c_v(\sigma)}e^{-\rho_Z(\sigma)|t-t'|} \eqqcolon \|l_{t'}\| \Lambda_\lambda(\sigma)e^{-\rho_Z(\sigma)|t-t'|},
\end{align}
where the fourth inequality is due to \eqref{nequ:20} and \eqref{eqn:eds-for-optimal-state}. Finally, we consider the control trajectory. By \eqref{eqn:sensitivity-control-expression} in \Cref{theorem:state-closed-loop-equation}, we know for $t'\neq t \in [0,T]$ that
\begin{align}\label{eqn:eds-of-u}
& \|\delta u(t)\|  \coloneqq \|u_p^*(t)-u^*(t)\| \nonumber\\
& = \|-R^{-1}(t)\{H(t)x_p^*(t)+B^{\top}(t)\lambda_p^*(t) + W^{\top}(t)d_p(t)\} +R^{-1}(t)\{H(t)x^*(t)+B^{\top}(t)\lambda^*(t) + W^{\top}(t)d(t)\}\| \nonumber\\
& = \|-R^{-1}(t)\{H(t)\delta x(t) + B^{\top}(t)\delta\lambda(t) + W^{\top}(t)l_{t'}\delta_{t'}(t)\}\|  = \|-R^{-1}(t)\{H(t)\delta x(t) + B^{\top}(t)\delta\lambda(t)\}\|  \nonumber\\
&\leq \frac{\lambda_H}{\gamma_R}\|\delta x(t)\| + \frac{\lambda_B}{\gamma_R}\|\delta\lambda(t)\|  \stackrel{\eqref{eqn:eds-for-optimal-state},\eqref{eqn:eds-of-lambda}}{\le}  \|l_{t'}\| \cdot \frac{\Lambda_x(\sigma)\lambda_H + \Lambda_\lambda(\sigma)\lambda_B}{\gamma_R}e^{-\rho_Z(\sigma)\left|t-t'\right|} \nonumber\\
& \eqqcolon \|l_{t'}\|\Lambda_u(\sigma)e^{-\rho_Z(\sigma)\left|t-t'\right|}.
\end{align}
Combining \eqref{eqn:eds-for-optimal-state}, \eqref{eqn:eds-of-lambda}, \eqref{eqn:eds-of-u}, we complete the proof by defining
\begin{align}\label{eqn:middle-eds}
\Lambda & = \Lambda(\sigma) \coloneqq \Lambda_x(\sigma) + \Lambda_u(\sigma) + \Lambda_\lambda(\sigma) \nonumber\\
& = \rbr{1+\frac{\lambda_H}{\gamma_R}}\Lambda_x(\sigma) + \rbr{1+\frac{\lambda_B}{\gamma_R}}\Lambda_\lambda(\sigma)  = \rbr{1+\frac{\lambda_H}{\gamma_R}}\Lambda_x(\sigma) + \rbr{1+\frac{\lambda_B}{\gamma_R}}(c_1(\sigma)\Lambda_{x}(\sigma) + c_v(\sigma)) \nonumber\\
& = \cbr{1+\frac{\lambda_H}{\gamma_R} + c_1(\sigma)\rbr{1+\frac{\lambda_B}{\gamma_R}}}\Lambda_x(\sigma)+ c_v(\sigma)\rbr{1+\frac{\lambda_B}{\gamma_R}},
\end{align}
where $c_1(\sigma)$ is in \Cref{proposition:mainprop} (cf. \eqref{eqn:final-riccati-bounds}), $\Lambda_x(\sigma)$ is in \eqref{eqn:eds-for-optimal-state}, and $c_v(\sigma)$ is in \eqref{nequ:20}.
\end{proof}

\Cref{theorem:eds-result-for-middle} provides the EDS result with respect to the general parameters $d$ that Problem \eqref{eqn:parameterized-linear-quadratic-program} has linear dependence on in the integral term in the cost functional. To complete the section, we analyze the effects of perturbing the boundary parameterizations $d_0$ and $d_T$.
Following the setup as in \Cref{theorem:eds-result-for-middle}, we consider the boundary perturbations of Problem \eqref{eqn:parameterized-linear-quadratic-program}:
\begin{equation}\label{nequ:21}
d_{p,0} = d_0+l_0,\quad\quad\quad d_{p,T} = d_T+l_T,
\end{equation}
where $l_0\in\mR^{n_x}, l_T\in \mR^{n_d}$ are perturbation directions. We also use $(x^*_p, u^*_p, \lambda^*_p)(d, d_{p,0}, d_{p,T})$ to denote the perturbed solution and investigate the solution perturbations measured by $\delta x = x_p^* - x^*$, $\delta u = u_p^* - u^*$, and $\delta \lambda = \lambda_p^* - \lambda^*$.

\begin{theorem}\label{theorem:eds-for-boundary}
Consider the perturbation of $d_0, d_T$ in \eqref{nequ:21} and suppose Assumptions \ref{assumption:matrices-are-bounded}--\ref{assumption:additional-boundedness-for-parameterized} hold. Then, there exists a constant $\Lambda_b = \Lambda_b(\sigma)>0$, independent of $T$, such that
\begin{equation*}
\|\delta x(t)\| + \|\delta u(t)\| + \|\delta \lambda(t)\| \le \Lambda_b(\sigma) \left(\|l_0\|e^{-\rho_Z(\sigma)t} + \|l_T\|e^{-\rho_Z(\sigma)(T-t)}\right),\quad\quad t\in[0, T],
\end{equation*}
where $\rho_Z(\sigma)$ is defined in \eqref{nequ:18}. (The expression of $\Lambda_b(\sigma)$ is provided in \eqref{eqn:boundary-eds}.)
\end{theorem}

\begin{proof}
The proof follows from carrying out a similar analysis to that of \Cref{theorem:eds-result-for-middle}, where we consider the mild solution of the optimal state (see \eqref{eqn:mild-solution}) before and after introducing the perturbations $l_0, l_T$. Define $v_p = v+\delta v$ as the perturbed solution of \eqref{eqn:vector-equation} with terminal condition $v_p(T) = G_T^\top(d_T+l_T)$. Then, $\delta v = v_p - v$~satisfies
\begin{equation*}
\delta\dot{v}(t) = -Z(t)^\top \delta v(t), \quad t\in[0, T),\quad\quad \delta v(T) = G_T^{\top}l_T,
\end{equation*}
where $Z$ is defined in \eqref{eqn:Zdefinition}. By using the reparameterization $r:=T-t$ and defining $\delta\tilde{v}(r):=-\delta v(T-r)$ and $\tilde{Z}(r) \coloneqq Z(T-r)$, we obtain the forward in time dynamics
\begin{equation*}
\delta\dot{\tilde{v}}(r) = \tilde{Z}^{\top}(r)\delta\tilde{v}(r), \quad r\in (0, T],\quad\quad \delta \tilde{v}(0)  = -G_T^{\top}l_T.
\end{equation*}
Applying \eqref{eqn:mild-solution}, we have for $t\in [0,T]$ that
\begin{equation}\label{eqn:eds2-for-vector-perturb}
\|\delta v(t)\| = \|\delta\tilde{v}(T-t)\| = \|\Phi_{\tilde{Z}^\top}(T-t, 0)G_T^{\top}l_T\| \stackrel{\eqref{nequ:22}}{\leq} \lambda_G\|l_T\|\|\Phi_Z(T, t)\| \leq \|l_T\| \lambda_G c_Z(\sigma)e^{-\rho_Z(\sigma)(T-t)},
\end{equation}	
where the last inequality is due to \Cref{proposition:feedback-stable}. Furthermore, by \eqref{eqn:state-closed-loop-system}, we know
\begin{equation*}
\dot{x}_p^*(t) = Z(t)x_p^*(t) - B(t)R^{-1}(t)B^{\top}(t)v_p(t) + \cbr{C(t)-B(t)R^{-1}(t)W^{\top}(t)}d(t), \;\; t\in (0,T],\quad\; x_p^*(0) = d_0+l_0,
\end{equation*}
which leads to
\begin{equation*}
\delta \dot{x}(t) = Z(t)\delta x(t) - B(t)R^{-1}(t)B^{\top}(t)\delta v(t), \;\;\; t\in (0,T], \quad\quad \delta x(0) = l_0.
\end{equation*}
Thus, we obtain from \eqref{eqn:mild-solution} that for any $t\in [0,T]$,
\begin{align}\label{nequ:23}
\|\delta x(t)\| &= \left\|\Phi_Z(t,0)l_0 - \int_0^t\Phi_Z(t,s)B(s)R^{-1}(s)B^{\top}(s)\delta v(s)ds\right\| \nonumber\\
&\leq \|l_0\|\cdot c_Z(\sigma)e^{-\rho_Z(\sigma)t} + \frac{\lambda_B^2}{\gamma_R}\int_0^t\|\Phi_Z(t,s)\|\|\delta v(s)\|ds \nonumber\\
& \stackrel{\mathclap{\eqref{eqn:eds2-for-vector-perturb}}}{\le} \|l_0\|\cdot c_Z(\sigma)e^{-\rho_Z(\sigma)t} + \|l_T\|\cdot c_Z^2(\sigma)\frac{\lambda_B^2\lambda_G}{\gamma_R}\int_0^te^{-\rho_Z(\sigma)(t-s)}{e^{-\rho_Z(\sigma)(T-s)}}ds \nonumber\\
& = \|l_0\|\cdot c_Z(\sigma)e^{-\rho_Z(\sigma)t} + \|l_T\|\cdot \frac{c_Z^2(\sigma)}{2\rho_Z(\sigma)}\frac{\lambda_B^2\lambda_G}{\gamma_R}\left(e^{-\rho_Z(\sigma)(T-t)} - e^{-\rho_Z(\sigma)(T+t)}\right) \nonumber\\
& \leq \|l_0\|\cdot c_Z(\sigma)e^{-\rho_Z(\sigma)t} + \|l_T\|\cdot \frac{c_Z^2(\sigma)}{2\rho_Z(\sigma)}\frac{\lambda_B^2\lambda_G}{\gamma_R}e^{-\rho_Z(\sigma)(T-t)} \nonumber\\
& \leq \Lambda_{x,b}(\sigma)\rbr{\|l_0\|e^{-\rho_Z(\sigma)t} + \|l_T\|e^{-\rho_Z(\sigma)(T-t)}},
\end{align}
where we define
\begin{equation*}
\Lambda_{x,b}(\sigma) :=  \max\left\{ c_Z(\sigma), \; \frac{c_Z^2(\sigma)\lambda_B^2\lambda_G}{2\rho_Z(\sigma)\gamma_R}\right\}.
\end{equation*}
For the adjoint trajectory, we apply \Cref{theorem:state-closed-loop-equation}, \Cref{proposition:mainprop}, and have for any $t\in [0,T]$,
\begin{align}\label{nequ:24}
\|\delta\lambda(t)\| & = \|\lambda_p^*(t)-\lambda^*(t)\| = \|S(t)(x_p^*(t)-x^*(t))+(v_p(t)-v(t))\| \le \|S(t)\|\|\delta x(t)\| + \|\delta v(t)\| \nonumber\\
& \stackrel{\mathclap{\eqref{eqn:eds2-for-vector-perturb}, \eqref{nequ:23}}}{\leq} c_1(\sigma)\Lambda_{x,b}(\sigma)\rbr{\|l_0\|e^{-\rho_Z(\sigma)t} + \|l_T\|e^{-\rho_Z(\sigma)(T-t)}} + \|l_T\| \lambda_G c_Z(\sigma)e^{-\rho_Z(\sigma)(T-t)} \nonumber\\
& \leq \Lambda_{x,b}(\sigma)\rbr{c_1(\sigma) + \lambda_G}\rbr{\|l_0\|e^{-\rho_Z(\sigma)t} + \|l_T\|e^{-\rho_Z(\sigma)(T-t)}} \nonumber\\
& \eqqcolon \Lambda_{\lambda,b}(\sigma)\rbr{\|l_0\|e^{-\rho_Z(\sigma)t} + \|l_T\|e^{-\rho_Z(\sigma)(T-t)}}.
\end{align}
Finally, we consider the control trajectory in \eqref{eqn:sensitivity-control-expression} and have for $t\in [0,T]$,
\begin{align}\label{nequ:25}
\|\delta u(t)\|  & := \|u_p^*(t)-u^*(t)\| \stackrel{\eqref{eqn:sensitivity-control-expression}}{=} \|-R^{-1}(t)\{H(t)\delta x(t) + B^{\top}(t)\delta\lambda(t)\| \leq \frac{\lambda_H}{\gamma_R}\|\delta x(t)\| + \frac{\lambda_B}{\gamma_R}\|\delta\lambda(t)\| \nonumber \\
& \stackrel{\mathclap{\eqref{nequ:23}, \eqref{nequ:24}}}{\leq}\;\;\;\; \frac{\Lambda_{x,b}(\sigma)\lambda_H + \Lambda_{\lambda,b}(\sigma)\lambda_B}{\gamma_R}\rbr{\|l_0\|e^{-\rho_Z(\sigma)t} + \|l_T\|e^{-\rho_Z(\sigma)(T-t)}} \nonumber\\
& \eqqcolon \Lambda_{u,b}(\sigma)\rbr{\|l_0\|e^{-\rho_Z(\sigma)t} + \|l_T\|e^{-\rho_Z(\sigma)(T-t)}}.
\end{align}
Finally, we combine \eqref{nequ:23}, \eqref{nequ:24}, \eqref{nequ:25}, and complete the proof by defining $\Lambda_b(\sigma)$ as
\begin{align}\label{eqn:boundary-eds}
\Lambda_b & = \Lambda_b(\sigma) = \Lambda_{x,b}(\sigma) + \Lambda_{u,b}(\sigma) + \Lambda_{\lambda,b}(\sigma) \nonumber\\
& = \rbr{1+\frac{\lambda_H}{\gamma_R}}\Lambda_{x,b}(\sigma) + \rbr{1+\frac{\lambda_B}{\gamma_R}}\Lambda_{\lambda,b}(\sigma) =  \cbr{1+\frac{\lambda_H}{\gamma_R} + \rbr{1+\frac{\lambda_B}{\gamma_R}}(c_1(\sigma)+\lambda_G)}\Lambda_{x,b} \nonumber\\
& = \cbr{1+\frac{\lambda_H}{\gamma_R} + \rbr{1+\frac{\lambda_B}{\gamma_R}}(c_1(\sigma)+\lambda_G)}\max\left\{ c_Z(\sigma), \; \frac{c_Z^2(\sigma)\lambda_B^2\lambda_G}{2\rho_Z(\sigma)\gamma_R}\right\},
\end{align}
where $c_1(\sigma)$ is in \Cref{proposition:mainprop} (cf. \eqref{eqn:final-riccati-bounds}) and $c_Z(\sigma)$ and $\rho_Z(\sigma)$ are in \eqref{nequ:18}.
\end{proof}

Up to this point, we have established the sensitivity of the solution of the parameterized linear-quadratic problem \eqref{eqn:parameterized-linear-quadratic-program} to the perturbations in the input mapping $d$ and two boundary vectors $d_0, d_T$. 
We recall \Cref{proposition:fullproblemsolvedimpliessubproblemsolved} that the overlapping Schwarz scheme introduces boundary vector perturbations due to imperfect boundary conditions. In the next section, we leverage \Cref{theorem:eds-for-boundary} to show that, by recursively updating the boundary conditions, the overlapping Schwarz scheme achieves linear convergence with the linear rate decaying exponentially fast in terms of the size of the overlap. Before analyzing the Schwarz scheme, we further explain the relationship between Theorem \ref{theorem:eds-for-boundary} and the turnpike property.

\begin{remark}[Relation to dissipativity and turnpike properties]
The exponential decay estimates established in Theorems \ref{theorem:eds-result-for-middle} and \ref{theorem:eds-for-boundary} are closely related to, but distinct from, classical turnpike properties in optimal control. The turnpike property typically states that, under suitable dissipativity and reachability or controllability assumptions, an optimal trajectory for a long-horizon problem remains close to a steady optimal pair for most of the horizon, with possible boundary layers near the initial and terminal times; see, e.g., \cite{Faulwasser2017turnpike, Gruene2013Economic,Damm2014Exponential,Gruene2020Exponential,Gruene2021Abstract}. In the present linear-quadratic setting, the coercivity condition in Assumption~\ref{assumption:matrices-are-bounded-below} and the uniform complete controllability condition in Assumption~\ref{assumption:uniform-controllability} play a role analogous to dissipativity~and reachability assumptions: together they imply a Riccati representation and the exponential stability of the~optimal closed-loop dynamics in Lemma \ref{proposition:feedback-stable}.

However, the EDS property used in this paper is not a turnpike theorem in the classical sense. Classical turnpike estimates measure the distance of a single optimal trajectory to a steady-state or reference optimal solution. By contrast, Theorems \ref{theorem:eds-result-for-middle} and \ref{theorem:eds-for-boundary} establish a sensitivity estimate between two optimal trajectories corresponding to different perturbations of the data. In particular, Theorem~\ref{theorem:eds-result-for-middle} shows that a localized perturbation at time $t'$ has an exponentially decaying influence away from $t'$, while Theorem~\ref{theorem:eds-for-boundary} shows that perturbations of the left boundary state and the right terminal parameter decay exponentially into the interior of the horizon.

This distinction is essential for the overlapping Schwarz analysis in Section \ref{sec:convergence-analysis}. The convergence proof does not require showing that each local trajectory is close to a steady turnpike point. Instead, it requires showing that the local solution obtained from imperfect boundary data remains close, on the retained non-overlapping interval, to the restriction of the full-horizon optimal trajectory. Thus, EDS serves as the sensitivity mechanism that converts boundary errors into smaller interior errors. The overlap then removes the boundary layers generated by inaccurate interface data, which yields the contraction estimate in Lemma~\ref{theorem:stepwiseimprovement} and the global linear convergence result in Theorem~\ref{proposition:linear-convergence-in-l-infty} as seen later. In this sense, the present analysis can be viewed as using a turnpike-like exponential localization phenomenon for a different purpose: proving convergence of a continuous-time overlapping Schwarz iteration rather than characterizing convergence to a steady state.$\hskip0.5cm$	
\end{remark}

\section{Convergence of Continuous-Time OSD}  \label{sec:convergence-analysis} 

In this section, we present the convergence analysis as a consequence of EDS established in \Cref{sec:parameterized-ocp}, which provides a rigorous justification for the use of overlapping subdomains in the proposed \Cref{alg:mainprocedure}. Specifically, we set $d = 0$ in \eqref{eqn:parameterized-linear-quadratic-program} and analyze the subproblems in \eqref{eqn:defineparameterizedsubproblem}, parameterized by $p_j, q_j$.
Since perturbations arising from inaccuracies in the initial and terminal parameters of subproblems are damped exponentially as they propagate into the interior of each subinterval, we demonstrate that even when the boundary parameters $d_0 := p_j$ and $d_T := q_j$ are not initialized exactly at $p_j^*$, $q_j^*$ -- the optimal boundary values of the global problem \eqref{eqn:linear-quadratic-ocp} (in particular, those satisfying the condition in Proposition \ref{proposition:fullproblemsolvedimpliessubproblemsolved}) -- the resulting subproblem solutions remain close to the global optimal trajectory on each~non-overlapping interval $[t_{j-1}, t_j]$, provided that the overlaps between subdomains~are sufficiently~large.

We now present a key result that quantifies perturbations within the domain to errors introduced at the subdomain boundaries. We recall that $(x^*, u^*, \lambda^*)$ denotes the (unique) optimal state, control, and adjoint trajectories of the full problem $\mathcal{P}([0,T]; x_0)$ in \eqref{eqn:linear-quadratic-ocp}.

\begin{lemma}[Stagewise improvement] \label{theorem:stepwiseimprovement}
\hskip-0.35cm Suppose Assumptions \ref{assumption:matrices-are-bounded}--\ref{assumption:matrices-are-bounded-below} hold~for~\mbox{Problem}~\eqref{eqn:linear-quadratic-ocp},~and~\mbox{consider}~applying Algorithm \ref{alg:mainprocedure} to solve Problem \eqref{eqn:linear-quadratic-ocp}. In particular, $(x^{(k)}, u^{(k)}, \lambda^{(k)})$ is the $k$-th solution iterate concatenating $m$ pieces $\{(x_j^{(k)}, u_j^{(k)}, \lambda_j^{(k)})\}_{j=1}^m$, which are solutions of $\mathcal{P}_j([t_j^0, t_j^1]; p_j^{(k)}, q_j^{(k)})$ in \eqref{eqn:defineparameterizedsubproblem} with boundary~parameters $p_j^{(k)} = x^{(k-1)}(t_j^0)$ and $q_j^{(k)} = x^{(k-1)}(t_j^1) - Q^{-1}(t_j^1) \lambda^{(k-1)}(t_j^1)$. Let $\tau := \min_{1\le j\le m}\min\{\tau_j^0,\tau_j^1\}$ denote the smallest overlap size between subdomains. Then, we have for any $k\geq 1$, any $1\leq j\leq m$, and any~$t\in[t_{j-1}, t_j]$,
\begin{multline*}
\|x^{(k)}(t) - x^*(t)\| + \|u^{(k)}(t) - u^*(t)\| + \|\lambda^{(k)}(t) - \lambda^*(t)\| \\
\le c(\sigma)e^{-\rho_Z(\sigma) \tau} \left( \|x^{(k-1)}(t_j^0) - x^*(t_j^0)\| + \cbr{\|x^{(k-1)}(t_j^1) - x^*(t_j^1)\| + \|\lambda^{(k-1)}(t_j^1) - \lambda^*(t_j^1)\|}\cdot \b1_{j\neq m} \right),
\end{multline*}
where $\rho_Z(\sigma)$ is defined in \eqref{nequ:18} and $c(\sigma) \coloneqq \Lambda_b(\sigma)(1+1/\gamma_Q)$ with $\Lambda_b(\sigma)$ in \eqref{eqn:boundary-eds}. Here, for both $\rho_Z(\sigma)$ and $\Lambda_b(\sigma)$, we use $\lambda_Q$ in place of $\lambda_G$ and set $\lambda_C = \lambda_W = 0$.

\end{lemma}

\begin{proof}
By \Cref{proposition:fullproblemsolvedimpliessubproblemsolved}, the solution $(x^*,u^*,\lambda^*)$ of $\mathcal{P}([0,T];x_0)$ restricted to $[t_{j-1}, t_j]$ is equivalent to the solution of $\mathcal{P}_j([t_j^0,t_j^1]; p_j^*,q_j^*)$ (see \eqref{eqn:optimalparameterchoices} for definitions of $p_j^*,q_j^*$) with the overlaps discarded. We quantify how boundary errors propagate into the interior of each subdomain. For any $k\geq 1$ and $1\leq j\leq m$, we define two perturbation vectors on the two ends (only initial perturbation vector when $j=m$):
\begin{align*}
l_{j,0}^{(k)} & := p_j^{(k)} - p_j^* \stackrel{\text{Alg.} \ref{alg:mainprocedure}}{=} x^{(k-1)}(t_j^0)-x^*(t_j^0),\\
l_{j,T}^{(k)} & := q_j^{(k)} - q_j^* \stackrel{\text{Alg.} \ref{alg:mainprocedure}}{=} (x^{(k-1)}(t_j^1)-x^*(t_j^1))-Q^{-1}(t_j^1)(\lambda^{(k-1)}(t_j^1)-\lambda^*(t_j^1)).
\end{align*}
Note that $\mathcal{P}_j([t_j^0,t_j^1]; p_j, q_j)$ is in the form of \eqref{eqn:parameterized-linear-quadratic-program} with $d(t)=0$, $C(t)=0$, $G(t)=0$, $W(t)=0$, $G_T = -Q(t_j^1)$, $d_0 = p_j$, $d_T = q_j$. Therefore, by \Cref{theorem:eds-for-boundary} and \Cref{proposition:fullproblemsolvedimpliessubproblemsolved}, for any $t\in[t_{j-1}, t_j]$, we have
\begin{align*}
& \|x^{(k)}(t)-x^*(t)\|  + \|u^{(k)}(t)-u^*(t)\| + \|\lambda^{(k)}(t)-\lambda^*(t)\| \\
& = \|x_j^{(k)}(t)-x^*(t)\| + \|u_j^{(k)}(t)-u^*(t)\| + \|\lambda_j^{(k)}(t)-\lambda^*(t)\|\\
& \le \Lambda_b(\sigma) \left(\|l_{j,0}^{(k)}\|e^{-\rho_Z(\sigma)(t-t_j^0)} + \|l_{j,T}^{(k)}\|e^{-\rho_Z(\sigma)(t_j^1-t)}\b1_{j\neq m}\right) \\
&\le \Lambda_b(\sigma)e^{-\rho_Z(\sigma)\tau}\left(\|l_{j,0}^{(k)}\| + \|l_{j,T}^{(k)}\|\b1_{j\neq m}\right) \\
&\le \Lambda_b(\sigma)(1+1/\gamma_Q)e^{-\rho_Z(\sigma)\tau}\rbr{
\|x^{(k-1)}(t_j^0)-x^*(t_j^0)\| + \cbr{\|x^{(k-1)}(t_{j}^1)-x^*(t_{j}^1)\| + \|\lambda^{(k-1)}(t_j^1)-\lambda^*(t_j^1)\|}\b1_{j\neq m}},
\end{align*}
where the third inequality is due to  $t-t_j^0\geq \tau$ and $t_j^1-t\geq \tau$ when $t\in [t_{j-1}, t_j]$. We complete the proof.
\end{proof}

We are now ready to present our main convergence result of the overlapping Schwarz scheme.

\begin{theorem}[Linear convergence of overlapping Schwarz] \label{proposition:linear-convergence-in-l-infty}

Suppose Assumptions \ref{assumption:matrices-are-bounded}--\ref{assumption:matrices-are-bounded-below} hold for Problem \eqref{eqn:linear-quadratic-ocp}, and consider applying Algorithm \ref{alg:mainprocedure} to solve Problem \eqref{eqn:linear-quadratic-ocp}. For any $k\geq 0$, we define
\begin{equation*}
\omega^{(k)} = \max\{\|x^{(k)} - x^*\|_\infty, \|u^{(k)} - u^*\|_\infty, \|\lambda^{(k)} - \lambda^*\|_\infty\}.
\end{equation*}
Then, Algorithm \ref{alg:mainprocedure} exhibits the linear convergence
\begin{equation*}
\omega^{(k+1)} \leq 3c(\sigma)e^{-\rho_Z(\sigma) \tau} \cdot \omega^{(k)},
\end{equation*}
where $c(\sigma)$, $\rho_Z(\sigma)$, and $\tau$ are defined in \Cref{theorem:stepwiseimprovement}.
\end{theorem}

\begin{proof}
The result directly follows by applying \Cref{theorem:stepwiseimprovement}:
\begin{align*}
\omega^{(k+1)} & \leq \sup_{t\in[0, T]}\|x^{(k+1)}(t) - x^*(t)\| + \|u^{(k+1)}(t) - u^*(t)\| + \|\lambda^{(k+1)}(t) - \lambda^*(t)\| \\
&\le c(\sigma)e^{-\rho_Z(\sigma) \tau} \rbr{2\|x^{(k)}(t) - x^*(t)\|_\infty  + \|\lambda^{(k)}(t) - \lambda^*(t)\|_\infty}\\
& \leq 3c(\sigma)e^{-\rho_Z(\sigma) \tau} \cdot \omega^{(k)},
\end{align*}
which concludes that the convergence is linear.
\end{proof}

We see from \Cref{proposition:linear-convergence-in-l-infty} that the linear rate $3c(\sigma)e^{-\rho_Z(\sigma) \tau}< 1$ holds as long as
\begin{equation*}
\tau > \log(3c(\sigma))/\rho_Z(\sigma),
\end{equation*}
and the above threshold depends only on the controllability parameter $\sigma\ll T$ in Assumption \ref{assumption:uniform-controllability}, but is independent of the full horizon length $T$. This makes our algorithm particularly promising for long-horizon problems. Furthermore, the linear rate $3c(\sigma)e^{-\rho_Z(\sigma) \tau}$ improves exponentially with respect to the overlap size $\tau$. Overall, the results in \Cref{proposition:linear-convergence-in-l-infty} demonstrate the advantage of incorporating domain overlaps during optimization in \Cref{alg:mainprocedure}, as it accelerates convergence in the interior regions. The solutions in the overlapped regions can then be discarded while concatenating the remainder. In the numerical experiment section, we provide an example to support this insight.

Before concluding this section, we justify why EDS is a key property underlying the convergence of~\Cref{alg:mainprocedure}. At iteration $k+1$, the $j$-th subproblem is solved using boundary data obtained from the previous global iterate. Thus, the error in the new local solution is driven by the errors at the endpoints $t_j^0$ and $t_j^1$ from the previous iterate. The EDS estimates in \Cref{theorem:eds-result-for-middle} suggest that the effect of such endpoint perturbations decays exponentially as one moves into the interior of the subdomain. This boundary-error decay mechanism is the continuous-time analogue of the sensitivity mechanism underlying discrete-time Schwarz~methods for nonlinear OCPs and graph-structured nonlinear programming \cite{Na2020Exponential,Na2022Convergence,Shin2020Overlapping, Ni2024Distributed}.

The overlap then provides an important algorithmic design for converting the EDS property into a global contraction of the Schwarz scheme. Algorithm~\ref{alg:mainprocedure} aggregates only the solution on the non-overlapping interval $[t_{j-1},t_j]$ and discards the overlap regions. As a result, every retained point is separated from the extended boundaries by at least the overlap length $\tau$. Consequently, the error transferred from one Schwarz iterate to the next is reduced by a factor proportional to $\exp(-\rho_Z\tau)$. \Cref{theorem:stepwiseimprovement} formalizes this stagewise improvement by bounding the in-domain error on each core interval in terms of the boundary errors from the previous iterate, and \Cref{proposition:linear-convergence-in-l-infty} subsequently converts this estimate into a global linear convergence result for \Cref{alg:mainprocedure}.

\section{Numerical Experiment}  \label{sec:experiments} 

To illustrate the effectiveness of the proposed overlapping Schwarz scheme, we draw inspiration from a nonlinear optimal control problem with a final-time constraint studied in \cite{Schorlepp2023Scalable}, and consider a linearized version of it around the origin. We aim to demonstrate three promising properties of our Schwarz scheme (\Cref{alg:mainprocedure}) via the experiment:
\begin{enumerate}[label=\textbf{(\alph*)}, topsep=3pt]
\setlength\itemsep{0.2em}
\item The overlapping Schwarz scheme converges linearly, with the linear convergence rate improving exponentially fast in terms of the overlap size.
\item Higher-order time integrators can be easily incorporated when solving the forward and backward equations \eqref{subeqn:pmp-state} and \eqref{subeqn:pmp-adjoint} for the truncated subproblems.
\item In contrast to ``discretize-then-optimize" approaches that fix the time step in advance, our continuous-time Schwarz scheme supports adaptive time-stepping, allowing efficient resolution of stiff dynamics.$\quad$
\end{enumerate}

To this end, we consider the original nonlinear OCP in \cite{Schorlepp2023Scalable} formulated in $\mR^2$ as:
\begin{subequations} \label{eqn:rawnonlineartestproblem}
\begin{align}
\min_{u(\cdot), x(\cdot)} \quad & \frac{1}{2}\int_0^T \|x(t)\|^2 + \|u(t)\|^2 \, dt, \label{subeqn:objective} \\
\text{s.t.} \quad\;\; & \dot{x}(t) = b_{\xi}(x(t)) + N u(t), \quad t \in (0, T], \label{subeqn:dynamics} \\
& x(0) = \begin{bmatrix} 0; \; 0 \end{bmatrix}, \quad \phi(x(T)) = \theta, \label{subeqn:boundary}
\end{align}
\end{subequations}
where $\theta\in\mR$ is an adjustable parameter, $N\in \mR^{2\times 2}$ is the control coefficient matrix, and the state dynamics map $b_\xi : \mathbb{R}^2 \to \mathbb{R}^2$ and the terminal constraint map $\phi : \mathbb{R}^2 \to \mathbb{R}$ are defined by
\begin{equation*}
b_\xi(x) \coloneqq \begin{bmatrix}
-x_1 - x_1 x_2 \\
-\xi x_2 + x_1^2
\end{bmatrix}, 
\quad \quad \quad 
\phi(x) \coloneqq x_1 + 2x_2.
\end{equation*}
Here, $\xi > 0$ is a stiffness parameter that characterizes the relative decay rate of the two state components, and the terminal constraint imposes a scalar condition on the final state. We linearize the nonlinear dynamics \eqref{subeqn:dynamics} at the equilibrium point $x = 0$ and incorporate the terminal constraint \eqref{subeqn:boundary} into the objective via a soft penalty, resulting in the following linear-quadratic OCP:
\begin{subequations} \label{eqn:fullproblem}
\begin{align}
\min_{u(\cdot), x(\cdot)} \quad & \frac{1}{2}\int_0^T \|x(t)\|^2 + \|u(t)\|^2 \, dt + \frac{\alpha}{2} \left(\phi(x(T)) - \theta\right)^2, \label{subeqn:objective_full} \\
\text{s.t.} \quad\;\; &\dot{x}(t) = M_{\xi}x(t) + N u(t), \quad t \in (0,T], \label{subeqn:dynamics_full} \\
&x(0) = 0, \label{subeqn:boundary_full}
\end{align}
\end{subequations}
where \( \alpha > 0 \) is the penalty parameter (chosen sufficiently large), and $M_\xi\in \mR^{2\times 2}$ is the Jacobian of $b_\xi$ at origin given by $M_{\xi} := \nabla b(0;\xi) = \text{diag}(-1, -\xi)$.

In our experiment, we first follow \cite{Schorlepp2023Scalable} and set $N = \text{diag}(1,0.25)$, $T=5$, $\theta=3$, $\xi=4$, and $\alpha=10^2$. We will also vary the parameters by increasing $\xi$ when testing the Schwarz method on stiff dynamics. Throughout the experiment, the true reference solution $(x^*, u^*, \lambda^*)$ is computed by applying the forward Euler discretization to the full problem \eqref{eqn:fullproblem} on a fine uniform grid with the time step $\Delta t = 10^{-3}$, and solving the resulting discrete problem using MATLAB’s built-in function \texttt{fmincon}.
For the overlapping Schwarz method, we partition the time interval $[0, T]$ into \( m = 3 \) uniform subdomains, each corresponding to a subproblem as defined in \eqref{eqn:defineparameterizedsubproblem}.
We express the overlap parameters $\tau_j^0 = \tau_j^1 = \tau$ as a fraction of the subdomain length, and will vary different $\tau$ when we illustrate the improvement of the linear rate in $\tau$ (i.e., the goal \textbf{(a)}). 
Each subproblem is solved in parallel using \Cref{alg:gradient-descent-subinterval}, with the step size \( \eta = 10^{-2} \) and the initial control sampled as $u(t) \sim \mathcal{N}(0, I)$. Note that Algorithm \ref{alg:gradient-descent-subinterval} (Lines 3 and 4) requires integrating the state and adjoint equations forward and backward, respectively (cf. \eqref{subeqn:pmp-state}, \eqref{subeqn:pmp-adjoint}). 
We will test the \mbox{robustness} of the Schwarz method to \mbox{different} time \mbox{integrators} (i.e., the goal \textbf{(b)}), as well as the benefits of enabling adaptive time-stepping when the dynamics is stiff (i.e., the goal \textbf{(c)}). We terminate Algorithm \ref{alg:gradient-descent-subinterval} when the norm of the local objective gradient falls below $10^{-6}$. Lastly, let $(x^{(k)}, u^{(k)}, \lambda^{(k)})$ denote the $k$-th Schwarz iterate; the convergence is measured by the maximum pointwise deviation from the reference solution:
 \begin{equation*}  \label{eqn:error-metric}
e^{(k)} \coloneqq \max_{0 \le j \le \frac{T}{\Delta t}} \left(\|x^{(k)}(t_j) - x^*(t_j)\| + \|u^{(k)}(t_j) - u^*(t_j)\| + \|\lambda^{(k)}(t_j) - \lambda^*(t_j)\|\right) \quad\;\; \text{with } \;\; t_j = j\cdot\Delta t.
\end{equation*}

\begin{figure}[t]
\centering
\includegraphics[width=6cm]{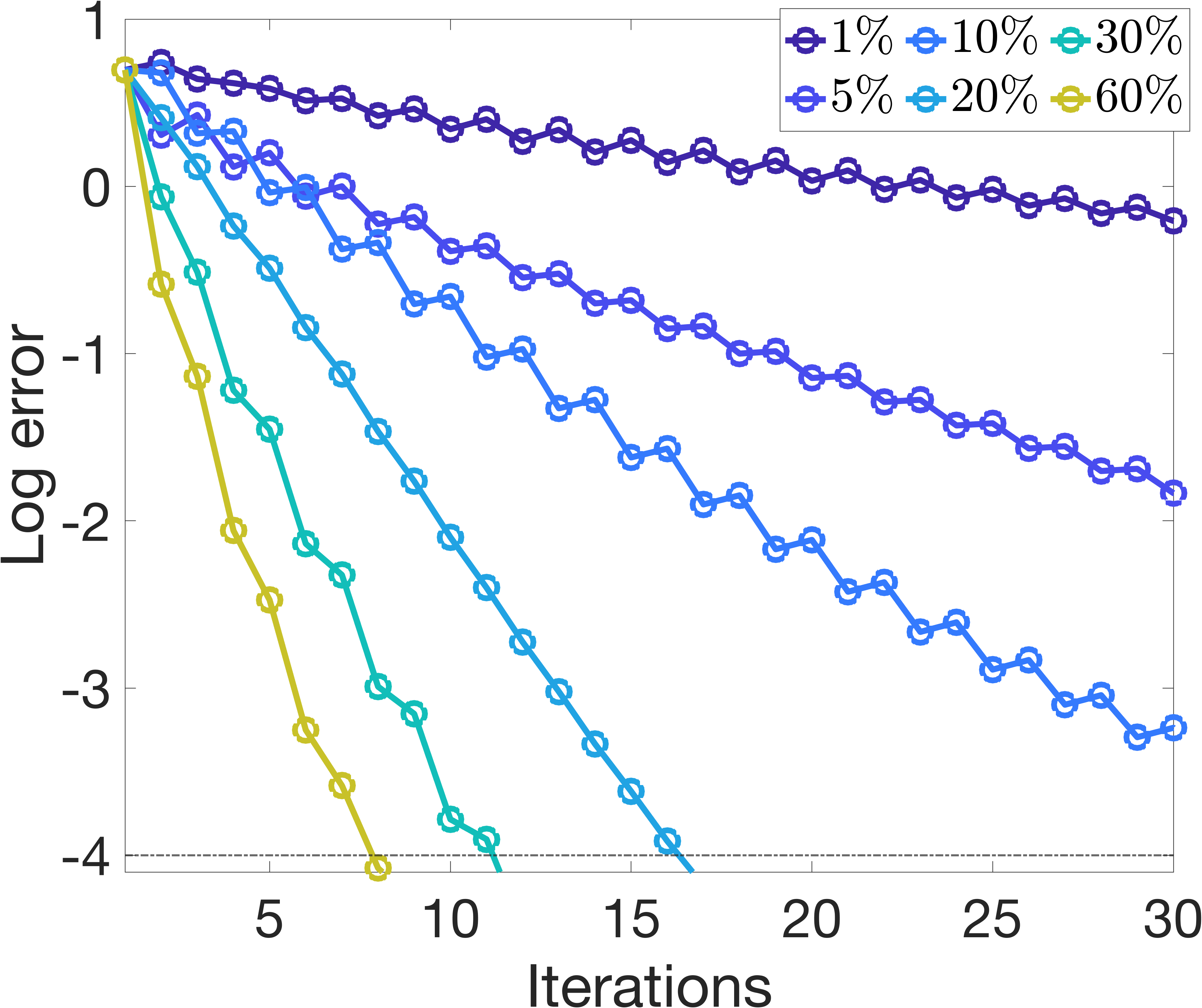}
\hspace{0.5cm}
\includegraphics[width=6cm]{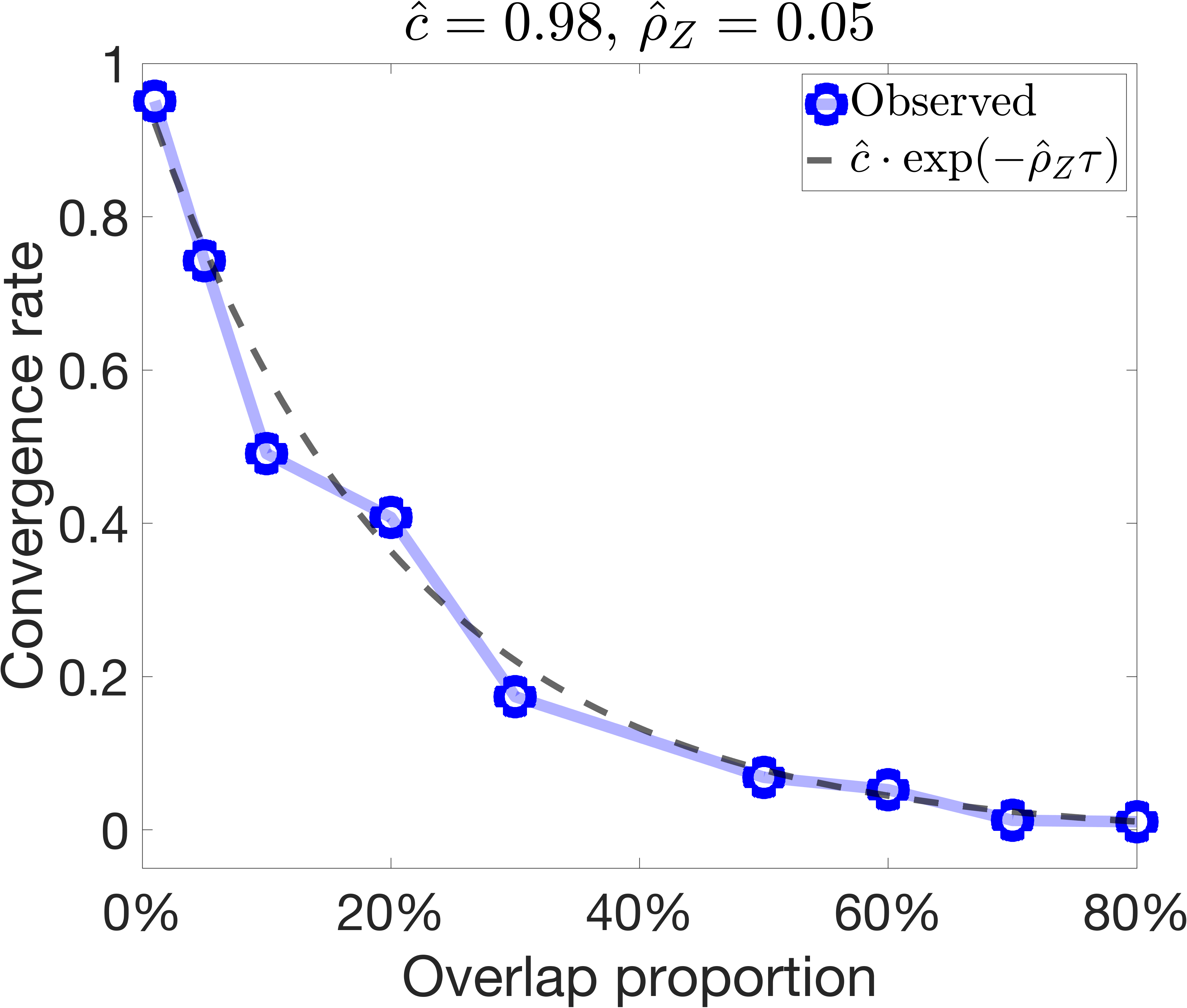}
\caption{(Left) Convergence of the overlapping Schwarz method under varying overlap sizes (\( \{1\%, 5\%, 10\%, 20\%, 30\%, 60\%\} \) of the subdomain length). Larger overlaps yield faster convergence. 
(Right) Observed convergence rate and its exponential fit. For each overlap size, we estimate the convergence rate by averaging the ratios of two consecutive errors, and fit an exponential curve of the theoretical form $c\exp(-\rho_Z \tau)$. The fitted values are $\hat{c} = 0.98$ and $\hat{\rho}_Z = 0.05$; and we have $\rho_Z<0.02$ at the significance level of  5\%. We observe that the convergence rate improves exponentially in the overlap size/proportion, which is consistent~with~\Cref{proposition:linear-convergence-in-l-infty}.}
\label{fig:overlappingschwarzconvergence}
\end{figure}

\noindent$\bullet$ \textbf{Goal (a)}. 
We verify \Cref{proposition:linear-convergence-in-l-infty} by illustrating the convergence behavior of the overlapping Schwarz method. We vary the \mbox{overlap} \mbox{proportion} \mbox{$\tau\in\{1\%, 5\%, 10\%, 20\%, 30\%, 60\%\}$} \mbox{relative} to the \mbox{subdomain} length, and the forward and backward equations of the subproblems (cf. Algorithm \ref{alg:gradient-descent-subinterval}, Lines 3 and 4) are simply integrated using the forward Euler method.

The convergence behavior is illustrated in \Cref{fig:overlappingschwarzconvergence} (left). From the figure, we observe that the errors decay linearly over the Schwarz iterations, with larger overlaps yielding faster convergence, which indicates that the linear rate improves as the overlap size increases.
To examine the linear rate more closely, we estimate it by averaging the ratios of two consecutive errors for each $\tau$ and plot the resulting rates versus the overlap proportion in Figure \ref{fig:overlappingschwarzconvergence} (right). 
We find that the observed convergence rate aligns with a fitted model of the form \( \hat{c}_Z \exp(-\hat{\rho}_Z \tau) \), suggesting that the rate improves precisely exponentially with the overlap size. 
This further supports the theoretical insight of the exponential improvement of the rate with the overlap size in \Cref{proposition:linear-convergence-in-l-infty}. Overall, the empirical results validate our analytic rate bound and highlight the reasoning of using large overlaps to accelerate convergence.

\vskip5pt
\noindent$\bullet$ \textbf{Goal (b).} 
We now investigate the effect of applying different time integration schemes to the subproblems on the convergence behavior of the overlapping Schwarz method.
In this experiment, we fix the overlap size at $5\%$ of the subdomain length. To clearly isolate the effect of the solver choice on convergence, we employ a coarser temporal discretization when applying the integration methods, with a step size of $\Delta t = 0.05$ (recall that the reference solution is computed on a much finer grid with $\Delta t=10^{-3}$).

Figure \ref{fig:compare-numerical-schemes} (left) compares the explicit forward Euler (FE), implicit backward Euler (BE), and explicit Runge-Kutta (RK45) methods for solving the subproblem dynamics \eqref{subeqn:pmp-state} and \eqref{subeqn:pmp-adjoint}. 
We observe that higher-order solvers attain improved solution accuracy, roughly half an additional order of accuracy per upgrade in method fidelity. While higher-order solvers are indeed more computationally expensive, they are applied only to the subproblems within the Schwarz scheme.
Furthermore, despite the differences in numerical integration error, the overall convergence with respect to Schwarz iterations remains linear, consistent with \Cref{proposition:linear-convergence-in-l-infty}.

\vskip5pt
\noindent$\bullet$ \textbf{Goal (c).} A key advantage of our continuous-time formulation is its compatibility with a wide range of time-stepping schemes. In contrast, high-order solvers such as RK45 are not easily integrated into discretize-then-optimize methods \cite{Shin2020Overlapping, Na2022Convergence}. Our approach enables adaptive and high-fidelity solvers to be used without compromising the structure of the optimality system.

To demonstrate the benefits of adaptive time-stepping, we compare the behavior of solvers in a stiff setting with $\xi=15$ and $T=\theta=10$, where the forward Euler method becomes unstable. 
In this experiment, we compare the fixed-step FE method with the adaptive solver \texttt{ode23} \cite{Shampine1997MATLAB}, applied to the Hamiltonian systems of the local subproblems.
As expected, we see from Figure \ref{fig:compare-numerical-schemes} (right) that forward Euler fails to converge unless the time step is reduced significantly. In contrast, \texttt{ode23} maintains stability and consistently achieves lower error across Schwarz iterations, demonstrating the robustness of our method when paired with adaptive~integration.

\begin{figure}[t]
\centering
\includegraphics[width=6cm]{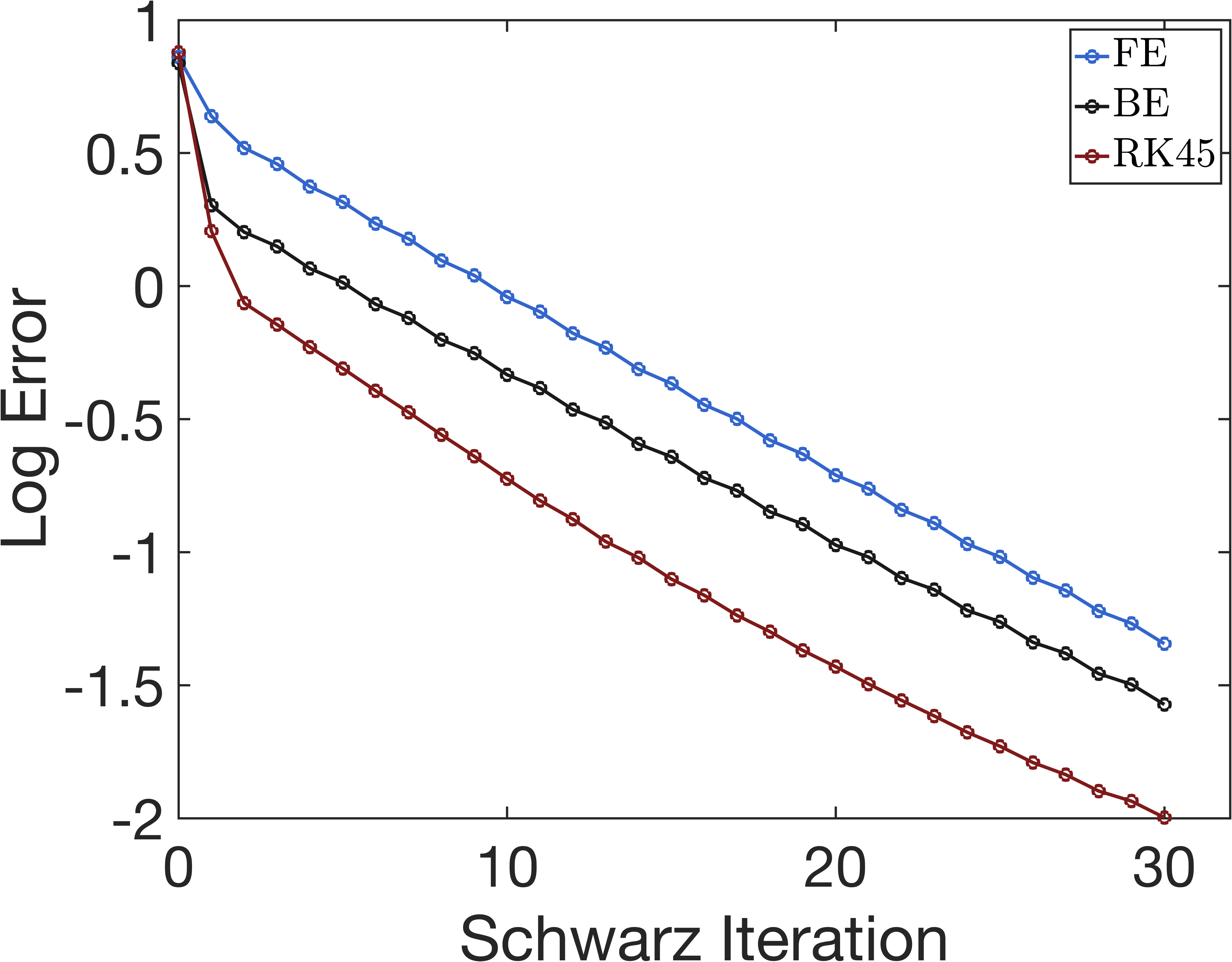}
\hspace{0.5cm}
\includegraphics[width=5.8cm]{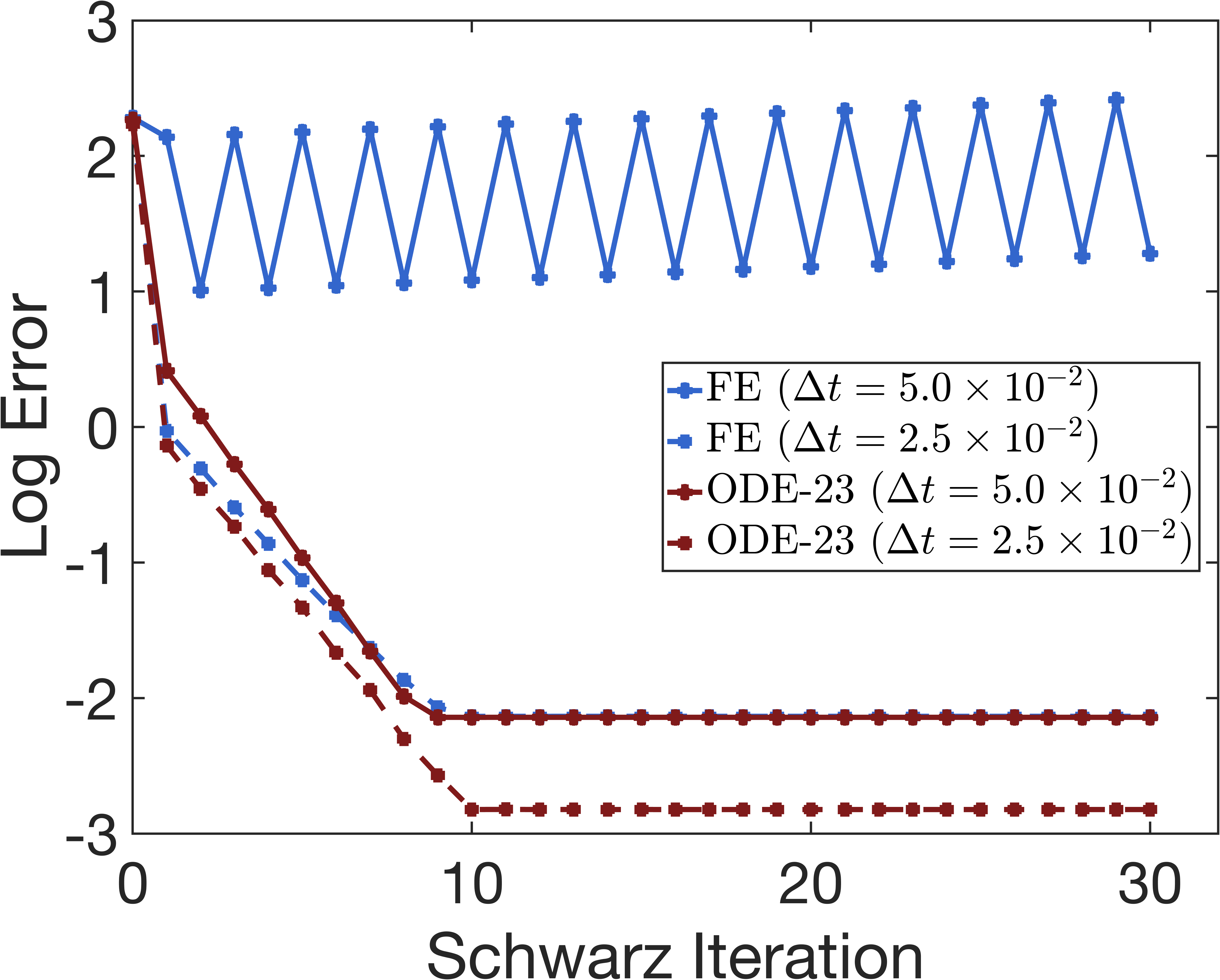}
\caption{(Left) Comparison of numerical solvers in resolving the subproblem Hamiltonian dynamics. The Runge-Kutta method (RK45) yields the highest accuracy per iteration. 
(Right) Comparison of first 30 Schwarz iterations for different solvers and time step sizes \( \Delta t \). Adaptive solvers maintain low errors even on coarse grids, while FE accuracy degrades significantly in the stiff~regime.}\label{fig:compare-numerical-schemes}
\end{figure}

\section{Conclusion and Future Work}\label{sec:conclusions} 

In this work, we introduced a continuous-time overlapping Schwarz decomposition method for linear-quadratic optimal control problems and established its convergence properties. The method decomposes the global control problem into smaller subproblems defined on overlapping temporal subdomains, with appropriate boundary conditions specified at the interfaces to enforce consistency between adjacent subdomains. Our analysis reveals that the method leverages the intrinsic exponential decay of sensitivity in continuous-time optimal control problems, such that the impact of imperfect boundary conditions on the solution trajectory decays exponentially as one moves away from the subdomain boundaries. As a result, the Schwarz method achieves linear convergence by recursively updating the boundary conditions. Notably, the convergence rate improves exponentially with the size of the overlap. The continuous-time formulation offers several advantages, including compatibility with higher-order and adaptive time-stepping schemes, as well as improved handling of stiff system dynamics. Numerical experiments corroborate our theoretical findings and demonstrate the practical effectiveness of the method on a stiff linear-quadratic optimal control problem.

Although the present analysis is carried out in the finite-horizon setting, the constants appearing in the EDS estimates are chosen independently of the terminal time \(T\) under the stated uniform controllability and coercivity assumptions, and therefore remain applicable in long-horizon regimes. A genuine infinite-horizon extension, however, would require additional analysis of the limiting Riccati equation, transversality conditions at infinity, and a rigorous definition of Schwarz boundary updates on an unbounded time interval. Time-transformation approaches for infinite-horizon optimal control and model predictive control, such as~\cite{Dinh2025Nonlinear}, provide a promising direction. Nevertheless, the resulting nonuniform weights and transformed notion of overlap would require a separate analysis.

There are also several other promising extensions can be pursued. First, we believe it is possible to extend the overlapping Schwarz decomposition from linear-quadratic program to the nonlinear case, as a continuous-time generalization of \cite{Na2022Convergence, Shin2020Overlapping}. However, doing so would require a functional analog of the convexification procedure \cite{Verschueren2017Sparsity}, which we leave for future exploration.  
Analogously, another important extension of the optimize-then-discretize Schwarz analysis is to inequality-constrained problems (or nonsmooth objectives when box constraints are imposed through soft penalties). In direct discretize-then-optimize formulations, such constraints can be handled using general nonlinear programming techniques, including active-set and interior-point methods~\cite{Gill2011Sequential,Hintermueller2002Primal,Ito2004Primal,Troeltzsch2010Optimal}. That being said, many local constrained algorithms reduce, after linearization and active-set identification, to equality-constrained linear-quadratic subproblems. This suggests that the EDS mechanism established here can still serve as a local building block for more general constrained Schwarz schemes, analogous to existing discretize-then-optimize theories and Schwarz methods for constrained quadratic programs~\cite{Shin2020Overlapping}.
Second, in this work we have adopted a gradient-based optimization approach in solving continuous-time subproblems (Algorithm \ref{alg:gradient-descent-subinterval}). The local convergence of the subproblems can be further enhanced using second-order optimization solvers \cite{Hinze2001Second}, which can be implemented by solving second-order adjoint equations. Other practical directions include developing robust solvers compatible with the Schwarz scheme. Possible extensions include directly solving matrix Riccati equations \cite{Schorlepp2023Scalable} or applying symplectic Runge-Kutta methods \cite{SanzSerna2016Symplectic}.
Finally, there are opportunities to apply the Schwarz method in deep learning, particularly in settings where models can be framed as locally \mbox{linear-quadratic} optimal control problems of the form \eqref{eqn:linear-quadratic-ocp}. The established convergence properties of the Schwarz method serve as a foundation for practical applications of broad control methods inspired by \mbox{turnpike} analysis \cite{Gruene2021Abstract}, which aligns well with deep learning architectures that can be interpreted as discrete dynamical systems, such as neural ODEs \cite{Geshkovski2022Turnpike}. Future work can explore these intersections, leading to model- and data-parallel distributed~learning algorithms.

\appendix
\SingleSpacedXI

\section{Gradient Method for Schwarz Subproblem}  \label{appendix:deriveadjointgradient}

To solve the subproblems $\P_j([t_j^0, t_j^1]; p_j,q_j)$, we can apply gradient-based optimization method in infinite-dimensional space summarized in \Cref{alg:gradient-descent-subinterval}. 

\begin{algorithm}[!b]
\caption{Gradient Descent for Subproblem $\P_j([t_j^0, t_j^1]; p_j,q_j)$}\label{alg:gradient-descent-subinterval}
\begin{algorithmic}[1]
\State \textbf{Input:} Subproblem control $u_j^{(0)}$; initial and terminal state parameters $(p_j, q_j)$; step size $\eta > 0$.
\For{$\ell = 0, 1, \ldots$}
\State Integrate state equation forward from $t_j^0$ to $t_j^1$ to obtain $x_j^{(\ell)}$ (see \eqref{eqn:mild-solution});
\State Integrate adjoint equation backward from $t_j^1$ to $t_j^0$ to obtain $\lambda_j^{(\ell)}$;
\State Compute the (total) gradient of the local objective:
\begin{equation*}
\left(\frac{\delta \widetilde{\mathcal{J}}_j}{\delta u_j}\right)^{(\ell)}(t) \longleftarrow H(t)x_j^{(\ell)}(t) + R(t)u_j^{(\ell)}(t) + B^\top(t) \lambda_j^{(\ell)}(t),\quad\quad t\in[t_j^0, t_j^1];
\end{equation*}
\State Update control: $u_j^{(\ell+1)} \gets u_j^{(\ell)} - \eta \cdot \left(\frac{\delta \widetilde{\mathcal{J}}_j}{\delta u_j}\right)^{(\ell)}$;
\EndFor
\end{algorithmic}
\end{algorithm}

We should mention that Algorithm~\ref{alg:gradient-descent-subinterval} is included for completeness and to specify the implementation used in the numerical experiments; the gradient computation itself is standard. It is a natural choice for the continuous-time subproblems because it does not introduce additional algorithmic complexity.
That being said, our continuous-time formulation is not restricted to this basic gradient descent implementation. More advanced gradient-based and quasi-Newton methods can be derived directly in function space using functional derivatives and adjoint equations before selecting a numerical time discretization. This flexibility is one of the key advantages of the optimize-then-discretize viewpoint, and the development of higher-order local solvers within the present Schwarz framework is a natural direction for further enhancing Algorithm~\ref{alg:gradient-descent-subinterval}.

For a given open-loop control, the state equation is integrated forward, the adjoint equation is integrated backward, and the control is updated using the total functional gradient. In this sense, the local subproblem solver is analogous to a single-shooting adjoint method, but this structure is used only within each Schwarz subproblem. The outer method remains an overlapping Schwarz iteration, where boundary information is exchanged between neighboring subdomains and the next global iterate is assembled by retaining only the non-overlapping intervals.

We provide a few clarifications regarding the derivation of the gradient (i.e., descent direction). In particular, the total gradient of the subproblem cost functional with respect to the control $u_j$ is derived using first-order perturbation analysis \citep{Plessix2006review}. The first-order optimality condition given by the PMP is numerically solved in an ``open-loop'' fashion: a control $u_j$ is input, the states $x_j$ and adjoint states $\lambda_j$ are first subsequently solved, and the control is then updated using information of the states and adjoint states in a direction that minimizes the cost functional. In this formulation, we account for the fact that the states $x_j = x_j[u_j; p_j]$ are fully determined by both the control $u_j$ and the initial boundary parameter $p_j$ through the solution \eqref{eqn:mild-solution}.

We begin by reviewing the subproblem \eqref{eqn:defineparameterizedsubproblem} as a cost functional in $u_j$ only:
\begin{subequations} \label{eqn:general-ode-ocp-appendix}
\begin{align}
\min_{u_j(\cdot), x_j(\cdot)} \quad &  \widetilde{\mathcal{J}}_j[u_j; p_j, q_j] =\mathcal{J}_j[u_j, x_j[u_j; p_j]; q_j], \label{subeqn:objective_general} \\
\text{s.t.} \quad\;\;\; & \dot{x}_j(t) = A(t)x_j(t)+B(t)u_j(t), \quad t \in (t_j^0, t_j^1], \label{subeqn:dynamics_general} \\
&x_j(t_j^0) = p_j. \label{subeqn:boundary_general}
\end{align}
\end{subequations}
Let $u_j$ denote any control function for \eqref{eqn:general-ode-ocp-appendix}. Consider a perturbation at $u_j$ in the direction of a function $\delta u_j$. The total variation (i.e. upon accounting for the variations of $x_j$ as a result of variations in $u_j$) of the cost functional can be decomposed as the following:
\begin{align}\label{eqn:totalderivative-du}
\left\langle\frac{\delta \widetilde{\mathcal{J}}_j[u_j;p_j,q_j]}{\delta u_j}, \delta u_j\right\rangle & \coloneqq \lim_{\epsilon\rightarrow 0}\frac{\widetilde{\mathcal{J}}_j[u_j+\epsilon\cdot\delta u_j; p_j, q_j] - \widetilde{\mathcal{J}}_j[u_j;p_j,q_j]}{\epsilon} \nonumber\\
& = \lim_{\epsilon\rightarrow 0}\frac{\mathcal{J}_j[u_j+\epsilon\cdot\delta u_j, x_j[u_j+\epsilon\cdot\delta u_j; p_j];q_j] - \mathcal{J}_j[u_j; x_j[u_j; p_j];q_j]}{\epsilon} \nonumber\\
& = \left\langle \frac{\partial \mathcal{J}_j[u_j, x_j; q_j]}{\partial u_j}, \delta u_j \right\rangle + \left\langle \frac{\partial \mathcal{J}_j[u_j, x_j; q_j]}{\partial x_j}, \delta x_j \right\rangle,
\end{align}
where we define:
\begin{equation*}
\delta x_j \coloneqq \lim_{\epsilon\rightarrow 0}\frac{x_j[u_j+\epsilon\cdot\delta u_j; p_j] - x_j[u_j;p_j]}{\epsilon}.
\end{equation*} 
To obtain the formula of $\delta \widetilde{\mathcal{J}}_j/\delta u_j$, we would like to convert the term involving $\delta x_j$ into an expression involving $\delta u_j$ by exploiting the structure of the PMP \eqref{equ:PMP}. We first derive a differential equation that $\delta x_j$ satisfies. Let $x_j^{\epsilon} := x_j[u_j+\epsilon\cdot\delta u_j; p_j]$. Then, $x_j^{\epsilon}$ is the solution to the following system on $[t_j^0, t_j^1]$:
\begin{equation*}
\dot{x}_j^{\epsilon}(t) = A(t)x_j^{\epsilon}(t) + B(t)(u_j(t)+\epsilon\cdot\delta u_j(t)), \quad t\in(t_j^0, t_j^1],\quad\quad x_j^{\epsilon}(t_j^0) = p_j,
\end{equation*} 
which is compared to the original system on $[t_j^0, t_j^1]$:
\begin{equation*}
\dot{x}_j(t) = A(t)x_j(t) + B(t)u_j(t), \quad t\in(t_j^0, t_j^1],\quad\quad x_j(t_j^0) = p_j.
\end{equation*}
Subtracting the above equations, dividing by $\epsilon$, and taking the limit as $\epsilon\rightarrow 0$, we have
\begin{subequations} \label{eqn:perturbedode}
\begin{align}
\delta \dot{x}_j(t) &= A(t) \delta x_j(t) + B(t)\delta u_j(t), \quad t\in (t_j^0, t_j^1], \label{subeqn:perturb-dynamics}\\
\delta x_j(t_j^0) &= 0. \label{subeqn:initial-perturb-zero}
\end{align}
\end{subequations}
On the other hand, let $\lambda_j(t)$ denote the adjoint states satisfying the differential equation (cf. \Cref{theorem:pontryagin}):
\begin{subequations} \label{eqn:adjoint-dynamics}
\begin{align}
\dot{\lambda}_j(t) &= -A^\top (t)\lambda_j(t) - \nabla_x \left(
\frac12\begin{bmatrix}
x_j(t) \\ u_j(t)
\end{bmatrix}^\top\begin{bmatrix}
Q(t) & H^\top(t) \\
H(t) & R(t)
\end{bmatrix}\begin{bmatrix}
x_j(t) \\ u_j(t)
\end{bmatrix}\right) \nonumber\\
&= -A^{\top}(t)\lambda_j(t) - Q(t)x_j(t) - H^{\top}(t)u_j(t), \quad\quad t\in [t_j^0, t_j^1),  \label{subeqn:dynamics-of-adjoint} \\
\lambda_j(t_{j}^1) &= \nabla_xL_j(x_j(t_j^1); q_j), \label{subeqn:adjoint-terminal-condition}
\end{align}
\end{subequations}
where $L_j$ is as defined in \eqref{eqn:terminalcostparameterized}. Then we take the inner product between the functions $\lambda_j$ and $\delta \dot{x}_j$, and apply integration by parts to obtain
\begin{equation}  \label{eqn:adjoint-derivative-product1}
\left\langle \lambda_j, \delta \dot{x}_j \right\rangle
= \int_{t_j^0}^{t_j^1}\lambda_j^{\top}(t)\delta \dot{x}_j(t) dt = \lambda_j^{\top}(t_j^1)\delta x_j(t_j^1) - \underbrace{\cancel{\lambda_j^{\top}(t_j^0)\delta x_j(t_j^0)}}_{=\; 0, \; \eqref{subeqn:initial-perturb-zero}} - \left\langle \dot{\lambda}_j, \delta x_j \right\rangle.
\end{equation}
Furthermore, we have
\begin{align}\label{eqn:adjoint-derivative-product2}
\left\langle  \lambda_j, \delta \dot{x}_j \right\rangle &\; \stackrel{\mathclap{\eqref{subeqn:perturb-dynamics}}}{=} \;
\int_{t_j^0}^{t_j^1}\lambda_j^{\top}(t)(A(t)\delta x_j(t)+B(t)\delta u_j(t))dt  \nonumber\\
&\;=\; \int_{t_j^0}^{t_j^1}(A^{\top}(t)\lambda_j(t))^{\top}\delta x_j(t)dt + \int_{t_j^0}^{t_j^1}(B^{\top}(t)\lambda_j(t))^{\top}\delta u_j(t)dt.
\end{align}
Combining \eqref{eqn:adjoint-derivative-product1} and \eqref{eqn:adjoint-derivative-product2}, we obtain
\begin{equation}\label{eqn:adjoint-relation-with-perturbation-u}
\lambda_j^{\top}(t_j^1)\delta x_j(t_j^1) - \int_{t_j^0}^{t_j^1}
\left(\dot{\lambda}_j(t)+{A}^{\top}(t)\lambda_j(t)\right)^{\top} \delta x_j(t)dt = \int_{t_j^0}^{t_j^1}
\left( B^{\top}(t)\lambda_j(t) \right)^{\top}\delta u_j(t)dt.
\end{equation}
Now, we revisit our main objective to evaluate \eqref{eqn:totalderivative-du}. We have
\begin{align*}
\left\langle 
\frac{\partial\mathcal{J}_j[u_j, x_j; q_j]}{\partial u_j}, \delta u_j \right\rangle & \coloneqq \lim_{\epsilon\rightarrow 0}\frac{\mathcal{J}_j[u_j+\epsilon\cdot\delta u_j, x_j[u_j]; p_j] - \mathcal{J}_j[u_j;x_j[u_j]; p_j]}{\epsilon} \\
& = \int_{t_j^0}^{t_j^1}(H(t)x_j(t)+R(t)u_j(t))^{\top}\delta u_j(t)dt \\
\left\langle \frac{\partial \mathcal{J}_j[u_j, x_j; q_j]}{\partial x_j}, \delta x_j \right\rangle & \coloneqq
\lim_{\epsilon\rightarrow 0}\frac{\mathcal{J}_j[u_j, x_j[u_j]+\epsilon\cdot\delta x_j; p_j] - \mathcal{J}_j[u_j;x_j[u_j]; p_j]}{\epsilon} \\
& = \int_{t_j^0}^{t_j^1} (Q(t)x_j(t)+H^{\top}(t)u_j(t))^{\top}\delta x_j(t)dt + \nabla_xL_j^{\top}(x_j(t_j^1);q_j)\delta x_j(t_j^1) \\
& \stackrel{\mathclap{\eqref{eqn:adjoint-dynamics}}}{=} 
-\int_{t_j^0}^{t_j^1}(\dot{\lambda}_j(t) + A^{\top}(t)\lambda_j(t))^{\top}\delta x_j(t)dt + \lambda_j^{\top}(t_j^1)\delta x_j(t_j^1) \\ &\stackrel{\mathclap{\eqref{eqn:adjoint-relation-with-perturbation-u}}}{=} \int_{t_j^0}^{t_j^1}
\left( B^{\top}(t)\lambda_j(t)\right)^{\top}\delta u_j(t)dt.
\end{align*}
Plugging the above display back to \eqref{eqn:totalderivative-du}, we know the total variation with respect to control $u_j$ can be evaluated in closed-form as:
\begin{equation*}
\left\langle  \frac{\delta \widetilde{\mathcal{J}}_j[u_j;p_j,q_j]}{\delta u_j}, \delta u_j \right\rangle = \int_{t_j^0}^{t_j^1}\left( H(t)x_j(t)+R(t)u_j(t)+B^{\top}(t)\lambda_j(t) \right)^{\top}\delta u_j(t)dt.
\end{equation*}
Therefore, the gradient is given by:
\begin{equation}  \label{eqn:functional-gradient}
\left(\frac{\delta \widetilde{\mathcal{J}}_j[u_j;p_j,q_j]}{\delta u_j}\right)(t) = H(t)x_j(t) + R(t)u_j(t) + B^{\top}(t)\lambda_j(t), \quad\quad t\in [t_j^0, t_j^1].
\end{equation}
We note that \eqref{eqn:functional-gradient} can be computed efficiently using a forward-backward-in-time procedure (cf. \Cref{alg:gradient-descent-subinterval}) for numerical integration, which supports adaptive time-stepping and is compatible with numerous standard solvers. Other gradient-based algorithms, such as steepest descent, conjugate gradient, or BFGS in infinite-dimensional space can be similarly defined through variational analysis
\citep{Lasdon1967conjugate, Neuberger2009Continuous}.

\section{Proof of \Cref{theorem:state-closed-loop-equation}}\label{appendix:closed-loop-sensitivity-ocp}

The proof proceeds by directly applying Hamilton-Jacobi-Bellman (HJB) equations (cf. \Cref{theorem:hjbequation}) to the parameterized optimal control problem \eqref{eqn:parameterized-linear-quadratic-program}. We define the Hamiltonian function:
\begin{equation*}
\mathcal{H}(t, x, u, \lambda; d) \coloneqq \frac12\begin{bmatrix}
x\\ u\\ d
\end{bmatrix}^\top\begin{bmatrix}
Q(t) & H^\top(t) & G^\top(t) \\
H(t) & R(t) & W^\top(t) \\
G(t) & W(t) & 0
\end{bmatrix}\begin{bmatrix}
x\\ u\\ d
\end{bmatrix} + \lambda^\top(A(t)x + B(t)u + C(t)d).
\end{equation*}
Since the terminal cost in \eqref{eqn:parameterized-linear-quadratic-program} is quadratic in $x(T)$, we propose a quadratic ansatz for the value function:
\begin{equation}\label{eqn:quadratic-ansatz}
J^*(t, x) := \frac12x^\top S(t) x + x^\top v(t) + \frac12z(t),
\end{equation}
where the symmetric matrix \( S: [0,T] \to \mathbb{R}^{n_x \times n_x} \) and the vectors \( v: [0,T] \to \mathbb{R}^{n_x} \) and \( z: [0,T] \to \mathbb{R} \) are to be determined. We know that the gradient of the value function with respect to the state yields the optimal adjoint trajectory \cite[(4)]{Bokanowski2021Relationship}
\begin{equation}\label{nequ:4}
\lambda^*(t) = \nabla_x J^*(t, x) = S(t)x + v(t).
\end{equation}
Substituting the above expression into the Hamiltonian, we obtain
\begin{multline*}
\mathcal{H}(t, x, u, \nabla_x J^*(t, x); d) = \frac12x^\top Q(t) x +  x^\top H^\top(t) u +  x^\top G^\top(t) d  + \frac12u^\top R(t) u + u^\top W^\top(t) d \\
+ (S(t)x + v(t))^\top (A(t)x + B(t)u + C(t)d).
\end{multline*}
We now proceed to compute the minimization in \eqref{subeqn:minimization-pde} with respect to $u$ for each $t\in [0,T]$, derive the associated Riccati and adjoint equations, and validate the resulting closed-loop dynamics as stated in Theorem \ref{theorem:state-closed-loop-equation}. From Assumption \ref{assumption:matrices-are-bounded-below}, the Hamiltonian is strongly convex in $u$, so the minimizer exists and is unique for every $t \in [0, T]$. Taking the gradient of the Hamiltonian with respect to $u$ and setting it to $0$ gives the following condition that the optimal control $u^*$ must satisfy:
\begin{equation*}
\nabla_u \mathcal{H}(t, x^*, u^*, \nabla_x J^*(t, x^*); d) = H(t) x^*(t) + R(t) u^*(t) + W^\top(t) d(t) + B^\top(t) S(t) x^*(t) + B^\top(t) v(t) = 0.
\end{equation*}
Upon solving this equation, we obtain
\begin{equation} \label{eqn:explicitoptimalcontrol}
u^*(t) = -R^{-1}(t) \cbr{ H(t) x^*(t) + B^\top(t) S(t) x^*(t) + B^\top(t) v(t) + W^\top(t) d(t) }.
\end{equation}
Substituting \eqref{eqn:explicitoptimalcontrol} into the Hamiltonian, we obtain the minimal value. In particular, we let $\gamma(t) := H(t)x^* + B^\top(t)S(t)x^* + B^\top(t)v(t) + W^\top(t)d(t)$ and have for every $t\in[0, T]$ that
\begin{align*}
\mathcal{H}(t, x^*, u^*, \nabla_x J^*; d) & = \frac12(x^*)^\top Q(t) x^* - (x^*)^\top H^\top(t) R^{-1}(t) \gamma(t) + (x^*)^\top G^\top(t) d(t) + \frac12\gamma^\top(t) R^{-1}(t) \gamma(t)\\
& \quad	- \gamma^\top(t) R^{-1}(t) W^\top(t) d(t) + (x^*)^\top S(t) A(t) x^* - (x^*)^\top S(t) B(t) R^{-1}(t) \gamma(t) \\
& \quad + (x^*)^\top S(t) C(t) d(t) + (x^*)^\top A^\top(t) v(t)  - \gamma^\top(t) R^{-1}(t) B^\top(t) v(t)	+ d^\top(t) C^\top(t) v(t).
\end{align*}
To clearly isolate the above terms and match them with those defined in \eqref{eqn:quadratic-ansatz}, we let
\begin{equation*}
\mathcal{H}(t,x^*,u^*,\nabla_xJ^*;d) = \mathcal{Q}(t,x^*;d)+\mathcal{L}(t, x^*;d)+\mathcal{C}(t,x^*;d),
\end{equation*} 
where \( \mathcal{Q}, \mathcal{L}, \mathcal{C} \) respectively denote terms that are quadratic, linear, and constant in optimal states \( x^* \), given~by
\begin{align*}
\mathcal{Q}(t, x^*; d) & = (x^*)^\top \Big[\frac12Q(t) - H^\top(t) R^{-1}(t) \cbr{H(t) + B^\top(t) S(t)} + \frac12\cbr{H(t) + B^\top(t) S(t)}^\top R^{-1}(t)\cbr{H(t) + B^\top(t) S(t)} \\
& \quad + S(t) A(t) - S(t) B(t) R^{-1}(t) \cbr{H(t) + B^\top(t) S(t)} \Big] x^*, \\
\mathcal{L}(t, x^*; d) & = (x^*)^\top \Big[
-H^\top(t) R^{-1}(t) \cbr{B^\top(t) v(t) + W^\top(t) d(t)}
+ G^\top(t) d(t) - \cbr{H(t) + B^\top(t) S(t)}^\top R^{-1}(t) W^\top(t) d(t) \\
&\quad  + \cbr{H(t) + B^\top(t) S(t)}^\top R^{-1}(t) \cbr{B^\top(t) v(t) + W^\top(t) d(t)} - S(t) B(t) R^{-1}(t) \cbr{B^\top(t) v(t) + W^\top(t) d(t)} \\
&\quad + S(t) C(t) d(t) + A^\top(t) v(t) - \cbr{H(t) + B^\top(t) S(t)}^\top R^{-1}(t) B^\top(t) v(t)\Big], \\
\mathcal{C}(t, x^*; d) & = \frac12\cbr{W^\top(t) d(t) + B^\top(t) v(t)}^\top R^{-1}(t) \cbr{W^\top(t) d(t) + B^\top(t) v(t)} + d^\top(t) C^\top(t) v(t) \\
& \quad - \cbr{W^\top(t) d(t) + B^\top(t) v(t)}^\top R^{-1}(t) W^\top(t) d(t)   - \cbr{W^\top(t) d(t) + B^\top(t) v(t)}^\top R^{-1}(t) B^\top(t) v(t).
\end{align*}
By \eqref{eqn:quadratic-ansatz}, we compute the time derivative of the value function as
\begin{equation*}
\frac{\partial J^*(t,x)}{\partial t} = \frac12x^\top \dot{S}(t)x + x^\top \dot{v}(t) + \frac12\dot{z}(t).
\end{equation*}
Now, we substitute the above expression into the HJB equation and obtain
\begin{equation*}
\frac{\partial J^*(t, x^*)}{\partial t} + \mathcal{Q}(t, x^*; d) + \mathcal{L}(t, x^*; d) + \mathcal{C}(t, x^*; d) = 0.
\end{equation*}
By matching terms according to their degree in $x^*$, a sufficient condition for this identity to hold for all $x^*$ is that each group of terms is $0$ independently. This yields a system of differential equations for $S$, $v$, and $z$. We now present the evolution equations for $S$ and $v$, which are responsible for determining the feedback control and closed-loop dynamics. The equation for $z$ is omitted $(0.5\dot{z}(t) + \mathcal{C}(t, x^*; d) = 0)$, as it does not influence the state $x^*$, the adjoint $\lambda^*$, or the optimal control $u^*$. In particular, by requiring $0.5(x^*)^\top \dot{S}(t) x^* + \mathcal{Q}(t, x^*; d) = 0$, we obtain
\begin{multline*}
(x^*)^\top \Big(\frac12\dot{S} +\frac{1}{2} Q - H^\top R^{-1} H - H^\top R^{-1} B^\top S + \frac12H^\top R^{-1} H + \frac12H^\top R^{-1} B^\top S \\
+ \frac12S B R^{-1} H + \frac12S B R^{-1} B^\top S + \frac12A^\top S + \frac12S A - S B R^{-1} H - S B R^{-1} B^\top S \Big) x^* = 0,
\end{multline*}
where we have used the symmetry of \( S \) and the identity \( x^\top S A x = x^\top A^\top S x \). Simplifying this expression and canceling the terms, applying \eqref{subeqn:minimization-pde1}, we have $S(T) = Q_T$ and
\begin{equation}\label{nequ:3}
\dot{S} 
+ Q - H^\top R^{-1} H - H^\top R^{-1} B^\top S - S B R^{-1} H
- S B R^{-1} B^\top S + A^\top S + S A = 0,\quad\quad t\in[0, T).
\end{equation}
Similarly, for the linear terms in $x^*$, by requiring $(x^*)^\top \dot{v}(t) + \mathcal{L}(t, x^*; d) = 0$, we obtain
\begin{multline*}
(x^*)^\top \Big(\dot{v}(t) + \cbr{A(t) - B(t) R^{-1}(t) H(t) - B(t) R^{-1}(t) B^\top(t) S(t)}^\top v(t) \\
+ \cbr{G(t) + C^\top(t) S(t) - W(t) R^{-1}(t) (B^\top(t) S(t) + H(t))}^\top d(t) \Big) = 0.
\end{multline*}
Therefore, we apply \eqref{subeqn:minimization-pde1}, let $v(T) = G_T^\top d_T$, and have
\begin{multline}\label{nequ:2}
\dot{v}(t) + \cbr{A(t) - B(t) R^{-1}(t) H(t) - B(t) R^{-1}(t) B^\top(t) S(t)}^\top v(t) \\
+ \cbr{G(t) + C^\top(t) S(t) - W(t) R^{-1}(t) (B^\top(t) S(t) + H(t))}^\top d(t) = 0,\quad\quad t\in[0, T).
\end{multline}
Combining \eqref{nequ:4}, \eqref{eqn:explicitoptimalcontrol}, \eqref{nequ:3}, and \eqref{nequ:2} together, we complete the proof.

\section{Proof of \Cref{proposition:mainprop}}  \label{appendix:proof-main-proposition} 

The existence and uniqueness of the solution $S$ is known from Lemma \ref{lemma:riccati-positive-definite}. As a preparation result, we review a controllability result for a closely related system, whose properties will become relevant in establishing the lower bound of the matrix Riccati solution.

\begin{lemma} \label{lemma:dual-controllability}
Let $A, B, H, R$ be the matrices satisfying \Cref{assumption:matrices-are-bounded}--\ref{assumption:matrices-are-bounded-below}, then the pair $(-(A - BR^{-1}H)^{\top}, I)$ satisfies UCC on $[0, T - \sigma]$ in the sense of \Cref{assumption:uniform-controllability}. In particular, there exist constants $\tilde{\alpha}_0(\sigma)$, $\tilde{\alpha}_1(\sigma)$, $\tilde{\beta}_0(\sigma)$, $\tilde{\beta}_1(\sigma)>0$, depending on $\sigma$ from \Cref{assumption:uniform-controllability} but independent of $T$, such that for every $t_0\in [0,T-\sigma]$, the choice $t_1=t_0+\sigma\in [0,T]$ ensures that the following bounds hold:
\begin{align*}
\tilde{\alpha}_0(\sigma) I \preceq W_{-(A+BR^{-1}H)^{\top}, I}&(t_0, t_1) \preceq \tilde{\alpha}_1(\sigma) I, \\
\tilde{\beta}_0(\sigma) I \preceq \Phi_{-(A+BR^{-1}H)^{\top}}(t_1, t_0) W_{-(A+BR^{-1}H)^{\top}, I}&(t_0, t_1) \Phi_{-(A+BR^{-1}H)^{\top}}^\top(t_1, t_0) \preceq \tilde{\beta}_1(\sigma) I.
\end{align*}	
(The expressions of $\tilde{\alpha}_0(\sigma), \tilde{\alpha}_1(\sigma), \tilde{\beta}_0(\sigma), \tilde{\beta}_1(\sigma) > 0$ are provided in \eqref{eqn:upper-lower-bounds-adjoint-gramian} and \eqref{eqn:lower-upper-bound-gramian-phi}.)
\end{lemma}

\begin{proof} 
By the definition of controllability Gramian in \eqref{eqn:definegrammain}, we have
\begin{align*}
W_{-(A+BR^{-1}H)^{\top}, I}(t_0,t_1) &= \int_{t_0}^{t_1}\Phi_{-(A+BR^{-1}H)^{\top}}(t_0,s)\Phi^{\top}_{-(A+BR^{-1}H)^{\top}}(t_0,s)ds \\
& = \int_{t_0}^{t_1}\Phi^{\top}_{A+BR^{-1}H}(s,t_0)\Phi_{A+BR^{-1}H}(s,t_0)ds.
\end{align*}
Let $0\neq v\in \rr^{n_x}$ be any vector, we first consider the upper bound and have
\begin{align*} 
v^{\top}W_{-(A+BR^{-1}H)^{\top},I}(t_0,t_1)v & = 
\int_{t_0}^{t_1}v^{\top}\Phi_{A+BR^{-1}H}^{\top}(s,t_0)\Phi_{A+BR^{-1}H}(s,t_0)vds = \int_{t_0}^{t_1}\|\Phi_{A+BR^{-1}H}(s,t_0)v\|^2ds \\
&= \int_{t_0}^{t_0+\sigma}\|\Phi_{A+BR^{-1}H}(s,t_0)v\|^2ds \le
\int_{t_0}^{t_0+\sigma}e^{2\left(\lambda_A+\lambda_B\lambda_H/\gamma_R\right)(s-t_0)}\|v\|^2ds \\
& = \frac{\|v\|^2}{2(\lambda_A + \lambda_B \lambda_H / \gamma_R)} \left( e^{2(\lambda_A + \lambda_B \lambda_H / \gamma_R)\sigma} - 1 \right),
\end{align*}
where the inequality is due to \eqref{nequ:7}. For the lower bound, for any $0\neq v\in \rr^{n_x}$, we have by the bound derived in \eqref{eqn:lower-bound-of-evolution-operator} that 
\begin{align*}
v^{\top}W_{-(A+BR^{-1}H)^{\top},I}& (t_0,t_1)v = \int_{t_0}^{t_1}\|\Phi_{A+BR^{-1}H}(s,t_0)v\|^2ds \ge \int_{t_0}^{t_1}e^{-2\left(\lambda_A+\lambda_B\lambda_H/\gamma_R\right)(s-t_0)}\|v\|^2ds \\
& = \int_{t_0}^{t_0+\sigma}e^{-2\left(\lambda_A+\lambda_B\lambda_H/\gamma_R\right)(s-t_0)}\|v\|^2ds  =\frac{\|v\|^2}{2(\lambda_A + \lambda_B \lambda_H / \gamma_R)} \left( 1 - e^{-2(\lambda_A + \lambda_B \lambda_H / \gamma_R)\sigma} \right).
\end{align*} 
Combining the above two displays and defining
\begin{subequations} \label{eqn:upper-lower-bounds-adjoint-gramian}
\begin{align}
\tilde{\alpha}_0(\sigma) &:= \frac{1}{2(\lambda_A + \lambda_B \lambda_H / \gamma_R)} \left( 1 - e^{-2(\lambda_A + \lambda_B \lambda_H / \gamma_R)\sigma}\right),\\
\tilde{\alpha}_1(\sigma) &:= \frac{1}{2(\lambda_A + \lambda_B \lambda_H / \gamma_R)} \left( e^{2(\lambda_A + \lambda_B \lambda_H / \gamma_R)\sigma} - 1 \right),
\end{align} 
\end{subequations}		
we prove the first part of the statement. Furthermore, we note from \eqref{nequ:7} and \eqref{eqn:lower-bound-of-evolution-operator} that for any $0\neq v \in\rr^{n_x}$,
\begin{align*}
v^{\top}\Phi_{-(A+BR^{-1}H)^{\top}}(t_1, t_0) & W_{-(A+BR^{-1}H)^{\top}, I}(t_0, t_1) \Phi_{-(A+BR^{-1}H)^{\top}}^\top(t_1, t_0)v \\
& \le \tilde{\alpha}_1(\sigma)\|\Phi_{-(A+BR^{-1}H)^{\top}}^\top(t_1, t_0)v\|^2 \le \tilde{\alpha}_1(\sigma)e^{2\left(\lambda_A+\lambda_B\lambda_H/\gamma_R\right)\sigma}\|v\|^2,
\end{align*}
and
\begin{align*}
v^{\top}\Phi_{-(A+BR^{-1}H)^{\top}}(t_1, t_0) &W_{-(A+BR^{-1}H)^{\top}, I}(t_0, t_1) \Phi_{-(A+BR^{-1}H)^{\top}}^\top(t_1, t_0)v \\
&\ge\tilde{\alpha}_0(\sigma)\|\Phi_{-(A+BR^{-1}H)^{\top}}^\top(t_1, t_0)v\|^2 \ge \tilde{\alpha}_0(\sigma)e^{-2\left(\lambda_A+\lambda_B\lambda_H/\gamma_R\right)\sigma}\|v\|^2.
\end{align*}
Thus, we define
\begin{equation}\label{eqn:lower-upper-bound-gramian-phi}
\tilde{\beta}_0(\sigma) \coloneqq \tilde{\alpha}_0(\sigma)e^{-2\left(\lambda_A+\lambda_B\lambda_H/\gamma_R\right)\sigma} \quad\quad \text{ and }\quad\quad
\tilde{\beta}_1(\sigma) \coloneqq \tilde{\alpha}_1(\sigma)e^{2\left(\lambda_A+\lambda_B\lambda_H/\gamma_R\right)\sigma},
\end{equation}	
and conclude the proof.
\end{proof}

\begin{remark}
We note that the statement of \Cref{lemma:dual-controllability} is slightly weaker than those of \Cref{assumption:uniform-controllability} and \Cref{lemma:ucc-carries-over} in that at any time $t_0\in[0, T-\sigma]$, any state $v\in\mathbb{R}^{n_x}$ is steerable to $0$ precisely at time $t_0+\sigma$ for the system $(-(A+BF)^{\top},I)$ rather than some other time $t_1<t_0+\sigma$.
However, since we only require the existence of a time in $[t_0, t_0+\sigma]$ for any $t_0\in [0,T-\sigma]$, the choice $t_1=t_0+\sigma$ remains valid and does~not~impact the proof of \Cref{proposition:mainprop}.
\end{remark}

We now present the main proof of \Cref{proposition:mainprop}. In what follows, we first establish the desired bounds on the controllable segment $[0, T - \sigma]$. 

\vskip4pt

\noindent $\bullet$ \textbf{Case 1:} $t_0\in[0, T-\sigma]$. By Lemmas \ref{lemma:controllabilitygramian} and \ref{lemma:ucc-carries-over}, we know for any initial time $t_0 \in [0, T - \sigma]$ and any initial state $0\neq x_{t_0} \in \mathbb{R}^{n_x}$, there exists a control input that steers the system $(A-BR^{-1}H, B)$ to the origin by some final time $t_1 \in [t_0,t_0+\sigma]\subseteq[0,T]$. An explicit control that accomplishes this is given by the following
\begin{equation} \label{eqn:explicitsteeringcontrol}
u(t;x_{t_0}) = \begin{cases} 
-B^\top(t)\Phi_{A - BR^{-1}H}^\top(t_0, t) \, W_{A - BR^{-1}H, B}^{-1}(t_0, t_1)x_{t_0}, & t \in [t_0, t_1], \\
0, & t\in (t_1,T].
\end{cases}
\end{equation}
Let us denote $x(t; x_{t_0})$, $t\in[t_0, T]$, to be state trajectory of the system $(A-BR^{-1}H, B)$ with the above control input \eqref{eqn:explicitsteeringcontrol}; and denote $x^*: [t_0,T]\rightarrow \rr^{n_x}$ and $u^*: [t_0,T]\rightarrow \rr^{n_u}$ to be the optimal state and \mbox{control} for the truncated, shifted linear-quadratic problem \eqref{eqn:riccati-equivalent-problem} on $[t_0, T]$ with initial state $x_{t_0}$. Then, by the definition of cost-to-go (an analogy of \eqref{eqn:optimal-cost-to-go} with different system, and an analogy of \eqref{eqn:quadratic-ansatz} with $v(t)=0$, $z(t)=0$), we have
\begin{equation*}
\frac{1}{2}x_{t_0}^\top S(t_0) x_{t_0} =  \frac12\int_{t_0}^{T}
\left(x^{*\top}(t) \big(Q(t) - H^\top(t) R^{-1}(t) H(t)\big) x^*(t) + u^{*\top}(t) R(t) u^*(t)\right) dt + \frac12x^{*\top}(T) Q_T x^*(T).
\end{equation*}
Thus, by the optimality of $(x^*, u^*)$ on $[t_0, T]$, we have
\begin{align*}
\frac{1}{2}x_{t_0}^\top S(t_0) x_{t_0} & \leq \frac12\int_{t_0}^T 
\begin{bmatrix}
x(t;x_{t_0}) \\ u(t;x_{t_0})
\end{bmatrix}^\top
\begin{bmatrix}
Q(t) - H^\top(t) R^{-1}(t) H(t) & 0 \\
0 & R(t)
\end{bmatrix}
\begin{bmatrix}
x(t;x_{t_0}) \\ u(t;x_{t_0})
\end{bmatrix} dt 
+ \frac12x^\top(T;x_{t_0}) Q_T x(T;x_{t_0})\\
& = \frac12\int_{t_0}^{t_1} x^\top(t;x_{t_0}) \big(Q(t) - H^\top(t) R^{-1}(t) H(t)\big) x(t;x_{t_0})
+ u^\top(t;x_{t_0}) R(t) u(t;x_{t_0})dt,
\end{align*}
where we use the fact that $u(t, x_{t_0})=0$ and $x(t; x_{t_0}) = 0$ for all $t > t_1$. For the second term on the right hand side, we have for $t\in[t_0, t_1]$,
\begin{align}\label{nequ:10}
& u^\top(t;x_{t_0}) R(t) u(t;x_{t_0})  \leq \lambda_R\|u(t;x_{t_0})\|^2 \stackrel{\eqref{eqn:explicitsteeringcontrol}}{\leq}\lambda_R \|B(t)\|^2  \|\Phi_{A - BR^{-1}H}(t_0, t)\|^2  \|W_{A - BR^{-1}H, B}^{-1}(t_0, t_1)\|^2  \|x_{t_0}\|^2 \nonumber \\
&\quad\quad\quad \leq \frac{\lambda_R\lambda_B^2}{(\alpha_0'(\sigma))^2}\|x_{t_0}\|^2\|\Phi_{(A - BR^{-1}H)^\top}(t, t_0)\|^2 \stackrel{\eqref{nequ:7}}{\leq}\frac{\lambda_R\lambda_B^2}{(\alpha_0'(\sigma))^2}e^{2\left(\lambda_A+\lambda_B\lambda_H/\gamma_R\right)(t-t_0)}\|x_{t_0}\|^2,
\end{align}
where the third inequality is due to \Cref{lemma:ucc-carries-over}. For the first term on the right hand side, we substitute the control \eqref{eqn:explicitsteeringcontrol} into \eqref{eqn:mild-solution}, and have for $t\in[t_0, t_1]$,
\begin{align}\label{nequ:11}
\|x(t; x_{t_0})\| & = \nbr{\Phi_{A - BR^{-1}H}(t, t_0)\left( I - W_{A - BR^{-1}H, B}(t_0, t) W_{A - BR^{-1}H, B}^{-1}(t_0, t_1) \right)x_{t_0}} \nonumber \\
& \stackrel{\mathclap{\eqref{eqn:gramian-can-only-grow}}}{\leq} \rbr{1 + \frac{\alpha'_1(\sigma)}{\alpha'_0(\sigma)}}\|x_{t_0}\|\cdot \|\Phi_{A - BR^{-1}H}(t, t_0)\| \stackrel{\eqref{nequ:7}}{\leq} \rbr{1 + \frac{\alpha'_1(\sigma)}{\alpha'_0(\sigma)}} e^{\left(\lambda_A+\lambda_B\lambda_H/\gamma_R\right)(t-t_0)}\|x_{t_0}\|,
\end{align}
where the second inequality is also due to \Cref{lemma:ucc-carries-over}. Combining the above three displays together, and noting that $\|Q(t) - H^\top(t) R^{-1}(t) H(t)\|\leq \lambda_Q + \lambda_H^2/\gamma_R$, we obtain
\begin{align}\label{nequ:13}
x_{t_0}^\top S(t_0) x_{t_0} & \leq \cbr{\left(\lambda_Q + \frac{\lambda_H^2}{\gamma_R} \right)\rbr{1 + \frac{\alpha'_1(\sigma)}{\alpha'_0(\sigma)}}^2 + \frac{\lambda_R\lambda_B^2}{(\alpha_0'(\sigma))^2}}\|x_{t_0}\|^2 \int_{t_0}^{t_1}  e^{2\left(\lambda_A+\lambda_B\lambda_H/\gamma_R\right)(t-t_0)} dt \nonumber\\
& \leq \cbr{\left(\lambda_Q + \frac{\lambda_H^2}{\gamma_R} \right)\rbr{1 + \frac{\alpha'_1(\sigma)}{\alpha'_0(\sigma)}}^2 + \frac{\lambda_R\lambda_B^2}{(\alpha_0'(\sigma))^2}}\|x_{t_0}\|^2\cdot \frac{e^{2\left(\lambda_A+\lambda_B\lambda_H/\gamma_R\right)\sigma} - 1}{2(\lambda_A+\lambda_B\lambda_H/\gamma_R)}.
\end{align}
To establish the lower bound, we know from \Cref{lemma:riccati-positive-definite} that the Riccati matrix $S$ remains positive definite for all $t\in [0,T]$. Thus, let us define $U(t) \coloneqq S^{-1}(t)$ and have $\dot{U}(t) = -S^{-1}(t) \dot{S}(t) S^{-1}(t)$. Substituting this relation into the Riccati equation \eqref{eqn:riccati-equation}, we have
\begin{subequations} \label{eqn:inversericcatiequation}
\begin{align} 
\dot{U}(t) &= \big(A(t) - B(t) R^{-1}(t) H(t)\big) U(t) + U(t) \big(A(t) - B(t) R^{-1}(t) H(t)\big)^\top \notag \\
&\quad + U(t)\big(Q(t) - H^\top(t) R^{-1}(t) H(t)\big) U(t) - B(t) R^{-1}(t) B^\top(t), \quad t\in [0,T),\\
U(T) &= Q_T^{-1}.
\end{align}
\end{subequations}
Note that \eqref{eqn:inversericcatiequation} itself is a matrix Riccati equation. The existence and uniqueness follow from that of $S$. This Riccati equation corresponds to a dual linear-quadratic optimal control problem on the interval $[0, T]$, formulated as \cite[(3.12-14)]{Kirk2004Optimal}
\begin{subequations}\label{nequ:12}
\begin{align}
\min_{y(\cdot),\, v(\cdot)} \;\; & \frac{1}{2}\int_{0}^T \begin{bmatrix}
y(t)\\
v(t)
\end{bmatrix}^\top\begin{bmatrix}
B(t) R^{-1}(t) B^\top(t)  & 0\\
0 & (Q(t) - H^\top(t) R^{-1}(t) H(t))^{-1}
\end{bmatrix}\begin{bmatrix}
y(t)\\
v(t)
\end{bmatrix} dt + \frac{1}{2}y^\top(T) Q_T^{-1} y(T), \nonumber \\
\text{s.t. } \;\;\;\;  & \dot{y}(t) = -\big(A(t) - B(t) R^{-1}(t) H(t)\big)^\top y(t) + v(t), \quad t\in (0,T],\\
& y(t_0) = y_0\in\mathbb{R}^{n_x}.
\end{align}
\end{subequations}
The state dynamics are now given by the pair $(-(A - BR^{-1}H)^{\top}, I)$, and the analysis of this system mirrors that of the shifted OCP in \Cref{theorem:shifted-ocp}. In particular, for any $t_0\in [0,T-\sigma]$, we apply \Cref{lemma:dual-controllability} to construct an explicit zero-steering control over the horizon $[t_0, t_0+\sigma]$ that drives any initial state $0\neq y_{t_0}\in \rr^{n_x}$ to the origin within time $\sigma$ from $t_0$, given by
\begin{equation*}
v(t; y_{t_0}) \coloneqq \begin{cases}
-\Phi_{-(A - BR^{-1}H)^{\top}}^{\top}(t_0, t) \, W_{-(A - BR^{-1}H)^{\top}, I}^{-1}(t_0, t_0+\sigma) \, y_{t_0}, & t \in[t_0, t_0+\sigma], \\
0, & t\in (t_0+\sigma, T].
\end{cases}
\end{equation*} 
The corresponding state with the above control trajectory is given by \eqref{eqn:mild-solution}: for $t\in [t_0,t_0+\sigma]$,
\begin{equation*}
y(t; y_{t_0}) = \Phi_{-(A - BR^{-1}H)^{\top}}(t, t_0) \left(I - W_{-(A - BR^{-1}H)^{\top}, I}(t_0, t) \, W_{-(A - BR^{-1}H)^{\top}, I}^{-1}(t_0, t_0+\sigma)\right) y_{t_0},
\end{equation*}
and for $t\in(t_0+\sigma, T]$, $y(t; y_{t_0}) = 0$. With the above specific state-control trajectories, we have
\begin{equation*}
\frac{1}{2} y_{t_0}^{\top} U(t_0) y_{t_0} \leq
\frac{1}{2}\int_{t_0}^{t_0+\sigma}
\begin{bmatrix}
y(t;y_{t_0})\\
v(t; y_{t_0})
\end{bmatrix}^{\top}
\begin{bmatrix}
B(t)R^{-1}(t)B^{\top}(t) & 0 \\
0 & \left(Q(t) - H^{\top}(t) R^{-1}(t) H(t)\right)^{-1}
\end{bmatrix}\begin{bmatrix}
y(t;y_{t_0})\\
v(t; y_{t_0})
\end{bmatrix} dt.
\end{equation*}
Following the same derivation as in \eqref{nequ:10}, we have from Assumption \ref{assumption:matrices-are-bounded-below} and \Cref{lemma:dual-controllability} that
\begin{equation*}
v^\top(t;y_{t_0}) (Q(t) - H^{\top}(t) R^{-1}(t) H(t))^{-1} v(t;y_{t_0}) \leq \frac{1}{\gamma_Q(\tilde{\alpha}_0(\sigma))^2}e^{2\left(\lambda_A+\lambda_B\lambda_H/\gamma_R\right)(t-t_0)}\|y_{t_0}\|^2.
\end{equation*}
Following the same derivation as in \eqref{nequ:11}, we have
\begin{equation*}
\|y(t;y_{t_0})\| \leq \left(1+\frac{\tilde{\alpha}_1(\sigma)}{\tilde{\alpha}_0(\sigma)}
\right)\|y_{t_0}\|e^{\left(\lambda_A+\lambda_B\lambda_H/\gamma_R\right)(t-t_0)}.
\end{equation*}
Combining the above three displays, we obtain
\begin{align}\label{nequ:14}
y_{t_0}^{\top} S^{-1}(t_0) y_{t_0}  = y_{t_0}^{\top} U(t_0) y_{t_0} & \leq \cbr{\frac{\lambda_B^2}{\gamma_R}\left(1+\frac{\tilde{\alpha}_1(\sigma)}{\tilde{\alpha}_0(\sigma)}\right)^2 + \frac{1}{\gamma_Q(\tilde{\alpha}_0(\sigma))^2}}\|y_{t_0}\|^2\int_{t_0}^{t_0+\sigma}e^{2\left(\lambda_A+\lambda_B\lambda_H/\gamma_R\right)(t-t_0)} dt \nonumber\\ 
& = \cbr{\frac{\lambda_B^2}{\gamma_R}\left(1+\frac{\tilde{\alpha}_1(\sigma)}{\tilde{\alpha}_0(\sigma)}\right)^2 + \frac{1}{\gamma_Q(\tilde{\alpha}_0(\sigma))^2}}\|y_{t_0}\|^2\cdot\frac{e^{2\left(\lambda_A+\lambda_B\lambda_H/\gamma_R\right)\sigma} - 1}{2(\lambda_A+\lambda_B\lambda_H/\gamma_R)}.
\end{align}
\noindent$\bullet$ \textbf{Case 2:} $t_0\in(T-\sigma, T]$. We now discuss the case without uniform controllability. The bounds will similarly be derived from the relationship between the Riccati matrix and optimal cost, and without the ability to steer an arbitrary initial state at $t_0$ to $0$ due to the lack of complete controllability, we simply apply the zero control trajectory $\tilde{u}(t) = 0$, $t\in[t_0, T]$. With the zero control, the state trajectory evolves according to the homogeneous dynamics, i.e., $\tilde{x}(t) = \Phi_{A - B R^{-1} H}(t, t_0) x_{t_0}$, $t\in [t_0,T]$. With the above state-control trajectories~$(\tilde{x}, \tilde{u})$, we have
\begin{align}\label{nequ:15}
x_{t_0}^{\top} S(t_0) x_{t_0} & \leq \int_{t_0}^{T}
\begin{bmatrix}
\tilde{x}(t) \\ \tilde{u}(t)
\end{bmatrix}^{\top}\begin{bmatrix}
Q(t) - H^{\top}(t) R^{-1}(t) H(t) & 0 \\
0 & R(t)
\end{bmatrix}\begin{bmatrix}
\tilde{x}(t) \\ \tilde{u}(t)
\end{bmatrix} dt + \tilde{x}(T)^{\top} Q_T\tilde{x}(T) \nonumber\\
& \le \left( \lambda_Q + \frac{\lambda_H^2}{\gamma_R} \right) \int_{t_0}^{T} \|\tilde{x}(t)\|^2\, dt 	+ \lambda_Q\|\tilde{x}(T)\|^2 \nonumber\\
& \stackrel{\mathclap{\eqref{nequ:7}}}{\leq} \left( \lambda_Q + \frac{\lambda_H^2}{\gamma_R} \right)\|x_{t_0}\|^2\int_{t_0}^Te^{2\left(\lambda_A +{\lambda_B \lambda_H}/{\gamma_R}\right)(t-t_0)} dt + \lambda_Q\|x_{t_0}\|^2e^{2\left(\lambda_A +{\lambda_B \lambda_H}/{\gamma_R}\right)(T-t_0)}  \nonumber\\
& \leq \left( \lambda_Q + \frac{\lambda_H^2}{\gamma_R} \right)\|x_{t_0}\|^2 \frac{e^{2\left(\lambda_A+\lambda_B\lambda_H/\gamma_R\right)\sigma} - 1}{2(\lambda_A+\lambda_B\lambda_H/\gamma_R)} + \lambda_Q\|x_{t_0}\|^2e^{2\left(\lambda_A +{\lambda_B \lambda_H}/{\gamma_R}\right)\sigma} \nonumber\\
& \leq \left( \lambda_Q + \frac{\lambda_H^2}{\gamma_R} \right)\rbr{1+\frac{1}{2(\lambda_A+\lambda_B\lambda_H/\gamma_R)}} e^{2\left(\lambda_A +{\lambda_B \lambda_H}/{\gamma_R}\right)\sigma}\cdot\|x_{t_0}\|^2.
\end{align}
To obtain the lower bound, we consider again the inverse \mbox{Riccati} \mbox{equation} \eqref{eqn:inversericcatiequation} and \mbox{corresponding} \mbox{optimal} control problem \eqref{nequ:12}. In particular, for any $t_0\in (T-\sigma,T]$ and any $y_{t_0}\in\rr^{n_x}$, we apply the zero control $\tilde{v}(t) = 0$, $t\in [t_0,T]$, and the state trajectory becomes $\tilde{y}(t) = \Phi_{-(A - B R^{-1} H)^{\top}}(t, t_0) y_{t_0}$. Therefore, the bound of $U(t) = S^{-1}(t)$ on $(T-\sigma, T]$ becomes 
\begin{align}\label{nequ:16}
& y_{t_0}^{\top} S^{-1}(t_0) y_{t_0}  = y_{t_0}^{\top} U(t_0) y_{t_0} \nonumber\\
& \leq \int_{t_0}^T \begin{bmatrix}
\tilde{y}(t)\\
\tilde{v}(t)
\end{bmatrix}^\top\begin{bmatrix}
B(t) R^{-1}(t) B^\top(t)  & 0\\
0 & (Q(t) - H^\top(t) R^{-1}(t) H(t))^{-1}
\end{bmatrix}\begin{bmatrix}
\tilde{y}(t)\\
\tilde{v}(t)
\end{bmatrix}dt + \tilde{y}^\top(T) Q_T^{-1} \tilde{y}(T) \nonumber\\
& \leq \frac{\lambda_B^2}{\gamma_R}\int_{t_0}^{T}\|\tilde{y}(t)\|^2 dt + \frac{1}{\gamma_Q}\|\tilde{y}(T)\|^2 \stackrel{\mathclap{\eqref{nequ:7}}}{\leq} \frac{\lambda_B^2}{\gamma_R}\|y_{t_0}\|^2\int_{t_0}^{T}e^{2 \left( \lambda_A + {\lambda_B\lambda_H}/{\gamma_R} \right)(t-t_0)} dt + \frac{\|y_{t_0}\|^2}{\gamma_Q}e^{2 \left( \lambda_A + {\lambda_B\lambda_H}/{\gamma_R} \right)(T-t_0)} \nonumber\\
& \leq \cbr{\frac{\lambda_B^2}{2\gamma_R\left( \lambda_A + {\lambda_B\lambda_H}/{\gamma_R}\right)} + \frac{1}{\gamma_Q}}e^{2 \left( \lambda_A + {\lambda_B\lambda_H}/{\gamma_R} \right)\sigma}\cdot\|y_{t_0}\|^2.
\end{align}
Finally, we combine \eqref{nequ:13}, \eqref{nequ:14}, \eqref{nequ:15}, and \eqref{nequ:16}, and define
\begin{subequations}  \label{eqn:final-riccati-bounds}
\begin{align}
c_0(\sigma) \coloneqq \min\Bigg\{ & 2\rbr{\lambda_A+\frac{\lambda_B\lambda_H}{\gamma_R}}\rbr{\frac{\lambda_B^2}{\gamma_R}\left(1+\frac{\tilde{\alpha}_1(\sigma)}{\tilde{\alpha}_0(\sigma)}\right)^2 + \frac{1}{\gamma_Q(\tilde{\alpha}_0(\sigma))^2}}^{-1}, \nonumber \\
& \rbr{\frac{\lambda_B^2}{2\gamma_R\left( \lambda_A + {\lambda_B\lambda_H}/{\gamma_R}\right)} + \frac{1}{\gamma_Q}}^{-1} \Bigg\}\cdot \exp(-2\left( \lambda_A + {\lambda_B\lambda_H}/{\gamma_R} \right)\sigma);\\
c_1(\sigma) \coloneqq \max\Bigg\{ & \rbr{\left(\lambda_Q + \frac{\lambda_H^2}{\gamma_R} \right)\rbr{1 + \frac{\alpha'_1(\sigma)}{\alpha'_0(\sigma)}}^2 + \frac{\lambda_R\lambda_B^2}{(\alpha_0'(\sigma))^2}}\frac{1}{2(\lambda_A+\lambda_B\lambda_H/\gamma_R)}, \nonumber\\
& \left( \lambda_Q + \frac{\lambda_H^2}{\gamma_R} \right)\rbr{1+\frac{1}{2(\lambda_A+\lambda_B\lambda_H/\gamma_R)}} \Bigg\}\exp(2\left( \lambda_A + {\lambda_B\lambda_H}/{\gamma_R} \right)\sigma),
\end{align}
\end{subequations}
where $\alpha'_0(\sigma)$, $\alpha'_1(\sigma)$ are from \Cref{lemma:ucc-carries-over}, and $\tilde{\alpha}_0(\sigma)$, $\tilde{\alpha}_1(\sigma)$ are from \Cref{lemma:dual-controllability}. This completes the proof.


\bibliographystyle{informs2014} 
\bibliography{ref}

\begin{flushright}
\scriptsize \framebox{\parbox{\textwidth}{Government License: The submitted manuscript has been created by UChicago Argonne, LLC, Operator of Argonne National Laboratory (``Argonne"). Argonne, a U.S. Department of Energy Office of Science laboratory, is operated under Contract No. DE-AC02-06CH11357.  The U.S. Government retains for itself, and others acting on its behalf, a paid-up nonexclusive, irrevocable worldwide license in said article to reproduce, prepare derivative works, distribute copies to the public, and perform publicly and display publicly, by or on behalf of the Government. The Department of Energy will provide public access to these results of federally sponsored research in accordance with the DOE Public Access Plan. http://energy.gov/downloads/doe-public-access-plan. }}
\normalsize
\end{flushright}

\end{document}